\documentclass[a4paper,reqno]{amsart}

\usepackage{graphicx}
\usepackage{amsmath}
\usepackage{amssymb}
\input xy
\xyoption{all}

\renewcommand{\Im}{\operatorname{Im}\nolimits}
\renewcommand{\mod}{\operatorname{mod}\nolimits}
\renewcommand{\top}{\operatorname{top}\nolimits}
\newcommand{\SL}{\operatorname{SL}\nolimits}
\newcommand{\soc}{\operatorname{soc}\nolimits}
\newcommand{\Cok}{\operatorname{Cok}\nolimits}
\newcommand{\add}{\operatorname{add}\nolimits}
\newcommand{\simple}{\operatorname{sim}\nolimits}

\newcommand{\Hom}{\operatorname{Hom}\nolimits}
\newcommand{\RHom}{\mathbf{R}\strut\kern-.2em\operatorname{Hom}\nolimits}

\newcommand{\op}{\operatorname{op}\nolimits}

\newcommand{\pr}{\operatorname{pr}\nolimits}
\newcommand{\inj}{\operatorname{in}\nolimits}
\newcommand{\CM}{\operatorname{CM}\nolimits}
\newcommand{\id}{\operatorname{id}\nolimits}
\newcommand{\pd}{\operatorname{pd}\nolimits}

\newcommand{\Ker}{\operatorname{Ker}\nolimits}
\newcommand{\depth}{\operatorname{depth}\nolimits}
\newcommand{\gl}{\operatorname{gl.dim}\nolimits}
\newcommand{\dom}{\operatorname{dom.dim}\nolimits}
\newcommand{\Ext}{\operatorname{Ext}\nolimits}
\newcommand{\End}{\operatorname{End}\nolimits}

\newcommand{\Tr}{\operatorname{Tr}\nolimits}

\newcommand{\ul}{\underline}

\newtheorem{theorem}{Theorem}[section]

\newtheorem{corollary}[theorem]{Corollary}
\newtheorem{definition}[theorem]{Definition}
\newtheorem{lemma}[theorem]{Lemma}
\newtheorem{proposition}[theorem]{Proposition}
\newtheorem{example}[theorem]{Example}

\newcommand{\F}{{\mathbf{F}}}
\newcommand{\G}{{\mathbf{G}}}
\newcommand{\PPP}{{\mathbf{P}}}

\newcommand{\Z}{{\mathbf{Z}}}
\newcommand{\SSS}{{\mathbf{S}}}

\newcommand{\BB}{{\mathcal B}}
\newcommand{\DD}{{\mathcal D}}
\newcommand{\CC}{{\mathcal C}}
\newcommand{\KK}{{\mathcal K}}
\newcommand{\MM}{{\mathcal M}}
\newcommand{\NN}{{\mathcal N}}
\newcommand{\PP}{{\mathcal P}}
\newcommand{\II}{{\mathcal I}}
\newcommand{\UU}{{\mathcal U}}
\newcommand{\XX}{{\mathcal X}}

\newcommand{\GG}{{\mathcal G}}
\newcommand{\HH}{{\mathcal H}}

\begin{document}

\title{Cluster tilting for higher Auslander algebras}

\author{Osamu Iyama}
\address{Graduate School of Mathematics, Nagoya University,
Chikusa-ku, Nagoya, 464-8602, Japan}
\email{iyama@math.nagoya-u.ac.jp}

\begin{abstract}
The concept of cluster tilting gives a higher analogue of classical Auslander
correspondence between representation-finite algebras and Auslander algebras.
The $n$-Auslander-Reiten translation functor $\tau_n$ plays an important
role in the study of $n$-cluster tilting subcategories. We study
the category $\MM_n$ of preinjective-like modules obtained by applying
$\tau_n$ to injective modules repeatedly.
We call a finite dimensional algebra $\Lambda$ \emph{$n$-complete} if
$\MM_n=\add M$ for an $n$-cluster tilting object $M$.
Our main result asserts that the endomorphism algebra
$\End_\Lambda(M)$ is $(n+1)$-complete.
This gives an inductive construction of $n$-complete algebras.
For example, any representation-finite hereditary algebra
$\Lambda^{(1)}$ is $1$-complete. Hence the Auslander algebra $\Lambda^{(2)}$ of
$\Lambda^{(1)}$ is $2$-complete. Moreover, for any $n\ge1$, we have an
$n$-complete algebra $\Lambda^{(n)}$ which has an $n$-cluster tilting
object $M^{(n)}$ such that $\Lambda^{(n+1)}=\End_{\Lambda^{(n)}}(M^{(n)})$.
We give the presentation of $\Lambda^{(n)}$ by a quiver with relations.
We apply our results to construct $n$-cluster tilting subcategories of
derived categories of $n$-complete algebras.
\end{abstract}
\maketitle

\tableofcontents
The concept of cluster tilting \cite{BMRRT} is fundamental to
categorify Fomin-Zelevinsky cluster algebras \cite{FZ}, and
a fruitful theory has been developped in recent years (see survey papers
\cite{BuM,Re,Rin,K}).
It also played an important role from the viewpoint of
higher analogue of Auslander-Reiten theory \cite{I3,I4,I5} in the study of
rigid Cohen-Macaulay modules, Calabi-Yau algebras and categories,
and non-commutative crepant resolutions
\cite{BIRSc,BIRSm,BIKR,GLS1,GLS2,GLS3,IR,IY,KR1,KR2,KMV,KZ,V1,V2}.
There are a lot of recent work on higher cluster tilting
\cite{ABST,BaM,BT,EH,HZ1,HZ2,HZ3,JH1,JH2,L,T,W,Z,ZZ}.
In this paper we shall present a
systematic method to construct a series of finite dimensional algebras
$\Lambda$ with $n$-cluster tilting objects.

In the representation theory of a representation-finite
finite-dimensional algebra $\Lambda$ with an additive generator $M$ in
$\mod\Lambda$, the endomorphism algebra $\Gamma:=\End_\Lambda(M)$
called the Auslander algebra gives a prototype of the use of
functor categories in Auslander-Reiten theory.
The Auslander algebra $\Gamma$ keeps all information of the category
$\mod\Lambda$ in its algebraic structure, and it is a prominent result
due to Auslander \cite{A1,ARS} that Auslander algebras are
characterized by `regularity of dimension two'
\[\gl\Gamma\le 2\le\dom\Gamma.\]
Since almost split sequences in $\mod\Lambda$ correspond to minimal
projective resolutions of simple $\Gamma$-modules, the
Auslander-Reiten quiver of $\Lambda$ coincides with the quiver of
$\Gamma^{\op}$.
As a result the quiver of $\Gamma$ has the structure of translation
quivers. Moreover, the structure theory due to Riedtmann \cite{Rie},
Bongartz-Gabriel \cite{BG}, Igusa-Todorov \cite{IT1,IT2},
Bautista-Gabriel-Roiter-Salmeron \cite{BGRS},...
realizes Auslander algebras as factor algebras of path algebras
of translation quivers modulo mesh relations.
They can be regarded as an analogue of the commutative relation
$xy=yx$ in the formal power series ring $S_2:=k[[x,y]]$ of two
variables since the mesh category of the translation quiver
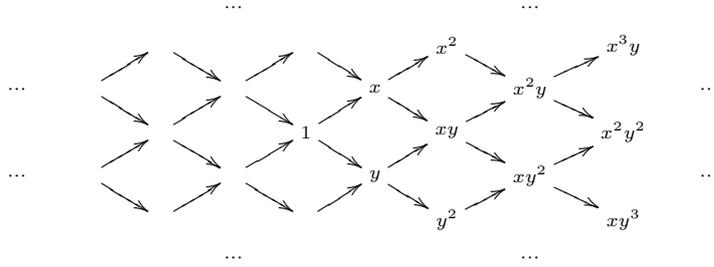
\[\xymatrix@C=0.5cm@R0.1cm{
&&&{\scriptstyle \cdots}&&&&{\scriptstyle \cdots}&&&&\\
&&\ \ \ar[dr]&&\ \ \ar[dr]&&{\scriptstyle x^2}\ar[dr]&&{\scriptstyle x^3y}&&\\
{\scriptstyle \cdots}&\ \ \ar[dr]\ar[ur]&&\ \ \ar[dr]\ar[ur]&&{\scriptstyle x}\ar[dr]\ar[ur]&&{\scriptstyle x^2y}\ar[dr]\ar[ur]&&{\scriptstyle \cdots}\\
&&\ \ \ar[dr]\ar[ur]&&{\scriptstyle 1}\ar[dr]\ar[ur]&&{\scriptstyle xy}\ar[dr]\ar[ur]&&{\scriptstyle x^2y^2}&&\\
{\scriptstyle \cdots}&\ \ \ar[dr]\ar[ur]&&\ \ \ar[dr]\ar[ur]&&{\scriptstyle y}\ar[dr]\ar[ur]&&{\scriptstyle xy^2}\ar[dr]\ar[ur]&&{\scriptstyle \cdots}\\
&&\ \ \ar[ur]&&\ \ \ar[ur]&&{\scriptstyle y^2}\ar[ur]&&{\scriptstyle xy^3}&&\\
&&&{\scriptstyle \cdots}&&&&{\scriptstyle \cdots}&&&&
}\]
gives a universal Galois covering of $S_2$ in Gabriel's sense \cite{G}.
This is a basic pattern of Auslander-Reiten quivers,
so it is suggestive in the representation theory to regard their
module categories as a certain analogy of $S_2$.
A typical example is given by the category $\CM(\Lambda)$ of
Cohen-Macaulay modules over a quotient singularity $\Lambda:=S_2^G$
corresponding to a finite subgroup $G$ of $\SL(2,k)$ \cite{A3,AR2,RV}.
In this case $\CM(\Lambda)$ has an additive generator $S_2$, and the
Auslander algebra $\Gamma:=\End_{\Lambda}(S_2)$ is isomorphic to
the skew group algebra $S_2*G$ which is regular in the sense that
$\gl\Gamma=2=\depth\Gamma$.
The Koszul complex
$0\to S\stackrel{{x\choose y}}{\longrightarrow}S^2\stackrel{(y,
  -x)}{\longrightarrow}S\to k\to0$
of $S$ induces almost split sequences in $\CM(\Lambda)$.
Hence the Auslander-Reiten quiver of $\Lambda$ is given by the McKay
quiver of $G$, and forms the translation quiver $\Z\Delta/\tau$
associated to an extended Dynkin diagram $\Delta$.

It is natural to consider a higher dimensional analogue of this
classical theory, and $n$-cluster tilting (=maximal
$(n-1)$-orthogonal) subcategories were introduced in \cite{I3,I4} in
this context.
The endomorphism algebra $\Gamma:=\End_\Lambda(M)$ of an $n$-cluster
tilting object $M$ in $\mod\Lambda$ is called an $n$-Auslander algebra,
and characterized by `regularity of dimension $n+1$'
\[\gl\Gamma\le n+1\le\dom\Gamma.\]
It is known that the category $\add M$ has $n$-almost split
sequences, which correspond to minimal projective resolutions of
simple $\Gamma$-modules. It is natural to regard $\Gamma$ as analogue
of the formal power series ring $S_{n+1}:=k[[x_1,\cdots,x_{n+1}]]$ of
$n+1$ variables.
Actually a typical example of $n$-cluster tilting objects is given by a
quotient singularity $\Lambda:=S_{n+1}^G$ corresponding to a finite
subgroup $G$ of $\SL(n+1,k)$ acting on $k^{n+1}\backslash\{0\}$
freely. In this case $\CM(\Lambda)$ has an
$n$-cluster tilting object $S_{n+1}$, and the $n$-Auslander algebra
$\Gamma:=\End_{\Lambda}(S_{n+1})$ is isomorphic to the skew group
algebra $S_{n+1}*G$. Again the Koszul complex of $S$ induces $n$-almost
split sequences in $\add S_{n+1}$.
So it is natural to hope in a certain generality that $n$-almost split
sequences in $\add M$ can be constructed as a certain analogue of Koszul
complexes of $S_{n+1}$, and that the basic pattern of quivers of
$n$-Auslander algebras is given by the Galois covering
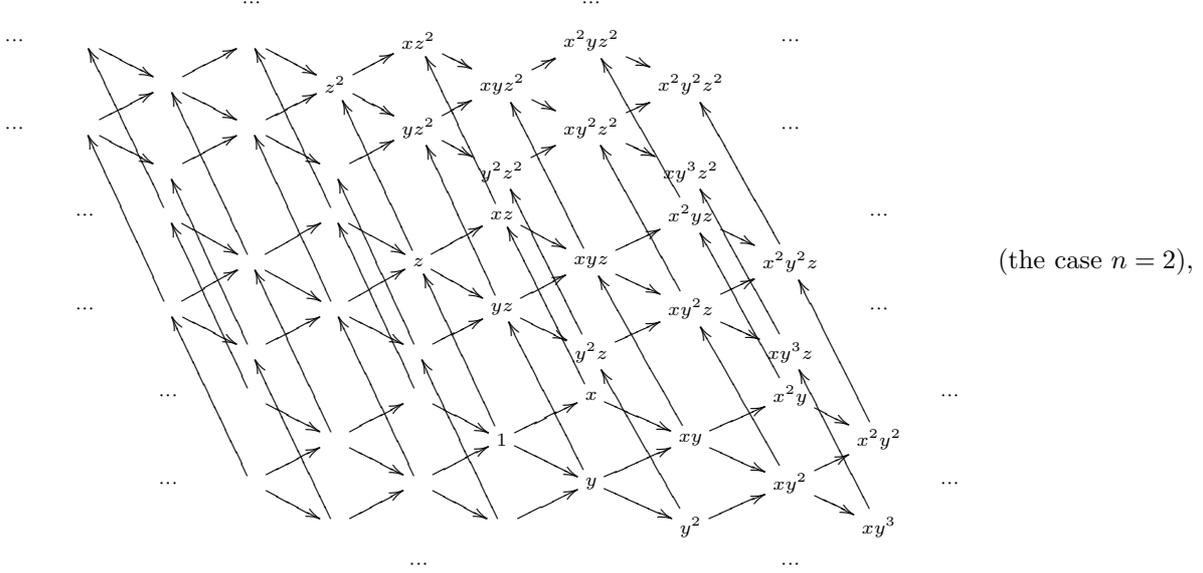
\[\xymatrix@C=0.3cm@R0.1cm{
&&&{\scriptstyle \cdots}&&&&{\scriptstyle \cdots}&&&&\\
{\scriptstyle \cdots}&\ \ \ \ \ \ar[dr]&&\ \ \ \ \ \ar[dr]&&{\scriptstyle xz^2}\ar[dr]&&{\scriptstyle x^2yz^2}\ar[dr]&&{\scriptstyle \cdots}\\
&&\ \ \ \ \ \ar[dr]\ar[ur]&&{\scriptstyle z^2}\ar[dr]\ar[ur]&&{\scriptstyle xyz^2}\ar[dr]\ar[ur]&&{\scriptstyle x^2y^2z^2}&&\\
{\scriptstyle \cdots}&\ \ \ \ \ \ar[dr]\ar[ur]&&\ \ \ \ \ \ar[dr]\ar[ur]&&{\scriptstyle yz^2}\ar[dr]\ar[ur]&&{\scriptstyle xy^2z^2}\ar[dr]\ar[ur]&&{\scriptstyle \cdots}\\
&&\ \ \ \ \ \ar[ur]&&\ \ \ \ \ \ar[ur]&&{\scriptstyle y^2z^2}\ar[ur]&&{\scriptstyle xy^3z^2}&&\\
&{\scriptstyle \cdots}&\ \ \ \ \ \ar[dr]\ar[uuuul]&&\ \ \ \ \ \ar[dr]\ar[uuuul]&&{\scriptstyle xz}\ar[dr]\ar[uuuul]&&{\scriptstyle x^2yz}\ar[dr]\ar[uuuul]&&{\scriptstyle \cdots}\\
&&&\ \ \ \ \ \ar[dr]\ar[ur]\ar[uuuul]&&{\scriptstyle z}\ar[uuuul]\ar[dr]\ar[ur]&&{\scriptstyle xyz}\ar[dr]\ar[ur]\ar[uuuul]&&{\scriptstyle x^2y^2z}\ar[uuuul]&&&\mbox{(the case $n=2$),}\\
&{\scriptstyle \cdots}&\ \ \ \ \ \ar[dr]\ar[ur]\ar[uuuul]&&\ \ \ \ \ \ar[dr]\ar[ur]\ar[uuuul]&&{\scriptstyle yz}\ar[dr]\ar[ur]\ar[uuuul]&&{\scriptstyle xy^2z}\ar[dr]\ar[ur]\ar[uuuul]&&{\scriptstyle \cdots}\\
&&&\ \ \ \ \ \ar[ur]\ar[uuuul]&&\ \ \ \ \ \ar[ur]\ar[uuuul]&&{\scriptstyle y^2z}\ar[ur]\ar[uuuul]&&{\scriptstyle xy^3z}\ar[uuuul]&&\\
&&{\scriptstyle \cdots}&\ \ \ \ \ \ar[dr]\ar[uuuul]&&\ \ \ \ \ \ar[dr]\ar[uuuul]&&{\scriptstyle x}\ar[dr]\ar[uuuul]&&{\scriptstyle x^2y}\ar[dr]\ar[uuuul]&&{\scriptstyle \cdots}\\
&&&&\ \ \ \ \ \ar[dr]\ar[ur]\ar[uuuul]&&{\scriptstyle 1}\ar[uuuul]\ar[dr]\ar[ur]&&{\scriptstyle xy}\ar[dr]\ar[ur]\ar[uuuul]&&{\scriptstyle x^2y^2}\ar[uuuul]&&\\
&&{\scriptstyle \cdots}&\ \ \ \ \ \ar[dr]\ar[ur]\ar[uuuul]&&\ \ \ \ \ \ar[dr]\ar[ur]\ar[uuuul]&&{\scriptstyle y}\ar[dr]\ar[ur]\ar[uuuul]&&{\scriptstyle xy^2}\ar[dr]\ar[ur]\ar[uuuul]&&{\scriptstyle \cdots}\\
&&&&\ \ \ \ \ \ar[ur]\ar[uuuul]&&\ \ \ \ \ \ar[ur]\ar[uuuul]&&{\scriptstyle y^2}\ar[ur]\ar[uuuul]&&{\scriptstyle xy^3}\ar[uuuul]&&\\
&&&&&{\scriptstyle \cdots}&&&&{\scriptstyle \cdots}&&&&
}\]
 of $S_{n+1}$, which has the set $\Z^{n+1}$ of vertices.

The aim of this paper is to give a class of finite dimensional
algebras with
$n$-cluster tilting objects satisfying the desired properties above.
Our construction is inductive in the following sense:
We introduce a class of algebras $\Lambda$ called $n$-complete
algebras, which are
algebras with $n$-cluster tilting objects $M$ satisfying certain nice
properties. Our main result asserts that the endomorphism algebra
$\Gamma:=\End_\Lambda(M)$ is $(n+1)$-complete, hence $\Gamma$ has 
an $(n+1)$-cluster tilting object $N$.
This procedure continues repeatedly, so $\End_\Gamma(N)$ is
$(n+2)$-complete and has an $(n+2)$-cluster tilting object, and so
on. 
We notice here that we consider not only $n$-cluster tilting objects
in whole module categories $\mod\Lambda$ but also those in
full subcategories
\[T^{\perp}:=\{X\in\mod\Lambda\ |\ \Ext^i_\Lambda(T,X)=0\ (0<i)\}\]
associated to tilting $\Lambda$-modules $T$.
Such a generalization is natural from the viewpoint of study of
Auslander-type conditions \cite{I4,HI},
and indispensable for our inductive construction to work.
It is interesting that our inductive construction reminds
us of a classical result due to Auslander-Reiten \cite{AR1}
which asserts that the category of coherent functors over
a dualizing variety again forms a dualizing variety.

In forthcoming papers \cite{IH,IO1,IO2,OT} $n$-complete algebras
will be studied further.

\medskip
\noindent{\bf Acknowledgement }
The author would like to thank organizers and participants of
``XII International Conference on Representations of Algebras''
(Torun, August 2007), where results from this paper were presented.
Part of this work was done while the author visited Boston in
April 2008. He would like to thank Kiyoshi Igusa and Gordana Todorov
and people in Northeastern University for their hospitality.
He also thanks Zhaoyong Huang, Xiaojin Zhang and an anonymous referee
for helpful comments.

\medskip
\noindent{\bf Conventions }
Throughout this paper, all \emph{subcategories} are assumed to be full
and closed under isomorphism, direct sums, and direct summands.
We denote by $J_{\CC}$ the Jacobson radical of an additive category
$\CC$ \cite{ARS,ASS}.

All modules are usually right modules, and the
composition $fg$ of morphisms means first $g$, then $f$.
We denote by $\mod\Lambda$ the category of finitely generated
$\Lambda$-modules, by $J_\Lambda$ the Jacobson radical of $\Lambda$.
For $M\in\mod\Lambda$, we denote by $\add M$ the subcategory of
$\mod\Lambda$ consisting of direct summands of finite direct sums of
copies of $M$. For example $\add\Lambda$ is the category $\pr\Lambda$
of finitely generated projective $\Lambda$-modules, and
$\add D\Lambda$ is the category $\inj\Lambda$
of finitely generated injective $\Lambda$-modules.

\section{Our results}

In this section, we shall present our results in this paper.
Let $\Lambda$ be a finite dimensional algebra.


\subsection{$n$-cluster tilting in module categories}\label{section: nCT1}

Let us recall a classical concept due to Auslander-Smalo \cite{ASm}.
A subcategory $\CC$ of an additive category $\XX$ is called
\emph{contravariantly finite} if for any $X\in\XX$, there exists a morphism
$f\in\Hom_{\XX}(C,X)$ with $C\in\CC$ such that
$\Hom_{\XX}(-,C)\stackrel{f}{\to}\Hom_{\XX}(-,X)\to0$
is exact on $\CC$. Dually a \emph{covariantly finite subcategory}
is defined. A contravariantly and covariantly finite subcategory
is called \emph{functorially finite}.

\begin{definition}\label{define n-cluster tilting}
Let $n\ge1$. 
Let $\CC$ be a subcategory of $\mod\Lambda$.
We call $\CC$ \emph{$n$-rigid} if $\Ext^i_\Lambda(\CC,\CC)=0$ for any $0<i<n$.
We call $\CC$ \emph{$n$-cluster tilting} if it is functorially finite and
\begin{eqnarray*}
\CC&=&\{X\in\mod\Lambda\ |\ \Ext^i_\Lambda(X,\CC)=0\ (0<i<n)\}\\
&=&\{X\in\mod\Lambda\ |\ \Ext^i_\Lambda(\CC,X)=0\ (0<i<n)\}.
\end{eqnarray*}
This equality can be understood such that the pair $(\CC,\CC)$ forms a
`cotorsion pair' with respect to $\Ext^i$ for $0<i<n$.
We call an object $C\in\mod\Lambda$ \emph{$n$-cluster tilting}
(respectively, \emph{$n$-rigid}) if so is $\add C$.
Clearly $\mod\Lambda$ is a unique $1$-cluster tilting subcategory, and
$2$-cluster tilting subcategories are often called \emph{cluster
  tilting}.
\end{definition}

Let us start with introducing basic terminologies.
We have the duality
\[D:=\Hom_k(-,k):\mod\Lambda\leftrightarrow\mod\Lambda^{\op}.\]
We denote by
\[\nu=\nu_\Lambda:=D\Hom_\Lambda(-,\Lambda):\mod\Lambda\to\mod\Lambda\ \mbox{ and }\ 
\nu^-=\nu_\Lambda^-:=\Hom_{\Lambda^{\op}}(D-,\Lambda):\mod\Lambda\to\mod\Lambda\]
the Nakayama functors of $\Lambda$.
They induce mutually quasi-inverse equivalences
$\nu:\add\Lambda\to\add D\Lambda$ and $\nu^-:\add D\Lambda\to\add\Lambda$.
We denote by
\[\ul{\mod}\Lambda\ \mbox{ and }\ \overline{\mod}\Lambda\]
the stable categories of $\mod\Lambda$ \cite{ARS,ASS}.
For a subcategory $\XX$ of $\mod\Lambda$, we denote by $\ul{\XX}$ (respectively, $\overline{\XX}$) the corresponding subcategory of $\ul{\mod}\Lambda$ (respectively, $\overline{\mod}\Lambda$).
We denote by
\[\Tr:\underline{\mod}\Lambda\leftrightarrow\underline{\mod}\Lambda^{\op},\ \ 
\Omega:\underline{\mod}\Lambda\rightarrow\underline{\mod}\Lambda
\ \mbox{ and }\ 
\Omega^-:\overline{\mod}\Lambda\rightarrow\overline{\mod}\Lambda\]
Auslander-Bridger transpose duality, the syzygy functor and the cosyzygy functor \cite{ABr}.

For $n\ge1$, we define \emph{$n$-Auslander-Reiten translations} \cite{I3} by
\begin{eqnarray*}
\tau_n&:=&D\Tr\Omega^{n-1}:\ul{\mod}\Lambda\to\overline{\mod}\Lambda,\\
\tau_n^-&:=&\Tr D\Omega^{-(n-1)}:\overline{\mod}\Lambda\to\ul{\mod}\Lambda.
\end{eqnarray*}
They are by definition given as follows: For $X\in\mod\Lambda$, take a
minimal projective resolution and a minimal injective resolution
\[P_n\stackrel{f}{\to}P_{n-1}\to\cdots\to P_0\to X\to0\ \mbox{ and }\ 0\to X\to I_0\to\cdots\to I_{n-1}\stackrel{g}{\to}I_n.\]
Then we have
\begin{equation}
\tau_nX=\Ker(\nu P_n\stackrel{\nu f}{\longrightarrow}\nu P_{n-1})\ \mbox{ and }\ \tau_n^-X=\Cok(\nu^-I_{n-1}\stackrel{\nu^-g}{\longrightarrow}\nu^-I_n).
\end{equation}
The functors $\tau=\tau_1=D\Tr$ and $\tau^-=\tau_1^-=\Tr D$ are classical Auslander-Reiten translations,
and we have $\tau_n=\tau\Omega^{n-1}$ and $\tau_n^-=\tau^-\Omega^{-(n-1)}$ by definition.
Moreover $X\in\mod\Lambda$ satisfies $\tau_nX=0$ (respectively,
$\tau_n^-X=0$) if and only if $\pd X_\Lambda<n$ (respectively, $\id
X_\Lambda<n$).

For the case $\gl\Lambda\le n$, clearly $\tau_n$ and $\tau_n^-$ are induced
by the functors $D\Ext^n_\Lambda(-,\Lambda):\mod\Lambda\to\mod\Lambda$
and $\Ext^n_{\Lambda^{\op}}(D-,\Lambda):\mod\Lambda\to\mod\Lambda$
respectively. In this case, we always lift
$\tau_n$ and $\tau_n^-$ to endofunctors of $\mod\Lambda$ by putting
\begin{eqnarray*}
\tau_n&:=&D\Ext^n_\Lambda(-,\Lambda):\mod\Lambda\to\mod\Lambda,\\
\tau_n^-&:=&\Ext^n_{\Lambda^{\op}}(D-,\Lambda):\mod\Lambda\to\mod\Lambda.
\end{eqnarray*}
Then $\tau_n$ (respectively, $\tau_n^-$) clearly
preserves monomorphisms (respectively, epimorphisms) in $\mod\Lambda$.

\medskip
Let us consider the relationship between the functor $\tau_n$ and
$n$-cluster tilting subcategories.
The following results \cite[Th. 2.3]{I3} show that the functor
$\tau_n$ plays the role of Auslander-Reiten translation
for $n$-cluster tilting subcategories.

\begin{proposition}\label{n-AR translation}
\begin{itemize}
\item[(a)] For any $n$-cluster tilting subcategory $\CC$ of $\mod\Lambda$, the functors $\tau_n$ and $\tau_n^-$ induce mutually quasi-inverse equivalences
$\tau_n:\underline{\CC}\to\overline{\CC}$ and
$\tau_n^-:\overline{\CC}\to\underline{\CC}$.
\item[(b)] $\tau_n$ gives a bijection from isoclasses of indecomposable
non-projective objects in $\CC$ to
isoclasses of indecomposable non-injective objects in $\CC$.
\end{itemize}
\end{proposition}

Immediately we have the following results.

\begin{proposition}\label{2 kinds}
Let $M$ be an $n$-cluster tilting object of $\mod\Lambda$.
\begin{itemize}
\item[(a)] For any indecomposable object $X\in\add M$, precisely one of the following statement holds.
\begin{itemize}
\item[(i)] $X$ is $\tau_n$-periodic, i.e. $\tau_n^\ell X\simeq X$ for
  some $\ell>0$.
\item[(ii)] $X\simeq \tau_n^\ell I$ for some indecomposable injective $\Lambda$-module $I$ and $\ell\ge0$, and
  $X\simeq \tau_n^{-m}P$ for some indecomposable projective $\Lambda$-module $P$ and $m\ge0$.
\end{itemize}
\item[(b)] A bijection from isoclasses of indecomposable
injective $\Lambda$-modules to isoclasses of indecomposable projective
$\Lambda$-modules
is given by $I\mapsto\tau_n^{\ell_I}I$, where $\ell_I$ is a maximal
number $\ell$ satisfying $\tau_n^\ell I\neq0$.
\item[(c)] If $\gl\Lambda\le n$, then the above (i) does not occur.
\end{itemize}
\end{proposition}

\begin{proof}
(a)(b) Immediate from Proposition \ref{n-AR translation}(b).

(c) Fix any indecomposable object $X\in\add M$.
We consider two possibilities in (a).
It is enough to show that $\tau_n^\ell X=0$ holds for some $\ell\ge0$.
Take an injective hull $0\to X\to I$.
Since $\tau_n$ preserves monomorphisms because $\gl\Lambda\le n$,
we have an exact sequence $0\to\tau_n^iX\to\tau_n^iI$ for any $i\ge0$.
Since $\tau_n^\ell I=0$ holds for sufficiently large $\ell$ by (b),
we have $\tau_n^\ell X=0$. Thus we have shown the assertion.
\end{proof}

These observation motivates to introduce the following analogue of
preinjective modules, which was studied by Auslander-Solberg for $n=1$
\cite{ASo}.

\begin{definition}
We define the \emph{$\tau_n$-closure} of $D\Lambda$ by
\[\MM=\MM_n(D\Lambda):=\add\{\tau_n^i(D\Lambda)\ |\ i\ge0\}\subset\mod\Lambda.\]
\end{definition}

Immediately from Proposition \ref{n-AR translation}(a), we have the
following result.

\begin{proposition}\label{M and C}
Any $n$-cluster tilting subcategory $\CC$ of $\mod\Lambda$ contains $\MM$.
\end{proposition}

Summarizing Propositions \ref{2 kinds}(c) and \ref{M and C},
we have the uniqueness result of $n$-cluster tilting objects
for algebras $\Lambda$ with $\gl\Lambda\le n$,
which is not valid if we drop the assumption $\gl\Lambda\le n$.

\begin{theorem}\label{unique n-cluster}
Assume $\gl\Lambda\le n$ and that $\Lambda$ has an $n$-cluster tilting
object $M$. Then $\MM=\add M$ holds. In particular, $\MM$ is a unique
$n$-cluster tilting subcategory of $\mod\Lambda$.
\end{theorem}

Thus the condition $\gl\Lambda\le n$ seems to be basic in the
study of $n$-cluster tilting subcategories.

We note here that $\MM$ enjoys nice properties below for the case $n=2$,
which we shall prove in Section \ref{section: preliminaries}.
In particular, $\MM$ provides us a rich source of $2$-rigid objects.

\begin{proposition}\label{n-AR translation new}
Assume $n=2$.
\begin{itemize}
\item[(a)] $\MM$ is $2$-rigid.
\item[(b)] Assume $\gl\Lambda\le 2$. Then $\Lambda$ has a
$2$-cluster tilting object if and only if $\Lambda\in\MM$.
\end{itemize}
\end{proposition}

Let us calculate $\MM$ for a few examples.

\begin{example}\normalfont\label{M calculation}
Let $\Lambda$ and $\Lambda'$ be Auslander algebras
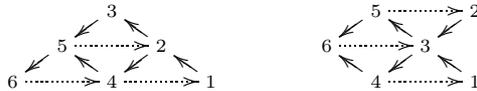
\[\xymatrix@C=0.3cm@R0.1cm{
&&{\scriptstyle 3}\ar[dl]\\
&{\scriptstyle 5}\ar[dl]\ar@{..>}[rr]&&{\scriptstyle 2}\ar[dl]\ar[ul]\\
{\scriptstyle 6}\ar@{..>}[rr]&&{\scriptstyle 4}\ar[ul]\ar@{..>}[rr]&&{\scriptstyle 1}\ar[ul]}\ \ \ \ \ \ \ \ \ \ \xymatrix@C=0.3cm@R0.1cm{
&{\scriptstyle 5}\ar[dl]\ar@{..>}[rr]&&{\scriptstyle 2}\ar[dl]\\
{\scriptstyle 6}\ar@{..>}[rr]&&{\scriptstyle 3}\ar[dl]\ar[ul]\\
&{\scriptstyle 4}\ar[ul]\ar@{..>}[rr]&&{\scriptstyle 1}\ar[ul]}\]
Then one can calculate
\begin{eqnarray*}
&D\Lambda=
\left(\begin{smallmatrix}
1
\end{smallmatrix}\oplus
\begin{smallmatrix}
1&\\
&2
\end{smallmatrix}\oplus
\begin{smallmatrix}
1&&\\
&2&\\
&&3
\end{smallmatrix}\oplus
\begin{smallmatrix}
&2\\
4&
\end{smallmatrix}\oplus
\begin{smallmatrix}
&2&\\
4&&3\\
&5&
\end{smallmatrix}\oplus
\begin{smallmatrix}
&&3\\
&5&\\
6&&
\end{smallmatrix}\right),\ 
\tau_2(D\Lambda)=
\left(\begin{smallmatrix}
4
\end{smallmatrix}\oplus
\begin{smallmatrix}
4&\\
&5
\end{smallmatrix}\oplus
\begin{smallmatrix}
&5\\
6&
\end{smallmatrix}\right),\ 
\tau_2^2(D\Lambda)=
\left(\begin{smallmatrix}
6
\end{smallmatrix}\right), 
\tau_2^3(D\Lambda)=0,&\\
&D\Lambda'=
\left(\begin{smallmatrix}
1
\end{smallmatrix}\oplus
\begin{smallmatrix}
2
\end{smallmatrix}\oplus
\begin{smallmatrix}
1&&2\\
&3&
\end{smallmatrix}\oplus
\begin{smallmatrix}
&&2\\
&3&\\
4&&
\end{smallmatrix}\oplus
\begin{smallmatrix}
1&&\\
&3&\\
&&5
\end{smallmatrix}\oplus
\begin{smallmatrix}
&3&\\
4&&5\\
&6&
\end{smallmatrix}\right),\ 
\tau_2(D\Lambda')=
\left(\begin{smallmatrix}
4
\end{smallmatrix}\oplus
\begin{smallmatrix}
5
\end{smallmatrix}\oplus
\begin{smallmatrix}
4&&5\\
&6&
\end{smallmatrix}\right),\ 
\tau_2^2(D\Lambda')=0.&
\end{eqnarray*}
The quivers of $\MM_2(D\Lambda)$ and $\MM_2(D\Lambda')$ are the following, where dotted arrows indicate $\tau_2$.
{\tiny
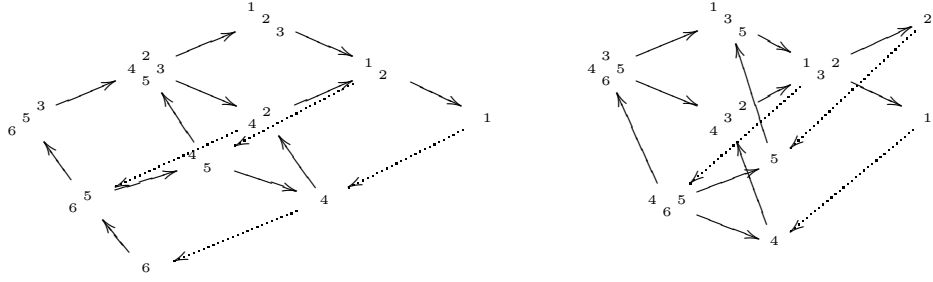
\[\xymatrix@C=0.0cm@R0.0cm{
&&&&{\begin{smallmatrix}
1&&\\
&2&\\
&&3
\end{smallmatrix}}\ar[drr]\\
&&{\begin{smallmatrix}
&2&\\
4&&3\\
&5&
\end{smallmatrix}}\ar[drr]\ar[urr]
&&&&{\begin{smallmatrix}
1&&\ \\
&2&
\end{smallmatrix}}\ar[drrrr]\ar@{..>}[ddlll]\\
{\begin{smallmatrix}
&&3\\
&5&\\
6&&
\end{smallmatrix}}\ar[urr]
&&&&{\begin{smallmatrix}
&2&\ \\
4&&
\end{smallmatrix}}\ar[urr]\ar@{..>}[ddlll]
&&&&&&{\begin{smallmatrix}
\ &1&\ 
\end{smallmatrix}}\ar@{..>}[ddlllll]\\
&&&{\begin{smallmatrix}
4&&\ \\
&5&
\end{smallmatrix}}\ar[uul]\ar[drr]\\
&{\begin{smallmatrix}
\ \\
&5&\ \\
6&&
\end{smallmatrix}}\ar[uul]\ar[urr]
&&&&{\begin{smallmatrix}
\ &4&\ 
\end{smallmatrix}}\ar[uul]\ar@{..>}[dddlll]\\
\ \\
\ \\
&&{\begin{smallmatrix}
\ \\
\ &6&\ \\
\ 
\end{smallmatrix}}\ar[uuul]}
\ \ \ \ \ \ \ \ \ \xymatrix@C=0.0cm@R0.0cm{
&&{\begin{smallmatrix}
1&&\\
&3&\\
&&5
\end{smallmatrix}}\ar[drr]
&&&&&&{\begin{smallmatrix}
2
\end{smallmatrix}}\ar@{..>}[dddlllll]\\
{\begin{smallmatrix}
&3&\\
4&&5\\
&6&
\end{smallmatrix}}\ar[drr]\ar[urr]
&&&&{\begin{smallmatrix}
1&&2\\
&3&
\end{smallmatrix}}\ar[drrrr]\ar[urrrr]\ar@{..>}[dddlll]\\
&&{\begin{smallmatrix}
&&2\\
&3&\\
4&&
\end{smallmatrix}}\ar[urr]
&&&&&&{\begin{smallmatrix}
\ &1&\ 
\end{smallmatrix}}\ar@{..>}[dddlllll]\\
&&&{\begin{smallmatrix}
\ \\
5\\
\ 
\end{smallmatrix}}\ar[uuul]\\
&{\begin{smallmatrix}
\ \\
4&&5\\
&6&
\end{smallmatrix}}\ar[drr]\ar[urr]\ar[uuul]\\
&&&{\begin{smallmatrix}
\ \\
4\\
\ 
\end{smallmatrix}}\ar[uuul]
}\]}
By Proposition \ref{n-AR translation new}(a), we have that
$\MM_2(D\Lambda)$ is a 2-cluster tilting subcategory of $\mod\Lambda$,
while $\MM_2(D\Lambda')$ is \emph{not} a 2-cluster tilting subcategory
of $\mod\Lambda'$.
Nevertheless $\MM_2(D\Lambda')$ can be regarded as a $2$-cluster tilting subcategory of a certain subcategory of $\mod\Lambda$ defined as follows.
\end{example}

\begin{definition}\label{define n-cluster tilting 2}
{\rm (Relative version of Definition \ref{define n-cluster tilting})}
Let $n\ge1$. 
Let $\XX$ be an extension closed subcategory of $\mod\Lambda$.
We call a subcategory $\CC$ of $\XX$ \emph{$n$-cluster tilting} if it is functorially finite and
\begin{eqnarray*}
\CC&=&\{X\in\XX\ |\ \Ext^i_\Lambda(X,\CC)=0\ (0<i<n)\}\\
&=&\{X\in\XX\ |\ \Ext^i_\Lambda(\CC,X)=0\ (0<i<n)\}.
\end{eqnarray*}
We call an object $C\in\XX$ \emph{$n$-cluster tilting} if so is $\add C$.
\end{definition}

Especially we deal with subcategories $\XX$ of $\mod\Lambda$
associated with tilting $\Lambda$-modules.
Recall that a $\Lambda$-module $T$ is called \emph{tilting} \cite{M,H} if there exists $m\ge0$ such that
\begin{itemize}
\item $\pd T_\Lambda\le m$,
\item $\Ext^i_\Lambda(T,T)=0$ for any $i>0$,
\item there exists an exact sequence $0\to\Lambda\to T_0\to\cdots\to T_m\to0$ with $T_i\in\add T$.
\end{itemize}
In this case, we have an extension closed subcategory
\[T^{\perp}=\{X\in\mod\Lambda\ |\ \Ext^i_\Lambda(T,X)=0\ (0<i)\}\]
of $\mod\Lambda$. This is a functorially finite subcategory of $\mod\Lambda$,
and plays an important role in tilting theory \cite{H} analogous to
the category of Cohen-Macaulay modules over commutative rings
\cite{ABu,AR3}.

For tilting modules $T$ with $\pd T_\Lambda\le m$, we call $n$-cluster
tilting subcategories of the category $T^\perp$ as
\emph{$m$-relative}.
For the case $m=0$ (i.e. $T^\perp=\mod\Lambda$), we use the
terminology \emph{absolute} instead of $0$-relative.
In Example \ref{M calculation}, $\MM_2(D\Lambda)$ is an absolute
$2$-cluster tilting subcategory of $\Lambda$, and
$\MM_2(D\Lambda')$ is a $1$-relative $2$-cluster tilting subcategory
of $\Lambda$ associated to a tilting $\Lambda'$-module
$T=\left(\begin{smallmatrix}
&&2\\
&3&\\
4&&
\end{smallmatrix}\oplus
\begin{smallmatrix}
1&&\\
&3&\\
&&5
\end{smallmatrix}\oplus
\begin{smallmatrix}
&3&\\
4&&5\\
&6&
\end{smallmatrix}\oplus
\begin{smallmatrix}
4
\end{smallmatrix}\oplus
\begin{smallmatrix}
5
\end{smallmatrix}\oplus
\begin{smallmatrix}
4&&5\\
&6&
\end{smallmatrix}\right)$.

\medskip
It was shown in \cite[Th. 4.2.1]{I4} that $n$-cluster tilting objects
are closely related to algebras of finite global dimension:

\begin{theorem}\label{Auslander correspondence}
Let $n\ge1$ and $n\ge m\ge 0$.
For a finite dimensional algebra $\Gamma$, the following conditions are equivalent.
\begin{itemize}
\item[(a)] There exists a finite dimensional algebra $\Lambda$ and an $m$-relative $n$-cluster tilting object $M$ of $\Lambda$
such that $\Gamma\simeq\End_\Lambda(M)$.

\item[(b)] The following conditions are satisfied.
\begin{itemize}
\item[(i)] $\gl\Gamma\le n+1$.
\item[(ii)] The minimal injective resolution
\[0\to\Gamma\to I_0\to\cdots\to I_n\to I_{n+1}\to0\]
of the $\Gamma$-module $\Gamma$ satisfies $\pd(I_i)_\Gamma\le m$ for any $0\le i\le n$.
\item[(iii)] The opposite side version of (ii).
\end{itemize}
\end{itemize}
In this case we call $\Gamma$ an \emph{($m$-relative) $n$-Auslander algebra (of $\Lambda$)}.
\end{theorem}

Now let us formalize Examples \ref{M calculation}, where $\MM$ give
a relative $n$-cluster tilting subcategory.

\begin{definition}\label{complete}
Let $\Lambda$ be a finite dimensional algebra and $n\ge1$. Let
$\MM=\MM_n(D\Lambda)$ be the $\tau_n$-closure of $D\Lambda$.
We define subcategories of $\MM$ by
\begin{eqnarray*}
\II(\MM)&:=&\add D\Lambda,\\
\PP(\MM)&:=&\{X\in\MM\ |\ \pd X_\Lambda<n\}=\{X\in\MM\ |\ \tau_nX=0\},\\
\MM_I&:=&\{X\in\MM\ |\ X\ \mbox{ has no non-zero summands in }\ \II(\MM)\},\\
\MM_P&:=&\{X\in\MM\ |\ X\ \mbox{ has no non-zero summands in }\ \PP(\MM)\}.
\end{eqnarray*}
\begin{itemize}
\item We call $\Lambda$ \emph{$\tau_n$-finite} if $\gl\Lambda\le n$
and $\tau_n^\ell(D\Lambda)=0$ holds for sufficiently large $\ell$.
In this case, it is easily shown that $\tau_n^\ell=0$ holds
(e.g. Proof of Proposition \ref{2 kinds}(c)).
\item We call $\Lambda$ \emph{$n$-complete} if $\gl\Lambda\le n$ and
the following conditions (A$_n$)--(C$_n$) are satisfied.
\begin{itemize}
\item[(A$_n$)] There exists a tilting $\Lambda$-module $T$ satisfying $\PP(\MM)=\add T$,
\item[(B$_n$)] $\MM$ is an $n$-cluster tilting subcategory of $T^{\perp}$,
\item[(C$_n$)] $\Ext^i_\Lambda(\MM_P,\Lambda)=0$ for any $0<i<n$.
\end{itemize}
We call $\Lambda$ \emph{absolutely $n$-complete} if $\PP(\MM)=\add\Lambda$.
\end{itemize}
\end{definition}

We have the properties of $n$-complete algebras below,
which we shall prove in Section \ref{section: preliminaries}.
The statement (a)--(c) are similar to Proposition \ref{n-AR translation}
and \ref{2 kinds}(b).

\begin{proposition}\label{weak is weak}
Let $\Lambda$ be an $n$-complete algebra.
\begin{itemize}
\item[(a)] We have mutually quasi-inverse equivalences $\tau_n:\MM_P\to\MM_I$
and $\tau_n^-:\MM_I\to\MM_P$.
\item[(b)] $\tau_n$ gives a bijection from isoclasses of indecomposable
objects in $\MM_P$ to those in $\MM_I$.
\item[(c)] A bijection from isoclasses of indecomposable
objects in $\II(\MM)$ to those in $\PP(\MM)$
is given by $I\mapsto\tau_n^{\ell_I}I$, where $\ell_I$ is a maximal
number $\ell$ satisfying $\tau_n^\ell I\neq0$.
\item[(d)] $\Lambda$ is $\tau_n$-finite.
\end{itemize}
\end{proposition}

In particular, if $\Lambda$ is $n$-complete, then $\MM$ has an additive
generator $M$ by (d) above.
We call $\End_\Lambda(M)$ the \emph{cone} of $\Lambda$.
This is by definition an $(n-1)$-relative $n$-Auslander algebra of $\Lambda$.

For example, any finite dimensional algebra $\Lambda$ with
$\gl\Lambda<n$ is clearly $n$-complete since
$\MM=\PP(\MM)=\II(\MM)=\add D\Lambda$ holds.
It is interesting to know a characterization of $n$-complete algebras.
For the case $n=1$, we have a nice characterization (a) below.
Also the following (b) gives a simple interpretation of
absolute $n$-completeness.

\begin{proposition}\label{1-complete}
\begin{itemize}
\item[(a)] A finite dimensional algebra is $1$-complete if and only if it is representation-finite and hereditary.
\item[(b)] A finite dimensional algebra $\Lambda$ is absolutely $n$-complete if and only if $\gl\Lambda\le n$ and
$\Lambda$ has an absolute $n$-cluster tilting object.
\end{itemize}
\end{proposition}

\begin{proof}
(b) Since the `only if' part is clear, we only have to show the `if' part.
By Theorem \ref{unique n-cluster}, we have that $\MM$ is an
$n$-cluster tilting subcategory of $\mod\Lambda$. Thus (C$_n$) holds.
By Propositions \ref{n-AR translation}(b) and \ref{2 kinds}(b),
we have that $\Lambda$ is $\tau_n$-finite and $\PP(\MM)=\add\Lambda$ hold.
Thus (A$_n$) and (B$_n$) also holds.

(a) Any $1$-complete algebra is absolutely $1$-complete by definition.
Since absolute $1$-cluster tilting objects are nothing but additive
generators of $\mod\Lambda$, the assertion follows from (b).
\end{proof}


Now we state our main theorem in this paper.
It gives an inductive construction of algebras with $n$-cluster
tilting objects.

\begin{theorem}\label{main}
For any $n\ge1$, the cone of an $n$-complete algebra is $(n+1)$-complete.
\end{theorem}

For the case $n=1$, we have the result below immediately from
Proposition \ref{1-complete}(a) and Theorem \ref{main}.
This explains the reason why the Auslander algebras in Example
\ref{M calculation} have $2$-cluster tilting objects.

\begin{corollary}
Let $\Lambda$ be a representation-finite hereditary algebra and
$\Gamma$ an Auslander algebra of $\Lambda$. Then $\Gamma$ has a
$1$-relative $2$-cluster tilting object.
\end{corollary}

Our Theorem \ref{main} gives the following inductive construction
of algebras with $n$-cluster tilting objects.

\begin{corollary}\label{main2}
Let $\Lambda^{(1)}$ be a representation-finite hereditary algebra.
Then there exists an algebra $\Lambda^{(n)}$ for any $n\ge1$ such that
$\Lambda^{(n)}$ is an $n$-complete algebra with the cone
$\Lambda^{(n+1)}$.
\end{corollary}

The quivers with relations of these algebras $\Lambda^{(n)}$ will be
given in Theorem \ref{n-AR quiver}.

Now we apply Corollary \ref{main2} to a special case.
We denote by $T_m(F)$ the $m\times m$ upper triangular matrix algebra
over an algebra $F$.
They form an important class of algebras by the following easy fact.

\begin{proposition}\label{0-Auslander}
Let $\Lambda$ be a ring-indecomposable finite dimensional algebra.
Then $\gl\Lambda\le1\le\dom\Lambda$ holds if and only if $\Lambda$ is
Morita equivalent to $T_m(F)$ for some division algebra $F$ and
$m\ge1$.
\end{proposition}

\begin{proof}
We provide a proof for the convenience of the reader.

Since $\dom\Lambda\ge1$, there exists an indecomposable projective-injective $\Lambda$-module $P_1$.
Put $P_i:=P_1J_\Lambda^{i-1}$ for any $i>0$.
Then there exists $m\ge1$ such that $P_{m+1}=0$.
Since $\gl\Lambda\le1$, each $P_i$ is a projective $\Lambda$-module.
Since $\soc P_1$ is simple, each $P_i$ is indecomposable, and so $P_i$ has a unique maximal submodule $P_{i+1}$.
Consequently $P_1$ has a unique composition series
\[P_1\supset P_2\supset\cdots\supset P_m\supset P_{m+1}=0.\]
We often use the fact that any non-zero morphism between indecomposable projective modules is injective, which is a conclusion of $\gl\Lambda\le1$.
In particular $F:=\End_\Lambda(P_1)$ is a division algebra.
Put $P:=\bigoplus_{i=1}^mP_i$.

(i) We shall show that $\End_\Lambda(P)\simeq T_m(F)$.

Since $P_1$ is injective, any non-zero morphism $P_i\to P_j$ extends to a morphism $P_1\to P_1$, which is an isomorphism.
Thus we have
\[\Hom_\Lambda(P_i,P_j)\simeq\left\{
\begin{array}{cc}
0&(i<j),\\
F&(i\ge j).
\end{array}\right.\]
This implies the assertion.

(ii) We shall show that $P$ is a progenerator of $\Lambda$.

Since $\Lambda$ is ring-indecomposable,
we only have to show that, for any indecomposable projective $\Lambda$-module $Q$ such that $\Hom_\Lambda(Q,P)\neq0$ or $\Hom_\Lambda(P,Q)\neq0$,
we have $Q\in\add P$.
If $\Hom_\Lambda(Q,P)\neq0$, then $\Hom_\Lambda(Q,P_1)\neq0$. Thus $Q$ is a submodule of $P_1$, and we have $Q\simeq P_i$ for some $i$.
If $\Hom_\Lambda(P,Q)\neq0$, then $\Hom_\Lambda(Q,P_1)\neq0$ since $P_1$ is an injective hull of each $P_i$.
Thus $Q\in\add P$ by the previous observation.
\end{proof}

Applying Corollary \ref{main2} to $T_m(F)$, we have a family of algebras with
$n$-cluster tilting objects. Moreover, they are absolute by the
following result.

\begin{theorem}\label{recall}
For any division algebra $F$ and $m\ge1$, there exists an algebra
$T_m^{(n)}(F)$ for any $n\ge 1$ such that $T_m^{(1)}(F)=T_m(F)$ and
$T_m^{(n)}(F)$ is an absolutely $n$-complete algebra with the cone
$T_m^{(n+1)}(F)$.
\end{theorem}

The quiver of $T_m^{(n)}(k)$ will be given in Theorem \ref{n-AR
  quiver} as follows. It looks like an $(n+1)$-simplex.
The relations are given by commutative relations for each small square,
and zero relations for each small half square.
{\tiny\[T_4^{(1)}(k)\ \ 
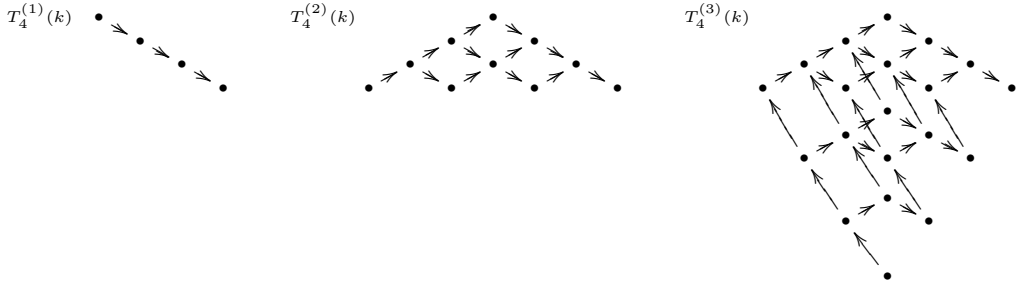
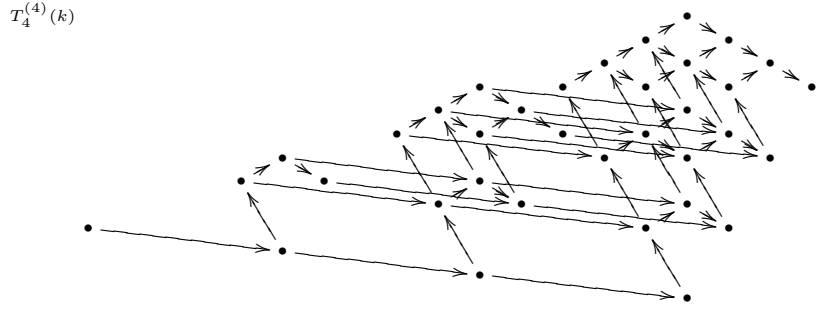
\xymatrix@C=0.2cm@R0.0cm{
\bullet\ar[dr]\\
&\bullet\ar[dr]\\
&&\bullet\ar[dr]\\
&&&\bullet}
\ \ \ \ \ \ \ \ T_4^{(2)}(k)\xymatrix@C=0.2cm@R0.0cm{
&&&\bullet\ar[dr]\\
&&\bullet\ar[ur]\ar[dr]&&\bullet\ar[dr]\\
&\bullet\ar[ur]\ar[dr]&&\bullet\ar[ur]\ar[dr]&&\bullet\ar[dr]\\
\bullet\ar[ur]&&\bullet\ar[ur]&&\bullet\ar[ur]&&\bullet}
\ \ \ \ \ \ \ \ T_4^{(3)}(k)\xymatrix@C=0.2cm@R0.0cm{
&&&\bullet\ar[dr]\\
&&\bullet\ar[ur]\ar[dr]&&\bullet\ar[dr]\\
&\bullet\ar[ur]\ar[dr]&&\bullet\ar[ur]\ar[dr]&&\bullet\ar[dr]\\
\bullet\ar[ur]&&\bullet\ar[ur]&&\bullet\ar[ur]&&\bullet\\
 &&&\bullet\ar[dr]\ar[uuul]\\
 &&\bullet\ar[ur]\ar[dr]\ar[uuul]&&\bullet\ar[dr]\ar[uuul]\\
 &\bullet\ar[ur]\ar[uuul]&&\bullet\ar[ur]\ar[uuul]&&\bullet\ar[uuul]\\
\ \\
&&&\bullet\ar[dr]\ar[uuul]\\
&&\bullet\ar[ur]\ar[uuul]&&\bullet\ar[uuul]\\
\ \\
\ \\
&&&\bullet\ar[uuul]
}\]}
{\tiny\[
T_4^{(4)}(k)\xymatrix@C=0.2cm@R0.0cm{
&&&&&&&&&&&&&\bullet\ar[dr]&&&&\\
&&&&&&&&&&&&\bullet\ar[ur]\ar[dr]&&\bullet\ar[dr]&&&\\
&&&&&&&&&&&\bullet\ar[ur]\ar[dr]&&\bullet\ar[ur]\ar[dr]&&\bullet\ar[dr]&&\\
&&&&&&&&\bullet\ar[dr]\ar[drrrrr]&&
\bullet\ar[ur]&&\bullet\ar[ur]&&\bullet\ar[ur]&&\bullet\\
&&&&&&&\bullet\ar[ur]\ar[dr]\ar[drrrrr]&&\bullet\ar[dr]\ar[drrrrr]&&
&&\bullet\ar[dr]\ar[uuul]\\
&&&&&&\bullet\ar[ur]\ar[drrrrr]&&\bullet\ar[ur]\ar[drrrrr]&&\bullet\ar[drrrrr]
&&\bullet\ar[ur]\ar[dr]\ar[uuul]&&\bullet\ar[dr]\ar[uuul]&&\\
&&&\bullet\ar[dr]\ar[drrrrr]&&&&&&&
&\bullet\ar[ur]\ar[uuul]&&\bullet\ar[ur]\ar[uuul]&&\bullet\ar[uuul]&&\\
&&\bullet\ar[ur]\ar[drrrrr]&&\bullet\ar[drrrrr]&&
&&\bullet\ar[dr]\ar[uuul]\ar[drrrrr]\\
&&&&&&&\bullet\ar[ur]\ar[uuul]\ar[drrrrr]&&\bullet\ar[uuul]\ar[drrrrr]&
&&&\bullet\ar[dr]\ar[uuul]\\
\bullet\ar[drrr]&\ \ \ \ \ \ \ \ \ \ \ \ &&&&&&&&&
&&\bullet\ar[ur]\ar[uuul]&&\bullet\ar[uuul]\\
&&&\bullet\ar[uuul]\ar[drrrrr]\\
&&&&&&&&\bullet\ar[uuul]\ar[drrrrr]\\
&&&&&&&&&&&&&\bullet\ar[uuul]
}\]}

While there are a lot of algebras with relative $n$-cluster tilting
objects, algebras with absolute $n$-cluster tilting objects are rather rare.
In fact the following result shows a certain converse of Theorem
\ref{recall}.

\begin{theorem}\label{iterated Auslander}
Let $\Lambda$ be a ring-indecomposable finite dimensional algebra
satisfying $\gl\Lambda\le n\le\dom\Lambda$ for some $n\ge1$.
Then $\Lambda$ has an absolute $n$-cluster tilting subcategory
if and only if $\Lambda$ is Morita equivalent to $T_m^{(n)}(F)$ for
some division algebra $F$ and $m\ge1$.
\end{theorem}

For example an Auslander algebra $\Gamma$ of a ring-indecomposable
representation-finite algebra $\Lambda$
has an absolute $2$-cluster tilting object if and only if $\Lambda$ is
Morita equivalent to $T_m^{(2)}(F)$ for some division algebra $F$ and
$m\ge1$.

A key step to prove Theorem \ref{iterated Auslander} is the following
more explicit version of Auslander correspondence
(Theorem \ref{Auslander correspondence}), which will be shown in
Proposition \ref{gl.dim n}.

\begin{theorem}
Let $n\ge1$. For a finite dimensional algebra $\Gamma$, the following conditions are equivalent.
\begin{itemize}
\item[(a)] There exists a finite dimensional algebra $\Lambda$ with $\gl\Lambda\le n$
and an absolute $n$-cluster tilting object $M$ of $\Lambda$ such that $\Gamma\simeq\End_\Lambda(M)$.
\item[(b)] $\gl\Gamma\le n+1\le\dom\Gamma$ and $\Ext^i_\Gamma(D\Gamma,\Gamma)=0$ for any $0<i\le n$.
\end{itemize}
\end{theorem}

In forthcoming papers \cite{IH,IO1,IO2}, absolutely $n$-complete algebras
will be called \emph{$n$-representation-finite algebras} and a lot of examples
will be constructed. Also combinatorial aspects of $T_m^{(n)}(F)$ will be
studied in \cite{OT}.

\subsection{$n$-cluster tilting in derived categories}

In Section \ref{section: derived},
we construct $n$-cluster tilting subcategories in triangulated
categories.
Let $\Lambda$ be a finite dimensional algebra with
$\id{}_{\Lambda}\Lambda=\id\Lambda_\Lambda<\infty$.
We denote by
\[\DD:=\KK^{\rm b}(\pr\Lambda)\]
the homotopy category of bounded complexes of finitely generated
projective $\Lambda$-modules,
and we identify it with $\KK^{\rm b}(\inj\Lambda)$ in the derived
category of $\Lambda$.
As in Definition \ref{define n-cluster tilting}, we call a
functorially finite subcategory $\CC$ of $\DD$ \emph{$n$-cluster
  tilting} if
\begin{eqnarray*}
\CC&=&\{X\in\DD\ |\ \Hom_{\DD}(X,\CC[i])=0\ (0<i<n)\}\\
&=&\{X\in\DD\ |\ \Hom_{\DD}(\CC,X[i])=0\ (0<i<n)\}.
\end{eqnarray*}
If $\mod\Lambda$ has an absolute $n$-cluster tilting subcategory and
$\gl\Lambda\le n$, then $\DD$ also has an $n$-cluster tilting subcategory
by the following result.

\begin{theorem}\label{derived}
Let $\Lambda$ be a finite dimensional algebra with $\gl\Lambda\le
n$. If $\CC$ is an absolute $n$-cluster tilting subcategory of
$\mod\Lambda$, then
\begin{equation}\label{tilde C}
\CC[n\Z]:=\add\{X[\ell n]\ |\ X\in\CC,\ \ell\in\Z\}
\end{equation}
is an $n$-cluster tilting subcategory of $\DD$.
\end{theorem}

Notice that we can not drop the assumption that $\CC$ is absolute.

For the case $n=1$, we have $\CC[\Z]=\DD$, which means the well-known
fact that any object in $\DD$ is a direct sum of stalk complexes
if $\Lambda$ is hereditary.
It is natural to hope that $\CC[n\Z]$ forms an
$(n+2)$-angulated category under a certain proper definition (\cite{GKO}).

In Theorem \ref{n-AR quiver}, we draw Auslander-Reiten quivers of
$\CC[n\Z]$ for $\Lambda=T_m^{(n)}(k)$ as follows.
The relations are again given by commutative relations for each small square,
and zero relations for each small half square.
{\tiny\[T_4^{(1)}(k)
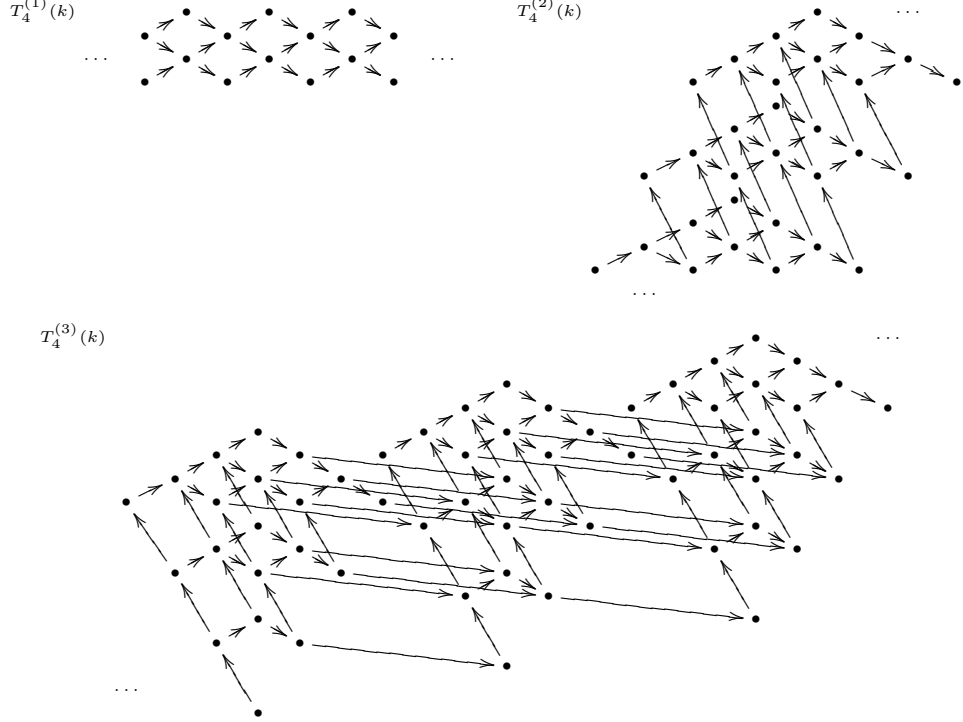
\xymatrix@C=0.2cm@R0.0cm{
&&\bullet\ar[dr]&&\bullet\ar[dr]&&\bullet\ar[dr]\\
&\bullet\ar[ur]\ar[dr]&&\bullet\ar[ur]\ar[dr]&&\bullet\ar[dr]\ar[ur]&&
\bullet\\
\cdots&&\bullet\ar[ur]\ar[dr]&&\bullet\ar[ur]\ar[dr]&&\bullet\ar[dr]\ar[ur]&&\cdots\\
&\bullet\ar[ur]&&\bullet\ar[ur]&&\bullet\ar[ur]&&\bullet}
\ \ \ \ \ \ \ \ T_4^{(2)}(k)\xymatrix@C=0.2cm@R0.0cm{
&&&&&\bullet\ar[dr]&&\cdots\\
&&&&\bullet\ar[ur]\ar[dr]&&\bullet\ar[dr]\\
&&&\bullet\ar[ur]\ar[dr]&&\bullet\ar[ur]\ar[dr]&&\bullet\ar[dr]\\
&&\bullet\ar[ur]&&\bullet\ar[ur]&&\bullet\ar[ur]&&\bullet\\
 &&&&\bullet\ar[dr]&&\\
 &&&\bullet\ar[ur]\ar[dr]&&\bullet\ar[dr]\ar[uuuul]\\
 &&\bullet\ar[ur]\ar[dr]&&\bullet\ar[ur]\ar[dr]\ar[uuuul]&&\bullet\ar[dr]\ar[uuuul]\\
 &\bullet\ar[ur]&&\bullet\ar[ur]\ar[uuuul]&&\bullet\ar[ur]\ar[uuuul]&&\bullet\ar[uuuul]\\
  &&&\bullet\ar[dr]&&\\
  &&\bullet\ar[ur]\ar[dr]&&\bullet\ar[dr]\ar[uuuul]\\
  &\bullet\ar[ur]\ar[dr]&&\bullet\ar[ur]\ar[dr]\ar[uuuul]&&\bullet\ar[dr]\ar[uuuul]\\
  \bullet\ar[ur]&&\bullet\ar[ur]\ar[uuuul]&&\bullet\ar[ur]\ar[uuuul]&&\bullet\ar[uuuul]\\
  &\cdots
 }\]}
{\tiny\[T_4^{(3)}(k)\xymatrix@C=0.2cm@R0.0cm{
&&&&&&&&&&&&&&&\bullet\ar[dr]&&&\cdots&\\
&&&&&&&&&&&&&&\bullet\ar[ur]\ar[dr]&&\bullet\ar[dr]&&&\\
&&&&&&&&&\bullet\ar[dr]&&&&\bullet\ar[ur]\ar[dr]&&\bullet\ar[ur]\ar[dr]&&\bullet\ar[dr]&&\\
&&&&&&&&\bullet\ar[ur]\ar[dr]&&\bullet\ar[dr]\ar[drrrrr]&&
\bullet\ar[ur]&&\bullet\ar[ur]&&\bullet\ar[ur]&&\bullet\\
&&&\bullet\ar[dr]&&&&\bullet\ar[ur]\ar[dr]&&\bullet\ar[ur]\ar[dr]\ar[drrrrr]&&\bullet\ar[dr]\ar[drrrrr]&&
&&\bullet\ar[dr]\ar[uuul]\\
&&\bullet\ar[ur]\ar[dr]&&\bullet\ar[dr]\ar[drrrrr]&&\bullet\ar[ur]&&
\bullet\ar[ur]\ar[drrrrr]&&\bullet\ar[ur]\ar[drrrrr]&&\bullet\ar[drrrrr]
&&\bullet\ar[ur]\ar[dr]\ar[uuul]&&\bullet\ar[dr]\ar[uuul]&&\\
&\bullet\ar[ur]\ar[dr]&&\bullet\ar[ur]\ar[dr]\ar[drrrrr]&&\bullet\ar[dr]\ar[drrrrr]&&&&\bullet\ar[dr]\ar[uuul]&&&
&\bullet\ar[ur]\ar[uuul]&&\bullet\ar[ur]\ar[uuul]&&\bullet\ar[uuul]&&\\
\bullet\ar[ur]&&\bullet\ar[ur]\ar[drrrrr]&&\bullet\ar[ur]\ar[drrrrr]&&\bullet\ar[drrrrr]&&
\bullet\ar[ur]\ar[dr]\ar[uuul]&&\bullet\ar[dr]\ar[uuul]\ar[drrrrr]\\
&&&\bullet\ar[dr]\ar[uuul]&&&&\bullet\ar[ur]\ar[uuul]&&\bullet\ar[ur]\ar[uuul]\ar[drrrrr]
&&\bullet\ar[uuul]\ar[drrrrr]&
&&&\bullet\ar[dr]\ar[uuul]\\
&&\bullet\ar[ur]\ar[dr]\ar[uuul]&&\bullet\ar[dr]\ar[uuul]\ar[drrrrr]&&&&&&&&
&&\bullet\ar[ur]\ar[uuul]&&\bullet\ar[uuul]\\
&\bullet\ar[ur]\ar[uuul]&&\bullet\ar[ur]\ar[uuul]\ar[drrrrr]&&\bullet\ar[uuul]\ar[drrrrr]&&&&\bullet\ar[dr]\ar[uuul]\\
&&&&&&&&\bullet\ar[ur]\ar[uuul]&&\bullet\ar[uuul]\ar[drrrrr]\\
&&&\bullet\ar[dr]\ar[uuul]&&&&&&&&&&&&\bullet\ar[uuul]\\
&&\bullet\ar[ur]\ar[uuul]&&\bullet\ar[uuul]\ar[drrrrr]\\
&&&&&&&&&\bullet\ar[uuul]\\
\cdots&&\\
&&&\bullet\ar[uuul]
}\]}

We also give another construction of an $n$-cluster tilting subcategory
of $\DD$ by using derived analogue of $n$-Auslander-Reiten
translations.
Recall that there exists an autoequivalence
\[\SSS:=D\circ\RHom_\Lambda(-,\Lambda)\simeq-\stackrel{\mathbf{L}}{\otimes}_\Lambda(D\Lambda):\DD\to\DD,\]
which gives the Serre functor of $\DD$,
i.e. there exists a functorial isomorphism
\[\Hom_{\DD}(X,Y)\simeq D\Hom_{\DD}(Y,\SSS X)\]
for any $X,Y\in\DD$ \cite{H,BK}. We define an autoequivalence of $\DD$ by
\[\SSS_n:=\SSS\circ[-n]:\DD\to\DD.\]
Any $n$-cluster tilting subcategory $\CC$ of $\DD$ satisfies
$\CC=\SSS_n\CC=\SSS_n^{-1}\CC$ \cite[Prop. 3.4]{IY}.
Therefore $\SSS_n$ plays the role of $n$-Auslander-Reiten translations, and it is natural to introduce the following subcategory.

\begin{definition}\label{S_n closure}
Define the \emph{$\SSS_n$-closure} of an object $X\in\DD$ by
\[\UU_n(X)=\add\{\SSS_n^\ell(X)\ |\ \ell\in\Z\}.\]
\end{definition}

The categories $\MM_n(D\Lambda)$ and $\UU_n(D\Lambda)$ are closely
related since $\tau_n^\ell\simeq H^0(\SSS_n^\ell-)$ holds on
$\mod\Lambda$ for any $\ell\ge0$ if $\gl\Lambda\le n$ by Lemma
\ref{tau and s_n}.
In particular, $\CC[n\Z]=\UU_n(\Lambda)$ holds in Theorem
\ref{derived} if $\CC$ has an additive generator.

We shall study the problem whether $\UU_n(\Lambda)$ is an $n$-cluster
tilting subcategory of $\DD$.
For a hereditary algebra $\Lambda$, one can easily show that
$\UU_1(\Lambda)$ is a $1$-cluster tilting subcategory of $\DD$ if and
only if $\Lambda$ is representation-finite.
This observation suggests that it is related to the $n$-complete
property. In fact we have the following another main result in Section
\ref{section: derived}.

\begin{theorem}\label{consequence}
Let $\Lambda$ be a $\tau_n$-finite algebra.
Then $\UU_n(\Lambda)$ is an $n$-cluster tilting subcategory of $\DD$.
Moreover, $\UU_n(T)$ is an $n$-cluster tilting subcategory of $\DD$
for any tilting complex $T\in\DD$ satisfying $\gl\End_{\DD}(T)\le n$.
\end{theorem}

As a special case, if $\gl\Lambda<n$ and $T$ is a tilting
$\Lambda$-module with $\pd T_\Lambda\le 1$, then $\gl\End_{\DD}(T)\le
n$ holds and $\UU_n(T)$ is an $n$-cluster tilting subcategory of $\DD$.
This generalizes the construction of $2$-cluster tilting objects in
cluster categories using tilting modules given in \cite{BMRRT}
as well as recent work of Amiot \cite[Prop. 5.4.2]{Am1} (see also
\cite[Th. 4.10]{Am2}) and Barot-Fernandez-Platzeck-Pratti-Trepode \cite{BFPPT}.

As a special case of Theorem \ref{consequence}, $\UU_n(\Lambda^{(n)})$
forms an $n$-cluster tilting subcategory of $\DD$ for algebras
$\Lambda^{(n)}$ given in Corollary \ref{main2}. The quivers with relations
of these categories will be given in Theorem \ref{n-AR quiver}.

\medskip
At the end of this section, we note the following left-right
symmetry of $\tau_n$-finite algebras, which will be shown
in Section \ref{section: derived}.

\begin{proposition}\label{left right}
A finite dimensional algebra $\Lambda$ is $\tau_n$-finite if and
only if $\Lambda^{\rm op}$ is $\tau_n$-finite.
\end{proposition}

We notice that easy examples show that $n$-completeness is not left-right
symmetric (e.g. \cite{IO1}).

\section{Preliminaries}\label{section: preliminaries}

In this section, we give some preliminary results.
Let us start with some properties of $n$-cluster titing objects.

\begin{definition}
Let $\CC$ be a Krull-Schmidt category.
\begin{itemize}
\item[(a)] For an object $X\in\CC$, a morphism $f_0\in J_{\CC}(X,C_1)$ is called \emph{left almost split} if $C_1\in\CC$ and
\[\Hom_{\CC}(C_1,-)\stackrel{f_0}{\to}J_{\CC}(X,-)\to0\]
is exact on $\CC$. A left minimal and left almost split morphism is called a \emph{source morphism}.

\item[(b)] We call a complex
\begin{eqnarray}\label{X source}
X\stackrel{f_0}{\to}C_1\stackrel{f_1}{\to}C_2\stackrel{f_2}{\to}\cdots
\end{eqnarray}
a \emph{source sequence} of $X$ if the following conditions are satisfied.
\begin{itemize}
\item[(i)] $C_i\in\CC$ and $f_i\in J_{\CC}$ for any $i$,

\item[(ii)] we have the following exact sequence on $\CC$.
\begin{eqnarray}\label{X source2}
\cdots\stackrel{f_2}{\to}\Hom_{\CC}(C_2,-)\stackrel{f_1}{\to}\Hom_{\CC}(C_1,-)\stackrel{f_0}{\to}J_{\CC}(X,-)\to0
\end{eqnarray}
\end{itemize}
A \emph{sink morphism} and a \emph{sink sequence} are defined dually.

\item[(c)] We call a complex
\[0\to X\stackrel{}{\to}C_1\stackrel{}{\to}C_2\stackrel{}{\to}\cdots\stackrel{}{\to}C_{n}\stackrel{}{\to}Y\to0\]
an \emph{$n$-almost split sequence} if this is a source sequence of $X\in\CC$ and a sink sequence of $Y\in\CC$.
\end{itemize}
\end{definition}

A source sequence \eqref{X source} corresponds to a minimal projective
resolution \eqref{X source2} of a functor $J_{\CC}(X,-)$ on $\CC$.
Thus any indecomposable object $X\in\CC$ has a unique source sequence
up to isomorphisms of complexes if it exists. 

Let us recall basic results for $n$-cluster tilting subcategories.

\begin{theorem}\label{basic}
Let $\Lambda$ be a finite dimensional algebra and $n\ge1$. Let $T$ be
a tilting $\Lambda$-module with $\pd T_\Lambda\le n$.
\begin{itemize}
\item[(a)] Let $\CC$ be an $n$-cluster tilting subcategory of $T^{\perp}$.
\begin{itemize}
\item[(i)] Any indecomposable object $X\in\CC\backslash\add D\Lambda$ (respectively, $Y\in\CC\backslash\add T$) has an $n$-almost split sequence
\[0\to X\stackrel{}{\to}C_1\stackrel{}{\to}C_2\stackrel{}{\to}\cdots\stackrel{}{\to}C_{n}\stackrel{}{\to}Y\to0\]
such that $Y\simeq\tau_n^-X$ and $X\simeq\tau_nY$.
\item[(ii)] Any indecomposable object $X\in\add D\Lambda$ has a source sequence of the form
\[X\stackrel{}{\to}C_1\to\cdots\to C_{n}\to0.\]
\item[(iii)] Any indecomposable object $X\in\add T$ has a sink sequence of the form
\[0\to C_n\to\cdots\to C_1\stackrel{}{\to}X.\]
\end{itemize}
\item[(b)] Let $\CC=\add M$ be an $n$-rigid subcategory of $T^{\perp}$ satisfying $T\oplus D\Lambda\in\CC$. Then the following conditions are equivalent.
\begin{itemize}
\item[(i)] $\CC$ is an $n$-cluster tilting subcategory of $T^{\perp}$.
\item[(ii)] $\Gamma:=\End_\Lambda(M)$ satisfies $\gl\Gamma\le n+1$.
\item[(iii)] Any indecomposable object $X\in\CC$ has a source sequence of the form
\[X\stackrel{}{\to}C_1\to\cdots\to C_{n+1}\to0.\]
\end{itemize}
\end{itemize}
\end{theorem}

\begin{proof}
In \cite{I4}, $n$-cluster tilting subcategories of the category
\[{}^{\perp} U=\{X\in\mod\Lambda\ |\ \Ext^i_\Lambda(X,U)=0\ (0<i)\}.\]
for a cotilting $\Lambda$-module $U$ is treated instead of $T^{\perp}$.
We can apply results in \cite{I4} since $DT$ is a cotilting $\Lambda^{\op}$-module with $\id{}_{\Lambda}(DT)\le n+1$,
and $\CC$ is an $n$-cluster tilting subcategory of $T^{\perp}$ if and only if $D\CC$ is an $n$-cluster tilting subcategory of ${}^{\perp}(DT)$.

(a)(i) This is shown in \cite[Th. 2.5.3]{I4}.

(ii) It is easily shown (cf. \cite[Prop. 2.4.1(2-$\ell$)]{I4}) that there exists an exact sequence $0\to X/\soc X\to C_1\to\cdots\to C_n\to0$
such that $C_i\in\CC$ and
\[0\to\Hom_\Lambda(C_n,-)\to\cdots\to\Hom_\Lambda(C_1,-)\to\Hom_\Lambda(X/\soc X,-)\to0\]
is exact on $\CC$.
Connecting with the natural surjection $X\to X/\soc X$, we have a source sequence $X\to C_1\to\cdots\to C_n\to0$ of the desired form.

(iii) Let $\Gamma:=\End_\Lambda(T)$.
Then $DT$ is a cotilting $\Gamma$-module, and Tilting theorem \cite{H,M} gives an equivalence
\[\F=\Hom_\Lambda(T,-):T_\Lambda^\perp\to{}^\perp(DT)_\Gamma\]
which preserves $\Ext$-groups.
Thus we have an $n$-cluster tilting subcategory $\F\CC$ of ${}^\perp(DT)_\Gamma$.
For any indecomposable object $X\in\add T_\Lambda$, there exists a sink sequence
$0\to C_n\to\cdots\to C_0\to\F X$ of the indecomposable projective
$\Gamma$-module $\F X$ with $C_i\in\F\CC$ by the dual of (ii).
Applying the quasi-inverse of $\F$, we have the desired sink sequence of $X$.

(b)(i)$\Leftrightarrow$(ii) 
Apply \cite[Th. 5.1(3)]{I4} for $d:=0$ and $m:=\id{}_{\Lambda}(DT)\le n$.


(ii)$\Leftrightarrow$(iii) $C_{n+2}=0$ holds in the source sequence
\eqref{X source} if and only if the simple $\Gamma^{\op}$-module
$\top\Hom_\Lambda(X,M)$ has projective dimension at most $n+1$.
Since $\gl\Gamma\le n+1$ if and only if any simple
$\Gamma^{\op}$-module has projective dimension at most $n+1$, we have
the assertion.
\end{proof}

%
%

Put
\begin{eqnarray*}
\GG_n&:=&\{X\in\mod\Lambda\ |\ \Ext^i_\Lambda(X,\Lambda)=0\ (0\le i<n)\},\\
\HH_n&:=&\{X\in\mod\Lambda\ |\ \Ext^i_\Lambda(D\Lambda,X)=0\ (0\le i<n)\}.
\end{eqnarray*}

\begin{lemma}\label{XY}
Let $\Lambda$ be a finite dimensional algebra with $\gl\Lambda\le n$ and $X\in\mod\Lambda$.
\begin{itemize}
\item[(a)] We have mutually quasi-inverse equivalences
\[\tau_n=D\Ext^n_\Lambda(-,\Lambda):\GG_n\to\HH_n\ \mbox{ and }\ \tau_n^-=\Ext^n_{\Lambda^{\op}}(D-,\Lambda):\HH_n\to\GG_n.\]
\item[(b)] If $X$ has no non-zero projective summands and $\Ext^i_\Lambda(X,\Lambda)=0$ for any $0<i<n$, then $X\in\GG_n$.
\item[(c)] If $X$ has no non-zero injective summands and $\Ext^i_\Lambda(D\Lambda,X)=0$ for any $0<i<n$, then $X\in\HH_n$.
\end{itemize}
\end{lemma}

\begin{proof}
Although the assertions are elementary, we provide a proof for the convenience of the reader.

(a) For $X\in\GG_n$, take a projective resolution
\begin{equation}\label{pr}
0\to P_n\to\cdots\to P_0\to X\to0.
\end{equation}
Applying $\nu$, we have an exact sequence
\begin{equation}\label{ir}
0\to\tau_nX\to\nu P_n\to\cdots\to\nu P_0\to0
\end{equation}
where we use $\Ext^i_\Lambda(X,\Lambda)=0$ for any $0\le i<n$.
The sequence (\ref{ir}) gives an injective resolution of $\tau_nX$.
Applying $\nu^-$ to (\ref{ir}), we have a complex
\[0\to P_n\to\cdots\to P_0\to0\]
whose homology at $P_i$ is $\Ext^{n-i}_\Lambda(D\Lambda,\tau_nX)$.
Comparing with (\ref{pr}), we have $\tau_nX\in\HH_n$ and $\tau_n^-\tau_nX\simeq X$.

(b) Take a projective resolution
\[0\to P_n\to\cdots\to P_1\stackrel{f_1}{\to} P_0\stackrel{f_0}{\to}X\to0.\]
Applying $\Hom_\Lambda(-,\Lambda)$, we have an exact sequence
\[0\to\Hom_\Lambda(X,\Lambda)\stackrel{f_0}{\to}\Hom_\Lambda(P_0,\Lambda)\stackrel{f_1}{\to}\Hom_\Lambda(P_1,\Lambda)\to\cdots\to\Hom_\Lambda(P_n,\Lambda)\to\Ext^n_\Lambda(X,\Lambda)\to0.\]
Since $X$ has no non-zero projective summand, $f_1$ is left minimal.
Since $\Hom_\Lambda(-,\Lambda):\add\Lambda_\Lambda\to\add{}_{\Lambda}\Lambda$ is an equivalence,
$f_1:\Hom_\Lambda(P_0,\Lambda)\to\Hom_\Lambda(P_1,\Lambda)$ is right minimal.
Since $f_0:\Hom_\Lambda(X,\Lambda)\to\Hom_\Lambda(P_0,\Lambda)$ is a split monomorphism by $\gl\Lambda\le n$, we have $\Hom_\Lambda(X,\Lambda)=0$ and $X\in\GG_n$.

(c) This is dual of (b).
\end{proof}

\begin{lemma}\label{tau_n on M}
Let $\Lambda$ be a finite dimensional algebra with $\gl\Lambda\le n$
and $\MM$ the $\tau_n$-closure of $D\Lambda$ satisfying the condition
(C$_n$) in Definition \ref{complete}.
Then we have the following.
\begin{itemize}
\item[(a)] We have full functors
\[\tau_n=D\Ext^n_\Lambda(-,\Lambda):\MM\to\MM\ \mbox{ and }\ \tau_n^-=\Ext^n_{\Lambda^{\op}}(D-,\Lambda):\MM\to\MM\]
which give mutually quasi-inverse equivalences
\[\tau_n:\MM_P\to\MM_I\ \mbox{ and }\ \tau_n^-:\MM_I\to\MM_P.\]
\item[(b)] $\tau_n$ gives a bijection from isoclasses of indecomposable objects in $\MM_P$ to those in $\MM_I$.
\item[(c)] $\MM_P\subset\GG_n$ and $\MM_I\subset\HH_n$.
\item[(d)] $\Hom_\Lambda(\MM_P,\PP(\MM))=0$.
\item[(e)] $\Hom_\Lambda(\tau_n^i(D\Lambda),\tau_n^j(D\Lambda))=0$ for any $i<j$.
\end{itemize}
\end{lemma}

\begin{proof}
(a)(c) We have $\MM_P\subset\GG_n$ by the condition (C$_n$) and Lemma \ref{XY}(b).
By Lemma \ref{XY}(a), we have a fully faithful functor $\tau_n:\MM_P\to\HH_n$.
By our construction of $\MM$, we have $\tau_n(\MM_P)=\MM_I$.
Thus we have an equivalence $\tau_n:\MM_P\to\MM_I$ with quasi-inverse $\tau_n^-$.
Since $\tau_n(\PP(\MM))=0$ and $\tau_n^-(\II(\MM))=0$, we have full functors
$\tau_n:\MM\to\MM$ and $\tau_n^-:\MM\to\MM$.

(b) Immediate from (a).
%

(d) For any $X\in\PP(\MM)$, take a projective resolution
\[0\to P_{n-1}\to\cdots\to P_0\to X\to0.\]
Applying $\Hom_\Lambda(\MM_P,-)$ and using $\Ext^i_\Lambda(\MM_P,\Lambda)=0$ for any $0\le i<n$, we have $\Hom_\Lambda(\MM_P,X)=0$.

(e) We have $\MM_I\subset\HH_{n}$ by (c), so $\Hom_\Lambda(D\Lambda,\tau_{n}^{j-i}(D\Lambda))=0$.
Since $\tau_{n}:\MM\to\MM_I$ is a full functor by (a), we have $\Hom_\Lambda(\tau_{n}^i(D\Lambda),\tau_{n}^j(D\Lambda))=0$.
\end{proof}

Now we shall prove Proposition \ref{weak is weak}.

By Lemma \ref{tau_n on M}, we have (a)--(c) Immediately.
The assertion (d) follows immediately from (c).
\qed

\medskip
We give the following general property of $\tau_n$-closures.

\begin{proposition}\label{n-rigid}
Let $\Lambda$ be a finite dimensional algebra and $\MM$ the $\tau_n$-closure of $D\Lambda$. 
\begin{itemize}
\item[(a)] If $\Ext^i_\Lambda(\MM_P,\Lambda)=0$ holds for any $1<i<n$, then $\MM$ is $n$-rigid.
\item[(b)] If $n=2$, then $\MM$ is $2$-rigid.
\end{itemize}
\end{proposition}

For the proof, we need the following general observation.

\begin{lemma}\label{surjection}
Let $\Lambda$ be a finite dimensional algebra, $X,Y\in\mod\Lambda$ and
$0<i<n$. If $\Ext^j_\Lambda(X,\Lambda)=0$ for any $n-i<j<n$, then
we have a surjection $\Ext^{n-i}_\Lambda(X,Y)\to D\Ext^i_\Lambda(Y,\tau_nX)$.
\end{lemma}

\begin{proof}
By Auslander-Reiten duality, we have
\[D\Ext^i_\Lambda(Y,\tau_nX)\simeq D\Ext^1_\Lambda(\Omega^{i-1}Y,\tau\Omega^{n-1}X)\simeq\ul{\Hom}_\Lambda(\Omega^{n-1}X,\Omega^{i-1}Y).\]
Since we assumed $\Ext^j_\Lambda(X,\Lambda)=0$ for $n-i<j<n$,
it is easily checked (e.g. \cite[7.4]{A2}) that
\[\ul{\Hom}_\Lambda(\Omega^{n-1}X,\Omega^{i-1}Y)\simeq\ul{\Hom}_\Lambda(\Omega^{n-i}X,Y).\]
Since in general we have a surjection
\[\Ext^{n-i}_\Lambda(X,Y)\to\ul{\Hom}_\Lambda(\Omega^{n-i}X,Y),\]
we have the desired surjection.
\end{proof}

Now we shall prove Proposition \ref{n-rigid}.
We only have to show (a).

(i) Let $X,Y\in\MM$. We shall show that, if $\Ext^i_\Lambda(X,Y)=0$ for any $0<i<n$,
then $\Ext^i_\Lambda(Y,\tau_nX)=0$ and $\Ext^i_\Lambda(\tau_nX,\tau_nY)=0$ for any $0<i<n$.

We can assume that $X\in\MM_P$.
Then we have $\Ext^j_\Lambda(X,\Lambda)=0$ for any $1<j<n$ by our assumption.
Thus we have $\Ext^i_\Lambda(Y,\tau_nX)=0$ for any $0<i<n$ by Lemma
\ref{surjection}. Replacing $(X,Y)$ by $(Y,\tau_nX)$, we have
$\Ext^i_\Lambda(\tau_nX,\tau_nY)=0$ for any $0<i<n$.

(ii) Let $0\le j,k$.
Since $\Ext^i_\Lambda(\tau_n^j(D\Lambda),D\Lambda)=0$ for any $0<i<n$,
we have $\Ext^i_\Lambda(\tau_n^{j+k}(D\Lambda),\tau_n^k(D\Lambda))=0$ and $\Ext^i_\Lambda(\tau_n^k(D\Lambda),\tau_n^{j+k+1}(D\Lambda))=0$
for any $0<i<n$ by (i).
Thus we have $\Ext^i_\Lambda(\MM,\MM)=0$ for any $0<i<n$.
\qed

\medskip
Now we shall prove Proposition \ref{n-AR translation new}.
It follows from Proposition \ref{n-rigid}(b) and the following result.

\begin{lemma}
Assume $\gl\Lambda\le n$. Then there exists an $n$-cluster tilting object
in $\mod\Lambda$ if and only if $\MM$ is $n$-rigid and $\Lambda\in\MM$.
\end{lemma}

\begin{proof}
`only if' part is clear from Theorem \ref{unique n-cluster}.
We shall show `if' part.

(i) Since $\Lambda\in\MM$, we have that $\MM$ satisfies the condition
(C$_n$) in Definition \ref{complete}.
Thus $\tau_n:\MM_P\to\MM_I$ is an equivalence by Lemma \ref{tau_n on M}(a).
A bijection from isoclasses of indecomposable
objects in $\II(\MM)$ to those in $\PP(\MM)$
is given by $I\mapsto\tau_n^{\ell_I}I$, where $\ell_I$ is a maximal
number $\ell$ satisfying $\tau_n^\ell I\neq0$.
In particular, $\tau_n^\ell(D\Lambda)=0$ holds for sufficiently
large $\ell$, and $\MM$ has an additive generator $M$.
Moreover, we have that $\tau_n^\ell X=0$ holds for any
$X\in\mod\Lambda$ by a similar argument as in the proof of Proposition
\ref{2 kinds}(c).

(ii) We shall show that $\Ext^i_\Lambda(X,\MM)=0$ for any $0<i<n$ implies $X\in\MM$.

We know $\Ext^i_\Lambda(\tau_nX,\tau_n\MM)=0$ for any $0<i<n$ by Lemma
\ref{surjection}. By our construction of $\MM$, we have
$\Ext^i_\Lambda(\tau_nX,\MM)=0$ for any $0<i<n$.
Consequently we have $\Ext^i_\Lambda(\tau_n^\ell X,\MM)=0$ for
any $\ell\ge0$ and $0<i<n$.
By (i), we can take a maximal number $\ell\ge0$ satisfying
$Y:=\tau_n^\ell X\neq0$. Since $\tau_nY=0$, we have $\pd Y_\Lambda<n$.
Since we have $\Ext^i_\Lambda(Y,\Lambda)=0$ for any $0<i<n$
by $\Lambda\in\MM$, we have that $Y$ is a projective $\Lambda$-module.
Thus we have $Y\in\MM$.
By Lemma \ref{tau_n on M}(a) again, we have $X\simeq\tau_n^{-\ell}Y\in\MM$.

(iii) By a dual argument, we have that $\Ext^i_\Lambda(\MM,X)=0$ for any
$0<i<n$ implies $X\in\MM$. Consequently $M$ given in (i) is an $n$-cluster
tilting object of $\mod\Lambda$.
\end{proof}

\section{Proof of Theorems \ref{main}\ and \ref{recall}}
\label{Proof of Main Theorem}

Throughout this section we assume that $\Lambda$ is $n$-complete with the $\tau_n$-closure
$\MM=\add M$ of $D\Lambda$ and $\PP(\MM)=\add T$ unless stated otherwise.
We put
\[\Gamma:=\End_\Lambda(M).\]

We have the following result immediately from Theorem \ref{Auslander correspondence}(a)$\Rightarrow$(b).

\begin{lemma}\label{tau and ext 2}
$\gl\Gamma\le n+1$.
\end{lemma}

This allows us to put
\begin{eqnarray*}
\tau_{n+1}&:=&D\Ext^{n+1}_\Gamma(-,\Gamma):\mod\Gamma\to\mod\Gamma,\\
\tau_{n+1}^-&:=&\Ext^{n+1}_{\Gamma^{\op}}(D-,\Gamma):\mod\Gamma\to\mod\Gamma.
\end{eqnarray*}
Clearly $\tau_{n+1}$ (respectively, $\tau_{n+1}^-$) preserves monomorphisms (respectively, epimorphisms) in $\mod\Gamma$. Moreover
$X\in\mod\Gamma$ satisfies $\tau_{n+1}X=0$ (respectively, $\tau_{n+1}^-X=0$)
if and only if $\pd X_\Gamma<n+1$ (respectively, $\id X_\Gamma<n+1$).
We denote by
\[\NN:=\MM_{n+1}(D\Gamma)=\add\{\tau_{n+1}^i(D\Gamma)\ |\ i\ge0\}\]
the $\tau_{n+1}$-closure of $D\Gamma$.

\subsection{$(n+1)$-rigidity of $\NN$}\label{section: (n+1)-rigidity}

The aim of this subsection is to show that $\NN$ is $(n+1)$-rigid.
As usual, we put
\begin{eqnarray*}
\II(\NN)&:=&\add D\Gamma,\\
\PP(\NN)&:=&\{X\in\NN\ |\ \pd X_\Gamma<n+1\}=\{X\in\NN\ |\ \tau_{n+1}X=0\},\\
\NN_I&:=&\{X\in\NN\ |\ X\ \mbox{ has no non-zero summands in }\ \II(\NN)\},\\
\NN_P&:=&\{X\in\NN\ |\ X\ \mbox{ has no non-zero summands in }\ \PP(\NN)\}.
\end{eqnarray*}
For $X\in\MM$, put
\[\ell_X:=\sup\{i\ge0\ |\ \tau_n^iX\neq0\}.\]
Let us start with the following easy observation.

\begin{proposition}\label{indecomposable in M}
\begin{itemize}
\item[(a)] $\ell_X<\infty$ for any $X\in\MM$.
\item[(b)] A bijection from isoclasses of indecomposable objects in
$\MM$ to pairs $(I,i)$ of isoclasses of indecomposable objects
$I\in\II(\MM)$ and $0\le i\le\ell_I$ is given by $(I,i)\mapsto\tau_n^iI$.
\end{itemize}
\end{proposition}

\begin{proof}
(b) By Lemma \ref{tau_n on M}(b) and the definition of $\MM$, any indecomposable object in $\MM$
can be written uniquely as $\tau_n^iI$ for indecomposable $I\in\II(\MM)$ and $0\le i\le\ell_I$.

(a) Since $\MM$ contains only finitely many isoclasses of indecomposable objects, we have $\ell_I<\infty$ for any indecomposable $I\in\II(\MM)$.
Since $\ell_{\tau_n^iI}=\ell_I-i$, we have the assertion.
\end{proof}

We put
\begin{eqnarray*}
\GG_{n+1}&:=&\{X\in\mod\Gamma\ |\ \Ext^i_\Gamma(X,\Gamma)=0\ (0\le i<n+1)\},\\
\HH_{n+1}&:=&\{X\in\mod\Gamma\ |\ \Ext^i_\Gamma(D\Gamma,X)=0\ (0\le i<n+1)\}.
\end{eqnarray*}
We need the following observation.

\begin{lemma}\label{XY2}
\begin{itemize}
\item[(a)] We have functors
\[\tau_{n+1}=D\Ext^{n+1}_\Gamma(-,\Gamma):\mod\Gamma\to\HH_{n+1}\ \mbox{ and }\ \tau_{n+1}^-=\Ext^{n+1}_{\Gamma^{\op}}(D-,\Gamma):\mod\Gamma\to\GG_{n+1}\]
which give mutually quasi-inverse equivalences
\[\tau_{n+1}:\GG_{n+1}\to\HH_{n+1}\ \mbox{ and }\ \tau_{n+1}^-:\HH_{n+1}\to\GG_{n+1}.\]
\item[(b)] $\GG_{n+1}$ and $\HH_{n+1}$ are Serre subcategories of $\mod\Gamma$.
\end{itemize}
\end{lemma}

\begin{proof}
We use the properties Theorem \ref{Auslander correspondence}(b)(ii) and (iii).

(a) We have the desired functors since $\tau_{n+1}X\in\HH_{n+1}$ and $\tau_{n+1}^-X\in\GG_{n+1}$ for any $X\in\mod\Gamma$ by \cite[6.1]{I1}.
They give equivalences between $\GG_{n+1}$ and $\HH_{n+1}$ by Lemma \ref{XY}.

(b) This follows from \cite[Prop. 2.4]{I2}.
\end{proof}

Define functors
\begin{eqnarray*}
\F^i&:=&\Ext_\Lambda^i(M,-):\mod\Lambda\to\mod\Gamma,\\
\G^i&:=&D\Ext_\Lambda^i(-,M):\mod\Lambda\to\mod\Gamma.
\end{eqnarray*}
Put $\F:=\F^0$ and $\G:=\G^0$ for simplicity.
They induce equivalences
\[\F:\MM\to\add\Gamma\ \mbox{ and }\ \G:\MM\to\add D\Gamma\subset\NN.\]
A crucial role is played by a monomorphism
\[\alpha:\tau_{n+1}\G\to\G\tau_n\]
of functors on $\MM$ given by the following proposition.

\begin{proposition}\label{functorial}
\begin{itemize}
\item[(a)] There exists an isomorphism $\F\nu_\Lambda\simeq\G$ of functors on $\add\Lambda$,
\item[(b)] There exists an isomorphism $\nu_\Gamma\F\simeq\G$ of functors on $\MM$,
\item[(c)] There exist isomorphisms $\F\tau_n\simeq\G^n$ and $\G\tau_n^-\simeq\F^n$ of functors on $\MM$,
\item[(d)] There exists a monomorphism $\alpha:\tau_{n+1}\G\to\G\tau_n$ of functors on $\MM$.
\item[(e)] $\alpha_X:\tau_{n+1}\G X\to\G\tau_nX$ is a minimal right $\HH_{n+1}$-approximation for any $X\in\MM$.
\item[(f)] We have a functorial monomorphism
\[\alpha^\ell:\tau_{n+1}^\ell\G\xrightarrow{\tau_{n+1}^{\ell-1}\alpha}\tau_{n+1}^{\ell-1}\G\tau_n\xrightarrow{\tau_{n+1}^{\ell-2}\alpha_{\tau_n}}
\tau_{n+1}^{\ell-2}\G\tau_n^2\to\cdots\to
\tau_{n+1}\G\tau_n^{\ell-1}\xrightarrow{\alpha_{\tau_n^{\ell-1}}}\G\tau_n^\ell.\]
\end{itemize}
\end{proposition}

\begin{proof}
(a)(b) Immediate.

(c) We only show $\F\tau_n\simeq\G^n$.
Since $\gl\Lambda\le n$, any $X\in\MM$ has a projective resolution
\begin{equation}\label{p1}
0\to P_n\to\cdots\to P_0\to X\to0.
\end{equation}
Applying $\G$, we have an exact sequence
\begin{equation}\label{p4}
0\to\G^nX\to\G P_n\to\cdots\to\G P_0\to\G X\to0,
\end{equation}
where we use $n$-rigidity of $\MM$.
On the other hand, applying $\nu_\Lambda$ to (\ref{p1}), we have an exact sequence
\[0\to\tau_nX\to\nu_\Lambda P_n\to\nu_\Lambda P_{n-1}.\]
Applying $\F$, we have an exact sequence
\begin{equation}\label{p3}
0\to\F\tau_nX\to\F\nu_\Lambda P_n\to\F\nu_\Lambda P_{n-1}.
\end{equation}
Comparing (\ref{p4}) and (\ref{p3}) by using (a), we have a commutative diagram
\[\begin{array}{ccccccc}
0\to&\G^nX&\to&\G P_n&\to&\G P_{n-1}&\to\cdots\to\G P_0\to\G X\to0,\\
&&&\wr |&&\wr |&\\
0\to&\F\tau_nX&\to&\F\nu_\Lambda P_n&\to&\F\nu_\Lambda P_{n-1}&
\end{array}\]
of exact sequences. Thus we have $\F\tau_nX\simeq\G^nX$.

(d) Since $\G P_i\simeq\F\nu_\Lambda P_i$ and $\gl\Lambda\le n$, the sequence (\ref{p4}) gives a projective resolution of a $\Gamma$-module $\G X$. Applying $\nu_\Gamma$, we have an exact sequence
\begin{equation}\label{tau_{n+1}GX}
0\to\tau_{n+1}\G X\to\nu_\Gamma\G^n X\to \nu_\Gamma\G P_n.
\end{equation}
Since $\nu_\Gamma\G^n\simeq\nu_\Gamma\F\tau_n\simeq\G\tau_n$ by (c) and (b) respectively, we have the assertion.

(e) We have $\tau_{n+1}\G X\in\HH_{n+1}$ by Lemma \ref{XY2}(a).
By \eqref{tau_{n+1}GX}, we have that $(\G\tau_nX)/(\tau_{n+1}\G X)$ is a submodule of $\nu_\Gamma\G P_n$.
By Lemma \ref{XY2}(b), we only have to show that $\soc\nu_\Gamma\G P_n$ does not belong to $\HH_{n+1}$. Since
\[\soc\nu_\Gamma\G P_n=\soc\nu_\Gamma\F\nu_\Lambda P_n=\soc\G\nu_\Lambda P_n\]
has injective dimension at most $n$ by Theorem \ref{basic}(a)(ii),
we have that $\soc\nu_\Gamma\G P_n$ does not belong to $\HH_{n+1}$.

(f) Since the functor $\tau_{n+1}$ preserves monomorphisms,
we have the assertion by (d).
\end{proof}

%
%

We need the following general observation, which is valid for arbitrary $\Lambda$ with $\gl\Lambda\le n$.

\begin{lemma}\label{hom and ext^n}
Let $n\ge1$ and $\Lambda$ a finite dimensional algebra with $\gl\Lambda\le n$.
Let $X\in\mod\Lambda$ and
\[0\to X_{0}\stackrel{f_{0}}{\longrightarrow}X_{1}\stackrel{f_{1}}{\longrightarrow}\cdots\stackrel{f_{n-1}}{\longrightarrow}X_n\stackrel{f_n}{\longrightarrow}X_{n+1}\to0\]
an exact sequence in $\mod\Lambda$ with $X_i\in\add X$.
\begin{itemize}
\item[(a)] If $W\in\mod\Lambda$ satisfies $\Ext^i_\Lambda(W,X)=0$ for any $0<i<n$, then we have an exact sequence
\begin{eqnarray*}
0\to\Hom_\Lambda(W,X_{0})\stackrel{f_{0}}{\longrightarrow}\cdots\stackrel{f_n}{\longrightarrow}\Hom_\Lambda(W,X_{n+1})
\to\Ext^n_\Lambda(W,X_{0})\stackrel{f_{0}}{\longrightarrow}\cdots\stackrel{f_n}{\longrightarrow}\Ext^n_\Lambda(W,X_{n+1})\to0.
\end{eqnarray*}
\item[(b)] If $Y\in\mod\Lambda$ satisfies $\Ext^i_\Lambda(X,Y)=0$ for any $0<i<n$, then we have an exact sequence
\begin{eqnarray*}
0\to\Hom_\Lambda(X_{n+1},Y)\stackrel{f_{n}}{\longrightarrow}\cdots\stackrel{f_{0}}{\longrightarrow}\Hom_\Lambda(X_{0},Y)
\to\Ext^n_\Lambda(X_{n+1},Y)\stackrel{f_{n}}{\longrightarrow}\cdots\stackrel{f_{0}}{\longrightarrow}\Ext^n_\Lambda(X_{0},Y)\to0.
\end{eqnarray*}
\end{itemize}
\end{lemma}

\begin{proof}
We only prove (a). We can assume $n>1$.
Put $L_i:=\Im f_{i-1}$. Then we have an exact sequence
\[0\to L_{i}\to X_i\to L_{i+1}\to0\ \ \ (1\le i\le n).\]
Applying $\Hom_\Lambda(W,-)$, we have exact sequences
\begin{eqnarray*}
&0\to\Hom_\Lambda(W,L_{i})\to\Hom_\Lambda(W,X_i)\to\Hom_\Lambda(W,L_{i+1})\to\Ext^1_\Lambda(W,L_{i})\to0,&\\
&0\to\Ext^{n-1}_\Lambda(W,L_{i+1})\to\Ext^n_\Lambda(W,L_{i})\to\Ext^n_\Lambda(W,X_i)\to\Ext^n_\Lambda(W,L_{i+1})\to0,&
\end{eqnarray*}
and an isomorphism
\[\Ext^j_\Lambda(W,L_{i+1})\simeq\Ext^{j+1}_\Lambda(W,L_{i})\ \ \ (0<j<n-1).\]
Since $L_{1}=X_0$ and $L_{n+1}=X_{n+1}$ belong to $\add X$, we have
\begin{eqnarray*}
&&\Ext^1_\Lambda(W,L_{i})\simeq\Ext^2_\Lambda(W,L_{i-1})\simeq\cdots\simeq\left\{\begin{array}{cc}\Ext^{i}_\Lambda(W,L_{1})=0&(1\le i<n)\\
\Ext^{n-1}_\Lambda(W,L_{2})&(i=n),
\end{array}\right.\\
&&\Ext^{n-1}_\Lambda(W,L_{i+1})\simeq\Ext^{n-2}_\Lambda(W,L_{i+2})\simeq\cdots\simeq\Ext^{i-1}_\Lambda(W,L_{n+1})=0\ \ \ (1<i\le n).
\end{eqnarray*}
Thus we have exact sequences
\begin{eqnarray*}
0\to\Hom_\Lambda(W,L_{i})\to\Hom_\Lambda(W,X_i)&\to&\Hom_\Lambda(W,L_{i+1})\to0\ \ \ (1\le i<n),\\
0\to\Ext^n_\Lambda(W,L_{i})\to\Ext^n_\Lambda(W,X_i)&\to&\Ext^n_\Lambda(W,L_{i+1})\to0\ \ \ (1<i\le n),\\
0\to\Hom_\Lambda(W,L_{n})\to\Hom_\Lambda(W,X_n)\to\Hom_\Lambda(W,L_{n+1})&\to&\Ext^n_\Lambda(W,L_{1})\\
&\to&\Ext^n_\Lambda(W,X_1)\to\Ext^n_\Lambda(W,L_2)\to0.
\end{eqnarray*}
Connecting them, we have the desired exact sequence.
\end{proof}

The following result is crucial to study properties of $\Gamma$.

\begin{lemma}\label{long sequence}
Let
\[0\to M_{0}\stackrel{f_{0}}{\longrightarrow} M_{1}\stackrel{f_{1}}{\longrightarrow}\cdots\stackrel{f_{n-1}}{\longrightarrow}M_n\stackrel{f_n}{\longrightarrow}M_{n+1}\to0\]
be an exact sequence in $\mod\Lambda$ with $M_i\in\MM$.
Put $X:=\Cok\F f_n$ and $Y:=\Ker\G f_{0}$.
\begin{itemize}
\item[(a)] We have exact sequences
\begin{eqnarray*}
&0\to \F M_{0}\stackrel{\F f_{0}}{\longrightarrow}\F M_{1}\stackrel{\F f_{1}}{\longrightarrow}\cdots\stackrel{\F f_{n-1}}{\longrightarrow}\F M_n\stackrel{\F f_n}{\longrightarrow}\F M_{n+1}\to X\to0,&\\
&0\to Y\to \G M_{0}\stackrel{\G f_{0}}{\longrightarrow}\G M_{1}\stackrel{\G f_{1}}{\longrightarrow}\cdots\stackrel{\G f_{n-1}}{\longrightarrow}\G M_n\stackrel{\G f_n}{\longrightarrow}\G M_{n+1}\to0.&
\end{eqnarray*}

\item[(b)] We have $X\in\GG_{n+1}$, $Y\in\HH_{n+1}$,
  $\tau_{n+1}X\simeq Y$ and $\tau_{n+1}^-Y\simeq X$.

\item[(c)] We have exact sequences
\begin{eqnarray*}
&0\to \G^nM_{0}\stackrel{\G^nf_{0}}{\longrightarrow} \G^nM_{1}\stackrel{\G^nf_{1}}{\longrightarrow}\cdots\stackrel{\G^nf_{n-1}}{\longrightarrow}\G^nM_n\stackrel{\G^nf_n}{\longrightarrow}\G^nM_{n+1}\to Y\to0,&\\
&0\to \F\tau_n M_{0}\stackrel{\F\tau_nf_{0}}{\longrightarrow}\F\tau_n M_{1}\stackrel{\F\tau_nf_{1}}{\longrightarrow}\cdots\stackrel{\F\tau_nf_{n-1}}{\longrightarrow}\F\tau_n M_n\stackrel{\F\tau_nf_n}{\longrightarrow}\F\tau_n M_{n+1}\to Y\to0.&
\end{eqnarray*}

\item[(d)] If $M_i\in\MM_P$ for any $i$, then we have an exact sequence
\[0\to\tau_n M_{0}\stackrel{\tau_nf_{0}}{\longrightarrow}\tau_n M_{1}\stackrel{\tau_nf_{1}}{\longrightarrow}\cdots\stackrel{\tau_nf_{n-1}}{\longrightarrow}\tau_n M_n\stackrel{\tau_nf_n}{\longrightarrow}\tau_n M_{n+1}\to0.\]

\item[(e)] We have exact sequences
\begin{eqnarray*}
&0\to X\to\F^nM_{0}\stackrel{\F^nf_{0}}{\longrightarrow}\F^nM_{1}\stackrel{\F^nf_{1}}{\longrightarrow}\cdots\stackrel{\F^nf_{n-1}}{\longrightarrow}\F^nM_n\stackrel{\F^nf_{n+1}}{\longrightarrow}\F^nM_{n+1}\to0,&\\
&0\to X\to\G\tau_n^-M_{0}\stackrel{\G\tau_n^-f_{0}}{\longrightarrow}\G\tau_n^-M_{1}\stackrel{\G\tau_n^-f_{1}}{\longrightarrow}\cdots\stackrel{\G\tau_n^-f_{n-1}}{\longrightarrow}\G\tau_n^-M_n\stackrel{\G\tau_n^-f_n}{\longrightarrow}\G\tau_n^-M_{n+1}\to0.&
\end{eqnarray*}

\item[(f)] If $M_i\in\MM_I$ for any $i$, then we have an exact sequence
\[0\to\tau_n^-M_{0}\stackrel{\tau_n^-f_{0}}{\longrightarrow}\tau_n^-M_{1}\stackrel{\tau_n^-f_{1}}{\longrightarrow}\cdots\stackrel{\tau_n^-f_{n-1}}{\longrightarrow}\tau_n^-M_n\stackrel{\tau_n^-f_n}{\longrightarrow}\tau_n^-M_{n+1}\to0.\]
\end{itemize}
\end{lemma}

\begin{proof}
(a) Since $M$ is $n$-rigid, we have desired exact sequences by Lemma \ref{hom and ext^n}.

(b) We apply $\nu_\Gamma$ to the upper sequence in (a) and compare with the lower sequence in (a) by using Proposition \ref{functorial}(b).
Then we have $X\in\GG_{n+1}$ and $\tau_{n+1}X\simeq Y$. We have
$Y\in\HH_{n+1}$ and $\tau_{n+1}^-Y\simeq X$ by Lemma \ref{XY}(a).

(c) The upper sequence is exact by Lemma \ref{hom and ext^n}(b).
Since $\F\tau_n\simeq\G^n$ holds by Proposition \ref{functorial}(c), the lower sequence is exact.

(d) We have $\Ext^i_\Lambda(\MM_P,\Lambda)=0$ for any $0\le i<n$ by Lemma \ref{tau_n on M}(c).
Thus the sequence is exact by Lemma \ref{hom and ext^n}(b).

(e)(f) These are shown dually.
\end{proof}

We need the following easy observation.

\begin{lemma}\label{induction2}
Let 
\[M_{0}\stackrel{f_{0}}{\longrightarrow}M_1\stackrel{f_1}{\longrightarrow}M_2\stackrel{f_2}{\longrightarrow}\cdots\]
be a complex with $M_i\in\MM$.
Assume $M_{0}\in\MM_P$ and that each $f_i$ is left minimal.
Then $M_i\in\MM_P$ for any $i$ and
we have a complex
\[\tau_nM_{0}\stackrel{\tau_nf_{0}}{\longrightarrow}\tau_nM_1\stackrel{\tau_nf_1}{\longrightarrow}\cdots\]
such that each $\tau_nf_i$ is left minimal.
\end{lemma}

\begin{proof}
We have $\Hom_\Lambda(\MM_P,\PP(\MM))=0$ by Lemma \ref{tau_n on M}(d).
Since any $f_i$ is left minimal, we have $M_i\in\MM_P$ inductively.
Since $\tau_n:\MM_P\to\MM_I$ is an equivalence by Lemma \ref{tau_n on M}(a),
each $\tau_nf_i$ is also left minimal.
\end{proof}

We shall use the following special case in inductive argument.

\begin{lemma}\label{induction}
Let
\[0\to M_{0}\stackrel{f_{0}}{\longrightarrow}M_1\stackrel{f_{1}}{\longrightarrow}\cdots\stackrel{f_{n-1}}{\longrightarrow}M_n\stackrel{f_n}{\longrightarrow}M_{n+1}\to0\]
be an exact sequence in $\mod\Lambda$ with $M_i\in\MM$.
Assume $M_{0}\in\MM_P$ and that each $f_i$ is left minimal.
Then $M_i\in\MM_P$ for any $i$
and we have an exact sequence
\[0\to\tau_n M_{0}\stackrel{\tau_nf_{0}}{\longrightarrow}\tau_n M_{1}\stackrel{\tau_nf_{1}}{\longrightarrow}\cdots\stackrel{\tau_nf_{n-1}}{\longrightarrow}\tau_n M_n\stackrel{\tau_nf_n}{\longrightarrow}\tau_n M_{n+1}\to0\]
such that each $\tau_nf_i$ is left minimal.
\end{lemma}

\begin{proof}
Immediate from Lemma \ref{induction2} and Lemma \ref{long sequence}(d).
\end{proof}

We are ready to show the following key observation.

\begin{proposition}\label{key}
Let $X\in\MM$ be indecomposable and $\ell\ge0$. Put $I:=\G X$.
\begin{itemize}
\item[(a)] If $\ell>\ell_X$, then $\tau_{n+1}^\ell I=0$.
\item[(b)] If $\ell=\ell_X$, then $\tau_{n+1}^\ell I$ is an indecomposable object in $\PP(\NN)$.
\item[(c)] If $0\le\ell<\ell_X$, then $\tau_{n+1}^\ell I\in\GG_{n+1}$ and $\pd(\tau_{n+1}^\ell I)_\Gamma=n+1$.
\end{itemize}
\end{proposition}

\begin{proof}
(a) Since $\tau_{n+1}^\ell I$ is a submodule of $\G\tau_n^\ell X=0$ by Proposition \ref{functorial}(f), we have $\tau_{n+1}^\ell I=0$.

(c) $\tau_n^\ell X\in\MM_P$ holds for any $0\le\ell<\ell_X$.
Take a minimal injective resolution
\begin{equation}\label{tau_nX no injective resolution}
0\to\tau_nX\stackrel{f_{-1}}{\to}I_0\stackrel{f_0}{\to}\cdots\stackrel{f_{n-1}}{\to}I_n\to0.
\end{equation}
Since each $f_i$ is left minimal, we have an exact sequence
\[0\to\tau^{\ell+1}_nX\stackrel{\tau_n^\ell f_{-1}}{\longrightarrow}\tau_n^\ell I_0\stackrel{\tau_n^\ell f_0}{\longrightarrow}\cdots\stackrel{\tau_n^\ell f_{n-1}}{\longrightarrow}\tau_n^\ell I_n\to0\]
for any $0\le\ell<\ell_X$ by applying Lemma \ref{induction} to \eqref{tau_nX no injective resolution} repeatedly.
Applying $\F$, we have a $\Gamma$-module $X_{\ell}$ with an exact sequence
\[0\to\F\tau^{\ell+1}_nX\stackrel{\F\tau_n^\ell f_{-1}}{\longrightarrow}\F\tau_n^\ell I_0\stackrel{\F\tau_n^\ell f_0}{\longrightarrow}\cdots\stackrel{\F\tau_n^\ell f_{n-1}}{\longrightarrow}\F\tau_n^\ell I_n\to X_{\ell}\to0\]
by Lemma \ref{long sequence}(a).
We have $X_0=\F^n\tau_nX\simeq\G X=I$ by Proposition \ref{functorial}(c).
Using Lemma \ref{long sequence}(b) and (c) repeatedly, we have
$\tau_{n+1}^\ell I\simeq\tau_{n+1}^\ell
X_0\simeq\tau_{n+1}^{\ell-1}X_1\simeq\cdots\simeq
X_{\ell}\in\GG_{n+1}$ for any $0\le\ell<\ell_X$.

(b) Since $\tau_{n+1}^{\ell_X+1}I=0$ by (a), we have $\tau_{n+1}^{\ell_X}I\in\PP(\NN)$.
We have that $\tau_{n+1}^{\ell_X}I$ is indecomposable by (c) and Lemma \ref{XY2}(a).
\end{proof}

Summarizing with Lemma \ref{tau and ext 2}, we have the following conclusion.

\begin{lemma}\label{C n+1}
$\Gamma$ is $\tau_{n+1}$-finite and satisfies the condition
(C$_{n+1}$) in Definition \ref{complete}.
\end{lemma}

We have the main results in this section.

\begin{theorem}\label{n+1-rigid}
\begin{itemize}
\item[(a)] $\NN$ is $(n+1)$-rigid.
\item[(b)] We have full functors
\[\tau_{n+1}=D\Ext^{n+1}_\Gamma(-,\Gamma):\NN\to\NN\ \mbox{ and }\ \tau_{n+1}^-=\Ext^{n+1}_{\Gamma^{\op}}(D-,\Gamma):\NN\to\NN\]
which give mutually quasi-inverse equivalences
\[\tau_{n+1}:\NN_P\to\NN_I\ \mbox{ and }\ \tau_{n+1}^-:\NN_I\to\NN_P.\]
\item[(c)] $\tau_{n+1}$ gives a bijection from isoclasses of indecomposable objects in $\NN_P$ to those in $\NN_I$.
\item[(d)] $\NN_P\subset\GG_{n+1}$ and $\NN_I\subset\HH_{n+1}$.
\item[(e)] $\Hom_\Gamma(\tau_{n+1}^i(D\Gamma),\tau_{n+1}^j(D\Gamma))=0$ for any $i<j$.
\item[(f)] $\Ext^i_\Gamma(\PP(\NN),\NN)=0$ for any $i>0$.
\end{itemize}
\end{theorem}

\begin{proof}
$\Gamma$ is $\tau_{n+1}$-finite and $\NN$ satisfies
the condition (C$_{n+1}$) by Lemmas \ref{tau and ext 2} and \ref{C n+1}.
Thus (a) follows from Lemma \ref{n-rigid}(a), and
(b)--(e) follow from Lemma \ref{tau_n on M}.
By (a) and $\pd X_\Gamma<n+1$ for any $X\in\PP(\NN)$, we have (f).
\end{proof}

We also have the following description of indecomposable objects in $\NN$.

\begin{corollary}\label{indecomposable}
\begin{itemize}
\item[(a)] There exist bijections among
\begin{itemize}
\item[$\bullet$] isoclasses of indecomposable objects in $\NN$,
\item[$\bullet$] pairs $(X,i)$ of isoclasses of indecomposable objects $X\in\MM$ and $0\le i\le \ell_X$,
\item[$\bullet$] triples $(I,i,j)$ of isoclasses of indecomposable objects $I\in\II(\MM)$ and $0\le i,j$ satisfying $i+j\le \ell_I$.
\end{itemize}
They are given by $\tau_{n+1}^j\G\tau_n^iI\leftrightarrow(\tau_n^iI,j)\leftrightarrow(I,i,j)$.
\item[(b)] Under the bijection in (a), $\tau_{n+1}^j\G\tau_n^iI$ belongs to $\PP(\NN)$ if and only if $i+j=\ell_I$.
\item[(c)] $\NN$ has an additive generator.
\end{itemize}
\end{corollary}

\begin{proof}
(a) By Theorem \ref{n+1-rigid}(c), any indecomposable object in $\NN$ can be written uniquely as $\tau_{n+1}^j\G X$ for an indecomposable object $X\in\MM$ and $0\le j$.
By Proposition \ref{key}, $\tau_{n+1}^j\G X$ is non-zero if and only if $0\le j\le\ell_X$.
By Proposition \ref{indecomposable in M}, $X$ can be written uniquely as $\tau_n^iI$ for an indecomposable object
$I\in\II(\MM)$ and $0\le i\le\ell_I$. Since $\ell_X=\ell_I-i$, we have the assertion.

(b)(c) Immediate from (a).
\end{proof}

\subsection{Tilting $\Gamma$-module in $\PP(\NN)$}

By Corollary \ref{indecomposable}, there exists a $\Gamma$-module $U$ such that $\PP(\NN)=\add U$.
In this subsection, we shall show that $U$ is a tilting $\Gamma$-module.

We have $\pd U_\Gamma<n+1$ by definition of $\PP(\NN)$.
By Theorem \ref{n+1-rigid}(f), we have $\Ext^i_\Gamma(U,U)=0$ for any $0<i$.
By Theorem \ref{n+1-rigid}(c), we have that $U$ and $\Gamma$ have the same number of non-isomorphic indecomposable direct summands.
For the case $n=1$, these imply that $U$ is a tilting $\Gamma$-module
\cite{ASS,H}. But we need more argument for arbitrary $n$.

\begin{lemma}\label{TU}
\begin{itemize}
\item[(a)] $\tau_{n+1}^\ell\G\tau_n^{-\ell}X\in\PP(\NN)$ for any $0\le\ell$ and $X\in\PP(\MM)$.
\item[(b)] $\tau_{n+1}^i\G\tau_n^{-\ell}X\in\NN_P$ for any $0\le i<\ell$ and $X\in\MM$.
\end{itemize}
\end{lemma}

\begin{proof}
(a) We can assume $\tau_n^{-\ell}X\neq0$.
Since $\ell_{\tau_n^{-\ell}X}=\ell$, the assertion follows from Corollary \ref{indecomposable}(b).

(b) Any indecomposable summand $Y$ of $\tau_n^{-\ell}X$ satisfies $\ell_Y\ge\ell$. Thus the assertion follows from Corollary \ref{indecomposable}(b).
\end{proof}

We shall often use the following result.

\begin{lemma}\label{grade exact}
Let
\[0\to X\stackrel{f_0}{\to}N_1\stackrel{f_1}{\to}N_2\stackrel{f_2}{\to}\cdots\stackrel{f_n}{\to}N_{n+1}\stackrel{f_{n+1}}{\to}N_{n+2}\to0\]
be an exact sequence in $\mod\Gamma$ with $N_i\in\NN_P$ and $X\in\GG_{n+1}$.
Then the sequence
\[0\to\tau_{n+1}X\stackrel{\tau_{n+1}f_0}{\longrightarrow}\tau_{n+1}N_1\stackrel{\tau_{n+1}f_1}{\longrightarrow}\tau_{n+1}N_2\stackrel{\tau_{n+1}f_2}{\longrightarrow}\cdots\stackrel{\tau_{n+1}f_{n}}{\longrightarrow}\tau_{n+1}N_{n+1}\stackrel{\tau_{n+1}f_{n+1}}{\longrightarrow}\tau_{n+1}N_{n+2}\to0\]
is exact.
\end{lemma}

\begin{proof}
We have $\Ext^i_\Gamma(\NN_P,\Gamma)=0$ for any $0\le i<n+1$ by Theorem \ref{n+1-rigid}(d).
We have the desired exact sequence by applying Lemma \ref{hom and ext^n}(b), where we replace $n$ there by $n+1$.
\end{proof}

\begin{lemma}\label{generate}
Let $I\in\II(\MM)$ be indecomposable.
\begin{itemize}
\item[(a)] There exist exact sequences
\begin{eqnarray*}
&0\to \nu_\Lambda^-I\stackrel{}{\to}T_0\stackrel{f_0}{\to}\cdots\stackrel{f_{n-2}}{\to}T_{n-1}\to0&\\
&0\to\F I\to\G T_0\stackrel{\G f_0}{\longrightarrow}\cdots\stackrel{\G f_{n-2}}{\longrightarrow}\G T_{n-1}\to0&
\end{eqnarray*}
with $T_i\in\PP(\MM)$.
\item[(b)] For any $0<\ell\le\ell_I$, there exist exact sequences
\begin{equation}\label{T-resolution}
0\to T_n\stackrel{f_n}{\to}\cdots\stackrel{f_1}{\to}T_0\stackrel{f_0}{\to}\tau_n^{\ell-1}I\to0
\end{equation}
\[0\to\F\tau_n^\ell I\to\tau_{n+1}^\ell\G\tau_n^{-\ell}T_n\stackrel{\tau_{n+1}^\ell\G\tau_n^{-\ell}f_n}{\longrightarrow}\cdots\stackrel{\tau_{n+1}^\ell\G\tau_n^{-\ell}f_1}{\longrightarrow}\tau_{n+1}^\ell\G\tau_n^{-\ell}T_0\to0\]
with $T_i\in\PP(\MM)$.
\end{itemize}
\end{lemma}

\begin{proof}
(a) Since $\nu_\Lambda^-I$ is a projective $\Lambda$-module and $T$ is a tilting $\Lambda$-module with $\pd T_\Lambda<n$, we have the upper sequence.
Applying $\G$, we have an exact sequence
\[0\to\G\nu_\Lambda^-I\stackrel{}{\longrightarrow}\G T_0\stackrel{\G f_0}{\longrightarrow}\cdots\stackrel{\G f_{n-2}}{\longrightarrow}\G T_{n-1}\to0\]
by Lemma \ref{hom and ext^n} since $\Ext^i_\Lambda(\Lambda\oplus T,M)=0$ for any $i>0$.
Since $\G\nu_\Lambda^-I=\F I$ by Proposition \ref{functorial}(a), we have the lower sequence.

(b) Since $\tau_n^{\ell-1}I\in\MM\subseteq T^{\perp}$,
there exists an exact sequence
\[\cdots\stackrel{f_2}{\to}T_1\stackrel{f_1}{\to}T_0\stackrel{f_0}{\to}\tau_n^{\ell-1}I\to0\]
with $T_i\in\PP(\MM)$ and $\Im f_i\in T^{\perp}$.
Since
\[\Ext^1_\Lambda(\Im f_{n},\Im f_{n+1})\simeq\Ext^2_\Lambda(\Im f_{n-1},\Im f_{n+1})\simeq\cdots\simeq\Ext^{n+1}_\Lambda(\Im f_0,\Im f_{n+1})=0\]
by $\gl\Lambda\le n$, we have that $\Im f_n\in\PP(\MM)$. Thus we have the sequence (\ref{T-resolution}).

Clearly we can assume that each $f_i$ is right minimal.
Applying $\tau_n^-$ to (\ref{T-resolution}) repeatedly, we have an exact sequence
\begin{equation}\label{tau_n to T-resolution}
0\to\tau_n^{-i}T_n\stackrel{\tau_n^{-i}f_n}{\longrightarrow}\cdots\stackrel{\tau_n^{-i}f_1}{\longrightarrow}\tau_n^{-i}T_0\stackrel{\tau_n^{-i}f_0}{\longrightarrow}\tau_n^{\ell-1-i}I\to0
\end{equation}
for any $0\le i\le\ell-1$ by the dual of Lemma \ref{induction} since $\tau_n^iI\in\MM_I$ for any $0<i$.

Applying $\F$ to (\ref{tau_n to T-resolution}), we have a $\Gamma$-module $X_i$ with an exact sequence
\[0\to\F\tau_n^{-i}T_n\stackrel{\F\tau_n^{-i}f_n}{\longrightarrow}\cdots\stackrel{\F\tau_n^{-i}f_1}{\longrightarrow}\F\tau_n^{-i}T_0\stackrel{\F\tau_n^{-i}f_0}{\longrightarrow}\F\tau_n^{\ell-1-i}I\to X_i\to0\]
for any $0\le i\le \ell-1$ by Lemma \ref{long sequence}(a), and for $i=-1$ by Lemma \ref{long sequence}(c).
Thus $X_{-1}\simeq\F\tau_n^\ell I$.
By Lemma \ref{long sequence}(b), we have $X_i\in\GG_{n+1}$ and
$\tau_{n+1}X_{i}\simeq X_{i-1}$ for any $0\le i\le \ell-1$. Especially we have
\begin{equation}\label{conditions on X_{l-1}}
\tau_{n+1}^iX_{\ell-1}\in\GG_{n+1}\ (0\le i\le\ell-1)\ \mbox{ and }\ 
\tau_{n+1}^\ell X_{\ell-1}\simeq X_{-1}\simeq\F\tau_n^\ell I.
\end{equation}

On the other hand, applying Lemma \ref{long sequence}(e) to the sequence (\ref{tau_n to T-resolution}) for $i=\ell-1$, we have an exact sequence
\begin{equation}\label{p10}
0\to X_{\ell-1}\to\G\tau_n^{-\ell}T_n\stackrel{\G\tau_n^{-\ell}f_n}{\longrightarrow}\cdots\stackrel{\G\tau_n^{-\ell}f_1}{\longrightarrow}\G\tau_n^{-\ell}T_0\to0.
\end{equation}
We have $\tau_{n+1}^i\G\tau_n^{-\ell}T\in\NN_P$ for any $0\le i\le\ell-1$ by Lemma \ref{TU}(b).
Applying $\tau_{n+1}$ to (\ref{p10}) repeatedly, we have an exact sequence
\[0\to\tau_{n+1}^iX_{\ell-1}\to\tau_{n+1}^i\G\tau_n^{-\ell}T_n\stackrel{\tau_{n+1}^i\G\tau_n^{-\ell}f_n}{\longrightarrow}\cdots\stackrel{\tau_{n+1}^i\G\tau_n^{-\ell}f_1}{\longrightarrow}\tau_{n+1}^i\G\tau_n^{-\ell}T_0\to0\]
for any $0\le i\le\ell$ by Lemma \ref{grade exact} since we have (\ref{conditions on X_{l-1}}).
Putting $i=\ell$, we have the desired sequence by (\ref{conditions on X_{l-1}}).
\end{proof}

Summarizing above results, we have the following desired result.

\begin{theorem}\label{U is tilting}
There exists a tilting $\Gamma$-module $U$ such that $\pd U_\Gamma<n+1$ and $\PP(\NN)=\add U$.
\end{theorem}

\begin{proof}
As we observed at the beginning of this subsection, we already have $\pd U_\Gamma<n+1$ and $\Ext^i_\Gamma(U,U)=0$ for any $i>0$.
By Lemma \ref{generate} and Lemma \ref{TU}(a),
there exists an exact sequence
\[0\to\Gamma\to U_0\to\cdots\to U_n\to0\]
with $U_i\in\add U$.
Thus $U$ is a tilting $\Gamma$-module.
\end{proof}

\subsection{Mapping cone construction of $(n+1)$-almost split sequences}\label{section: Mapping}

In this subsection, we complete our proof of Theorem \ref{main}.
Our method is to construct source sequences in $\NN$ and apply Theorem \ref{basic}(b)(iii)$\Rightarrow$(i).
A crucial role is played by a monomorhism
\[\alpha:\tau_{n+1}\G\to\G\tau_n\]
of functors on $\MM$ in Proposition \ref{functorial}(d).

%

\begin{lemma}\label{tau directed}
Fix an indecomposable object $X\in\MM$ and $\ell\ge0$. Take a source morphism $f_0:X\to M_1$ in $\MM$.
\begin{itemize}
\item[(a)] Any morphism $\tau_{n+1}^\ell\G X\to\tau_{n+1}^i\G M$ with $i>\ell$ is zero.
\item[(b)] Any morphism $\tau_{n+1}^\ell\G X\to\tau_{n+1}^\ell\G M$ which is not a split monomorphism factors through $\tau_{n+1}^\ell\G f_0:\tau_{n+1}^\ell\G X\to\tau_{n+1}^\ell\G M_1$.
\item[(c)] Any morphism $\tau_{n+1}^\ell\G X\to\tau_{n+1}^i\G M$ with $0\le i<\ell$ factors through $\tau_{n+1}^{\ell-1}\alpha_X:\tau_{n+1}^\ell\G X\to\tau_{n+1}^{\ell-1}\G\tau_nX$.
\end{itemize}
\end{lemma}

\begin{proof}
(a) Immediate from Theorem \ref{n+1-rigid}(e).

(b) By Theorem \ref{n+1-rigid}(b), any morphism $\tau_{n+1}^\ell\G X\to\tau_{n+1}^\ell\G M$ which is not a split monomorphism 
can be written as $\tau_{n+1}^\ell\G g$ with a morphism $g:X\to M$ which is not a split monomorphism.
Since $g$ factors through $f_0$, we have that $\tau_{n+1}^\ell\G g$ factors through $\tau_{n+1}^\ell\G f_0$.

(c) By Theorem \ref{n+1-rigid}(b), any morphism $\tau_{n+1}^\ell\G X\to\tau_{n+1}^i\G M$ can be written as
$\tau_{n+1}^ig$ for a morphism $g:\tau_{n+1}^{\ell-i}\G X\to\G M$.
Since $\tau_{n+1}$ preserves monomorphisms,
we have that $\tau_{n+1}^{\ell-i-1}\alpha_X:\tau_{n+1}^{\ell-i}\G X\to \tau_{n+1}^{\ell-i-1}\G\tau_nX$ is a monomorphism.
Since $\G M$ is an injective $\Gamma$-module, $g$ factors through $\tau_{n+1}^{\ell-i-1}\alpha_X$.
Thus $\tau_{n+1}^ig$ factors through $\tau_{n+1}^\ell \alpha_X$.
\end{proof}

Immediately we have the following conclusion.

\begin{proposition}\label{source}
Fix an indecomposable object $X\in\MM$ and $0\le\ell\le\ell_X$. Take a source morphism $f_0:X\to M_1$ in $\MM$.
Then a left almost split morphism of $\tau_{n+1}^\ell\G X$ is given by
\[{\tau_{n+1}^\ell\G f_0\choose\tau_{n+1}^{\ell-1}\alpha_X}:\tau_{n+1}^\ell\G X\to(\tau_{n+1}^\ell\G M_1)\oplus(\tau_{n+1}^{\ell-1}\G\tau_nX).\]
\end{proposition}

Recall that any indecomposable object $X\in\MM$ has a source sequence by Theorem \ref{basic}(a).
By using it, we shall construct source sequences of an indecomposable object $\tau^\ell_{n+1}\G X$ in $\NN$ for $0\le\ell\le\ell_X$.

First we consider the case $\ell=0$.

\begin{proposition}\label{injective source}
For an indecomposable object $X\in\MM$, take a source sequence
\[X\stackrel{f_0}{\to}M_{1}\stackrel{f_1}{\to}\cdots\stackrel{f_{n-1}}{\to}M_{n}\stackrel{f_n}{\to}M_{n+1}\to0\]
in $\MM$. Applying $\G$, we have an exact sequence
\[\G X\stackrel{\G f_0}{\longrightarrow}\G M_{1}\stackrel{\G f_1}{\longrightarrow}\cdots\stackrel{\G f_{n-1}}{\longrightarrow}\G M_{n}\stackrel{\G f_n}{\longrightarrow}\G M_{n+1}\to0\]
which is a source sequence of $\G X$ in $\NN$.
\end{proposition}

\begin{proof}
Clearly the sequence is exact. Since each $f_i$ is left minimal and $\G$ is a fully faithful functor, each $\G f_i$ is left minimal.
Moreover $\G f_0$ is a source morphism in $\NN$ by Proposition \ref{source}.
\end{proof}

Next we consider the case $0<\ell\le\ell_X$.
Recall that any indecomposable object in $\MM_I$
is a left term of an $n$-almost split sequence by Theorem \ref{basic}(a).
Let us start with the following observation.

\begin{lemma}\label{tau correspondence}
For an indecomposable object $X\in\MM_P$, take the following $n$-almost split sequence.
\[0\to\tau_n X\stackrel{f_0}{\to}M_{1}\stackrel{f_1}{\to}\cdots\stackrel{f_{n-1}}{\to}M_{n}\stackrel{f_n}{\to}X\to0\]
\begin{itemize}
\item[(a)] We have the following source sequence in $\MM$.
\begin{equation}\label{sequence in (a)}
X\stackrel{\tau_n^-f_0}{\longrightarrow}\tau_n^-M_{1}\stackrel{\tau_n^-f_1}{\longrightarrow}\cdots\stackrel{\tau_n^-f_{n-1}}{\longrightarrow}\tau_n^-M_{n}\stackrel{\tau_n^-f_n}{\longrightarrow}\tau_n^-X\to0
\end{equation}
\item[(b)] $\soc\G X\in\GG_{n+1}$ and $\tau_{n+1}(\soc\G X)\simeq\soc(\G\tau_nX)$.
\end{itemize}
\end{lemma}

\begin{proof}
(a) Any term in (\ref{sequence in (a)}) belongs to $\MM_P$.
Since we have $\Hom_\Lambda(\MM_P,\PP(\MM))=0$ by Lemma \ref{tau_n on M}(d) and we have an equivalence $\tau_n^-:\MM_I\to\MM_P$, 
we have the assertion.

(b) Applying $\G$ to source sequences of $X$ and $\tau_nX$, we have exact sequences
\begin{eqnarray*}
&0\to\soc\G X\to\G X\stackrel{\G \tau_n^-f_0}{\longrightarrow}\G \tau_n^-M_{1}\stackrel{\G \tau_n^-f_1}{\longrightarrow}\cdots\stackrel{\G \tau_n^-f_{n-1}}{\longrightarrow}\G \tau_n^-M_{n}\stackrel{\G \tau_n^-f_n}{\longrightarrow}\G \tau_n^-X\to0,&\\
&0\to\soc(\G\tau_n X)\to\G\tau_n X\stackrel{\G f_0}{\longrightarrow}\G M_{1}\stackrel{\G f_1}{\longrightarrow}\cdots\stackrel{\G f_{n-1}}{\longrightarrow}\G M_{n}\stackrel{\G f_n}{\longrightarrow}\G X\to0.&
\end{eqnarray*}
Comparing these sequences by Lemma \ref{long sequence}(a)(e), we have the assertions.
\end{proof}

We also need the following information.

\begin{proposition}\label{simple correspondence}
For an indecomposable object $X\in\MM$ and $\ell\ge0$, let $S:=\tau_{n+1}^\ell(\soc\G X)$.
\begin{itemize}
\item[(a)] We have $S\simeq\soc(\G\tau_n^\ell X)$.
\item[(b)] If $0\le\ell<\ell_X$, then $S\in\GG_{n+1}$. If $\ell=\ell_X$,
then $S\notin\GG_{n+1}$. If $\ell>\ell_X$, then $S=0$.
\end{itemize}
\end{proposition}

\begin{proof}
(a) We only have to show the case $\ell=1$.
If $X\in\MM_P$, then the assertion follows from Lemma \ref{tau correspondence}(b).
Assume $X\in\PP(\MM)$. Since $\tau_nX=0$, the right hand side is zero.
Since $\soc\G X=\top\F X$ has projective dimension at most $n$ by Theorem \ref{basic}(a)(iii),
the left hand side is also zero.

(b) Immediate from (a) and Lemma \ref{tau correspondence}(b).
\end{proof}

For $X\in\MM$, define a morphism $\iota_X:\tau_n\tau_n^-X\to X$ by
taking an isomorphism $X\simeq Y\oplus I$ with $Y\in\MM_I$ and $I\in\II(\MM)$
and putting
\[\iota_X:\tau_n\tau_n^-X\simeq Y\xrightarrow{{1\choose0}}Y\oplus I\simeq X.\]
We denote by $\beta_X$ the morphism
\[\beta_X:\tau_{n+1}\G\tau_n^-X\stackrel{\alpha_{\tau_n^-X}}{\longrightarrow}\G\tau_n\tau_n^-X\stackrel{\G \iota_X}{\to}\G X.\]
Although $\iota_X$ depends on the choice of decomposition of $X$,
we have the following.

\begin{lemma}\label{beta}
$\beta_X$ is independent of the choice of decomposition of $X$.
In particular, $\beta$ gives a monomorphism $\beta:\tau_{n+1}\G\tau_n^-\to\G$ of functors on $\MM$.
\end{lemma}

\begin{proof}
By Theorem \ref{basic}(a)(ii), we have $\id(\soc\G I)_\Gamma\le n$.
Since we have $\tau_{n+1}\G\tau_n^-X\in\HH_{n+1}$ by Lemma \ref{XY2}(a)
and $\HH_{n+1}$ is a Serre subcategory of $\mod\Gamma$ by Lemma \ref{XY2}(b),
we have $\Hom_\Gamma(\tau_{n+1}\G\tau_n^-X,\G I)=0$.
This implies the assertion.
\end{proof}

\begin{proposition}\label{n+1-almost split}
For an indecomposable object $X\in\MM_P$ and $0<\ell\le\ell_X$, take the following $n$-almost split sequence in $\MM$.
\[0\to\tau_n X\stackrel{f_0}{\to}M_{1}\stackrel{f_1}{\to}\cdots\stackrel{f_{n-1}}{\to}M_{n}\stackrel{f_n}{\to}X\to0\]
\begin{itemize}
\item[(a)] We have the following commutative diagram of exact sequences.
\begin{eqnarray*}
&0\to\tau_{n+1}^\ell(\soc\G X)\to\tau_{n+1}^\ell\G X\xrightarrow{\tau_{n+1}^\ell\G\tau_n^- f_0}\tau_{n+1}^\ell\G\tau_n^{-}M_{1}\to\cdots
\xrightarrow{\tau_{n+1}^\ell\G\tau_n^- f_n}\tau_{n+1}^\ell\G\tau_n^{-}X\to0&\\
&\ \ \ \ \ \ \downarrow\wr\ \ \ \ \ \ \ \ \ \ \ \ \ \ \ \downarrow^{\tau_{n+1}^{\ell-1}\beta_{\tau_nX}}\ \ \ \ \ \ \ \ \ \ \ \ \ \ \ \ \ \ 
\downarrow^{\tau_{n+1}^{\ell-1}\beta_{M_1}}\ \ \ \ \ \ \ \ \ \ \ \ \ \ \ \ \ \ \ \ \downarrow^{\tau_{n+1}^{\ell-1}\beta_X}&\\
&0\to\tau_{n+1}^{\ell-1}(\soc\G\tau_n X)\ \to\ \tau_{n+1}^{\ell-1}\G\tau_n X\ \xrightarrow{\tau_{n+1}^{\ell-1}\G f_0}\ \tau_{n+1}^{\ell-1}\G M_{1}\to
\ \ \cdots\ \ \xrightarrow{\tau_{n+1}^{\ell-1}\G f_n}\ \tau_{n+1}^{\ell-1}\G X\to0&
\end{eqnarray*}
\item[(b)] Taking a mapping cone, we have the following $(n+1)$-almost split sequence in $\NN$.
\begin{eqnarray*}
0\to\tau_{n+1}^\ell\G X\xrightarrow{{\tau_{n+1}^\ell\G\tau_n^- f_0\choose\tau_{n+1}^{\ell-1}\beta_{\tau_nX}}}(\tau_{n+1}^\ell\G\tau_n^{-}M_1)\oplus(\tau_{n+1}^{\ell-1}\G\tau_nX)
\xrightarrow{{\tau_{n+1}^\ell\G\tau_n^- f_1  \ \ \ \ \ 0\ \ \ \ \choose\tau_{n+1}^{\ell-1}\beta_{M_1}\ \ \ -\tau_{n+1}^{\ell-1}\G f_0}}\cdots\\
\cdots\xrightarrow{{\tau_{n+1}^\ell\G\tau_n^- f_n\ \ \ \ \ 0\ \ \ \ \ \choose\tau_{n+1}^{\ell-1}\beta_{M_{n}}\ -\tau_{n+1}^{\ell-1}\G f_{n-1}}}(\tau_{n+1}^\ell\G\tau_n^{-}X)\oplus(\tau_{n+1}^{\ell-1}\G M_{n})
\xrightarrow{(\tau_{n+1}^{\ell-1}\beta_X\ -\tau_{n+1}^{\ell-1}\G f_n)}\tau_{n+1}^{\ell-1}\G X\to0
\end{eqnarray*}
\end{itemize}
\end{proposition}

\begin{proof}
(a) By Lemma \ref{tau correspondence}(a), we have an exact sequence
\begin{eqnarray}\label{first}
0\to\soc\G X\to\G X\stackrel{\G\tau_n^-f_0}{\longrightarrow}\G\tau_n^{-}M_{1}\stackrel{\G\tau_n^-f_1}{\longrightarrow}\cdots
\stackrel{\G\tau_n^-f_{n-1}}{\longrightarrow}\G\tau_n^{-}M_{n}\stackrel{\G\tau_n^-f_n}{\longrightarrow}\G\tau_n^{-}X\to0.
\end{eqnarray}
Since $\ell\le\ell_X$, we have $\tau_n^iX\in\MM_P$ for any $0\le i<\ell$.
By Proposition \ref{simple correspondence}, we have
\[\tau_{n+1}^i(\soc\G X)\simeq\soc(\G\tau_n^iX)\ \ \ (0\le i\le\ell)\ \mbox{ and }\ \tau_{n+1}^i(\soc\G X)\in\GG_{n+1}\ \ \ (0\le i<\ell).\]
Since $\tau_{n+1}^i\G X\in\NN_P$ for any $0\le i<\ell$ by Proposition \ref{key}(c), we have an exact sequence
\[0\to\tau_{n+1}^i(\soc\G X)\to\tau_{n+1}^i\G X\stackrel{\tau_{n+1}^i\G\tau_n^-f_0}{\longrightarrow}\tau_{n+1}^i\G\tau_n^{-}M_{1}\stackrel{\tau_{n+1}^i\G\tau_n^-f_1}{\longrightarrow}
\cdots\stackrel{\tau_{n+1}^i\G\tau_n^-f_n}{\longrightarrow}\tau_{n+1}^i\G\tau_n^{-}X\to0\]
for any $0\le i\le\ell$ by applying Lemma \ref{induction2}
(replace $\MM$ there by $\NN$) and Lemma \ref{grade exact}
to the sequence (\ref{first}) repeatedly.

By a similar argument, we have an exact sequence 
\[0\to\tau_{n+1}^i(\soc\G\tau_nX)\to\tau_{n+1}^i\G\tau_nX\stackrel{\tau_{n+1}^i\G f_0}{\longrightarrow}
\tau_{n+1}^i\G M_{1}\stackrel{\tau_{n+1}^i\G f_1}{\longrightarrow}\cdots\stackrel{\tau_{n+1}^i\G f_n}{\longrightarrow}\tau_{n+1}^i\G X\to0\]
for any $0\le i<\ell$.
Thus we have the desired exact sequences.

Using the morphism $\tau_{n+1}^{\ell-1}\beta:\tau_{n+1}^\ell\G\tau_n^-\to\tau_{n+1}^{\ell-1}\G$ of functors, we have the desired commutative diagram.
The left hand side morphism $\tau_{n+1}^\ell(\soc\G X)\to\tau_{n+1}^{\ell-1}(\soc\G\tau_nX)$ is an isomorphism since it is a monomorphism between simple $\Gamma$-modules.

(b) Clearly our sequence is exact. By Proposition \ref{source},
the morphism ${\tau_{n+1}^\ell\G\tau_n^- f_0\choose\tau_{n+1}^{\ell-1}\beta_{\tau_nX}}:\tau_{n+1}^\ell\G X\to(\tau_{n+1}^\ell\G\tau_n^{-}M_1)\oplus(\tau_{n+1}^{\ell-1}\G\tau_nX)$ is left almost split.
Since all morphisms in our sequence are contained in the Jacobson radical of the category $\MM$,
they are left minimal.
By Lemma \ref{hom and ext^n}(b) (replace $n$ there by $n+1$), our sequence gives a source sequence of $\tau_{n+1}^\ell\G X$.
Since source sequences are unique up to isomorphisms of complexes, it is an $(n+1)$-almost split sequence by Theorem \ref{basic}(a)(i).
\end{proof}

Consequently we have the following.

\begin{proposition}\label{source in N}
Any indecomposable object $X\in\NN$ has a source sequence in $\NN$ of the form
\[X\to N_1\to\cdots\to N_{n+2}\to0.\]
\end{proposition}

\begin{proof}
Immediate from Proposition \ref{injective source} and Proposition \ref{n+1-almost split}.
\end{proof}

We have the following conclusion.

\begin{theorem}\label{N is cluster tilting}
$N$ is an $(n+1)$-cluster tilting object in $U^{\perp}$.
\end{theorem}

\begin{proof}
This follows from Theorem \ref{basic}(b)(iii)$\Rightarrow$(i) and Proposition \ref{source in N}.
\end{proof}

Now we complete the proof of Theorem \ref{main}.
We have that $\Gamma$ is $\tau_{n+1}$-finite and satisfies (C$_{n+1}$)
by Lemma \ref{C n+1}, (A$_{n+1}$) by Theorem \ref{U is tilting},
and (B$_{n+1}$) by Theorem \ref{N is cluster tilting}.\qed

\subsection{Description of the cone of $\Gamma$}

The aim of this subsection is to give a description of the cone of $\Gamma$ and prove Theorem \ref{recall} as an application.
For simplicity, we assume that our additive generator $M$ of $\MM$ is basic,
so $\Gamma$ is also basic.

For $\ell\ge0$, we denote by
\[\simple^\ell\Gamma\ \mbox{ (respectively, $\simple_{\ell}\Gamma$)}\]
the set of simple $\Gamma$-modules $S$ satisfying $\tau_{n+1}^\ell S\neq0$ (respectively, $\tau_{n+1}^{-\ell}S\neq0$).
Define subcategories of $\mod\Gamma$ by
\begin{eqnarray*}
\GG_{n+1}^\ell&:=&\{X\in\mod\Gamma\ |\ \mbox{any composition factor $S$ of $X$ belongs to }\simple^\ell\Gamma\},\\
\HH_{n+1}^\ell&:=&\{X\in\mod\Gamma\ |\ \mbox{any composition factor $S$ of $X$ belongs to }\simple_{\ell}\Gamma\}.
\end{eqnarray*}
We have the following preliminary properties.

\begin{lemma}\label{XY l}
\begin{itemize}
\item[(a)] $\simple^\ell\Gamma$ (respectively, $\simple_\ell\Gamma$) consists of $\soc\G X$ for any indecomposable $X\in\MM$
satisfying $\tau_n^\ell X\neq0$ (respectively, $\tau_n^{-\ell}X\neq0$).
\item[(b)] We have $\GG_{n+1}^1=\GG_{n+1}$ and $\HH_{n+1}^1=\HH_{n+1}$.
\item[(c)] For any $\ell>0$, we have equivalences
\[\GG_{n+1}^\ell\stackrel{\tau_{n+1}}{\longrightarrow}\GG_{n+1}^{\ell-1}\cap\HH_{n+1}^1\stackrel{\tau_{n+1}}{\longrightarrow}\GG_{n+1}^{\ell-2}\cap\HH_{n+1}^2\stackrel{\tau_{n+1}}{\longrightarrow}
\cdots\stackrel{\tau_{n+1}}{\longrightarrow}\GG_{n+1}^1\cap\HH_{n+1}^{\ell-1}\stackrel{\tau_{n+1}}{\longrightarrow}\HH_{n+1}^\ell\]
whose quasi-inverses are given by $\tau_{n+1}^-$.
\end{itemize}
\end{lemma}

\begin{proof}
(a) Immediate from Proposition \ref{simple correspondence}(a).

(b) We only show $\GG_{n+1}^1=\GG_{n+1}$.
By Proposition \ref{simple correspondence}(b), we have that a simple $\Gamma$-module $S$ belongs to $\GG_{n+1}^1$ if and only if it belongs to $\GG_{n+1}$.
Since $\GG_{n+1}^1$ and $\GG_{n+1}$ are Serre subcategories by Lemma \ref{XY2}(b), we have the assertion.

(c) By (b) and Lemma \ref{XY2}(a), we have an equivalence $\tau_{n+1}:\GG_{n+1}^1\to\HH_{n+1}^1$. This gives the desired equivalences.
\end{proof}

There exist idempotents $e^\ell$ and $e_\ell$ of $\Gamma$ such that
the factor algebras $\Gamma^{\ell}:=\Gamma/\langle e^\ell\rangle$ and
$\Gamma_{\ell}:=\Gamma/\langle e_\ell\rangle$ satisfy
\[\GG_{n+1}^\ell=\mod\Gamma^{\ell}\ \mbox{ and }\ \HH_{n+1}^\ell=\mod\Gamma_{\ell}.\]
We have surjections
\[\Gamma=\Gamma^0\to\Gamma^1\to\Gamma^2\to\cdots\ \mbox{ and }\ \Gamma=\Gamma_0\to\Gamma_{1}\to\Gamma_{2}\to\cdots\]
of algebras. We have the following description of $e_\ell$.

\begin{lemma}\label{e ell}
$\add(e_\ell M)=\add\{\tau_n^i(D\Lambda)\ |\ 0\le i<\ell\}$ holds for any $\ell\ge0$.
\end{lemma}

\begin{proof}
For an indecomposable object $X\in\MM$, we put $S:=\soc\G X$.
By definition of $e_\ell$, we have that $X\in\add(e_\ell M)$ if and only if $S\notin\simple_\ell\Gamma$.
By Lemma \ref{XY l}(a), we have that $S\notin\simple_\ell\Gamma$ if and only if
$\tau_n^{-\ell}X=0$ if and only if $X\in\add\{\tau_n^i(D\Lambda)\ |\ 0\le i<\ell\}$.
\end{proof}

Now we need the functorial monomorphism $\alpha^\ell:\tau_{n+1}^\ell\G\to\G\tau_n^\ell$ given in Proposition \ref{functorial}(f).
The following result generalizes Proposition \ref{functorial}(e).

\begin{lemma}\label{right approximation}
\begin{itemize}
\item[(a)] $\alpha^\ell_X:\tau_{n+1}^\ell\G X\to\G\tau_n^\ell X$ is a minimal right $\HH_{n+1}^\ell$-approximation of $\G\tau_n^\ell X$
for any $X\in\MM$ and $\ell\ge0$.
\item[(b)] We have $\tau_{n+1}^\ell\G=\Hom_\Gamma(\Gamma_{\ell},\G\tau_n^\ell-)$.
\end{itemize}
\end{lemma}

\begin{proof}
(a) Fix $\ell>0$ and assume that the assertion is true for $\ell-1$.
By Lemma \ref{XY l}(c), any object in $\HH_{n+1}^\ell$ can be written as $\tau_{n+1}^\ell Y$ with $Y\in\GG_{n+1}^\ell$.
Take any morphism $f:\tau_{n+1}^\ell Y\to\G\tau_n^\ell X$.
We write $\alpha_X^\ell$ as a composition
\[\tau_{n+1}^\ell\G X\xrightarrow{\tau_{n+1}\alpha^{\ell-1}_X}\tau_{n+1}\G\tau_n^{\ell-1}X\xrightarrow{\alpha_{\tau_n^{\ell-1}X}}\G\tau_n^\ell X.\]
By Proposition \ref{functorial}(e), there exists $g:\tau_{n+1}^\ell Y\to\tau_{n+1}\G\tau_n^{\ell-1}X$ such that $f=(\alpha_{\tau_n^{\ell-1}X})g$.
By Lemma \ref{XY l}(c), there exists $h:\tau_{n+1}^{\ell-1}Y\to\G\tau_n^{\ell-1}X$ such that $g=\tau_{n+1}h$.
By the inductive hypothesis, we have $\Im h\subset\tau_{n+1}^{\ell-1}\G X$.
Thus we have $\Im f\subset\tau_{n+1}^\ell\G X$.

(b) Since $\HH_{n+1}^\ell=\mod\Gamma_{\ell}$, the assertion follows immediately from (a). 
\end{proof}

We have the following description of the cone of $\Gamma$.

\begin{lemma}\label{tau_{n+1} calculation}
For $\ell\ge0$, we have $\tau_{n+1}^\ell(D\Gamma)\simeq D\Gamma_\ell$ as $\Gamma$-modules.
\end{lemma}

\begin{proof}
Clearly the inclusion $D\Gamma_\ell\to D\Gamma$ is a minimal right $\HH_{n+1}^\ell$-approximation of the $\Gamma$-module $D\Gamma$.

On the other hand, $\tau_{n+1}^\ell\G M\subset\G\tau_n^\ell M$ is a right $\HH_{n+1}^\ell$-approximation by Lemma \ref{right approximation}.
Since we assume that $M$ is basic, we have $M\simeq X\oplus\tau_n^\ell M$ as $\Lambda$-modules,
where $X$ and $\tau_n^\ell M$ have no non-zero direct summands in common.
Then $\tau_n^{-\ell}X=0$ holds.
Since $\soc\G X\notin\HH_{n+1}^\ell$ holds by Lemma \ref{XY l}(a), we have $\Hom_\Gamma(\HH_{n+1}^\ell,\G X)=0$.
Consequently, the minimal right $\HH_{n+1}^\ell$-approximation of $D\Gamma\simeq(\G X)\oplus(\G\tau_n^\ell M)$
is given by $\tau_{n+1}^\ell\G M=\tau_{n+1}^\ell(D\Gamma)$.
\end{proof}

We have the following description of the cone of $\Gamma$.

\begin{proposition}\label{structure}
The cone of $\Gamma$ is Morita equivalent to
\[\def\arraystretch{.5}\left(\begin{array}{cccc}
{\Gamma^0}&{0}&{0}&{\cdots}\\
{\Gamma^1}&{\Gamma^1}&{0}&{\cdots}\\
{\Gamma^2}&{\Gamma^2}&{\Gamma^2}&{\cdots}\\
{\vdots}&{\vdots}&{\vdots}&{\ddots}
\end{array}\right)
\ \mbox{ and }\ 
\def\arraystretch{.5}\left(\begin{array}{cccc}
{\Gamma_0}&{\Gamma_{1}}&{\Gamma_{2}}&{\cdots}\\
{0}&{\Gamma_{1}}&{\Gamma_{2}}&{\cdots}\\
{0}&{0}&{\Gamma_{2}}&{\cdots}\\
{\vdots}&{\vdots}&{\vdots}&{\ddots}
\end{array}\right).\]
\end{proposition}

\begin{proof}
We only have to show that $\End_\Gamma(\bigoplus_{\ell\ge0}\tau_{n+1}^\ell(D\Gamma))$ has the desired form.
By Lemma \ref{tau_{n+1} calculation}, we have
\[\Hom_\Gamma(\tau_{n+1}^i(D\Gamma),\tau_{n+1}^j(D\Gamma))=\Hom_\Gamma(D\Gamma_i,D\Gamma_j)=\Hom_{\Gamma^{\op}}(\Gamma_j,\Gamma_i)=\Gamma_i\]
for any $i\ge j$.
By Theorem \ref{n+1-rigid}(e), we have $\Hom_\Gamma(\tau_{n+1}^i(D\Gamma),\tau_{n+1}^j(D\Gamma))=0$ for any $i<j$. Thus the assertion follows.
\end{proof}

The following crucial result gives a sufficient condition for $\Lambda$ such that $\Gamma$ is absolutely $(n+1)$-complete.

\begin{lemma}\label{criterion for whole module category}
Let $\Lambda$ be an absolutely $n$-complete algebra. Assume that there exist surjections
\[\Lambda=\Lambda_0\to\Lambda_{1}\to\Lambda_{2}\to\cdots\]
of algebras such that $\tau_n^\ell(D\Lambda)\simeq D\Lambda_\ell$ as $\Lambda$-modules for any $\ell\ge0$.
Then the cone $\Gamma$ of $\Lambda$ is absolutely $(n+1)$-complete.
\end{lemma}

\begin{proof}
(i) First we shall show the following assertions.
\begin{itemize}
\item $(1-e_\ell)M$ is a $\Lambda_{\ell}$-module,
\item $\Hom_\Lambda(e_\ell M,D\Lambda_{\ell})=0$,
\item $e_\ell M=\langle e_\ell\rangle M$.
\end{itemize}

Since $\tau_n^i(D\Lambda)\simeq D\Lambda_i$, the first assertion follows from Lemma \ref{e ell}.
By Lemmas \ref{tau_n on M}(e) and \ref{e ell} again, we have $\Hom_\Lambda(e_\ell M,D\Lambda_{\ell})=0$
and $\Hom_\Lambda(e_\ell M,(1-e_\ell)M)=0$.
This implies that $e_\ell M$ is a sub $\Gamma^{\op}$-module of $M$, and we have $e_\ell M=\langle e_\ell\rangle M$.

(ii) Next we shall show the assertion.

We only have to show that $\PP(\NN)=\add\Gamma$.
By Corollary \ref{indecomposable}, any indecomposable object in $\PP(\NN)$ can be written as $\tau_{n+1}^{\ell}\G X$
for some indecomposable object $X\in\MM$ such that $P:=\tau_n^\ell X$ belongs to $\PP(\MM)$.
Since $\Lambda$ is absolutely $n$-complete, $P$ is a projective $\Lambda$-module.
By Lemma \ref{right approximation}(b), we have $\tau_{n+1}^{\ell}\G X=\Hom_\Gamma(\Gamma_{\ell},\G P)$, which is a direct summand of
\[\Hom_\Gamma(\Gamma_{\ell},\G\Lambda)=\Hom_\Gamma(\Gamma_{\ell},DM)=D(\Gamma_{\ell}\otimes_\Gamma M)=D(M/\langle e_\ell\rangle M).\]
Using observations in (i), we have
\[D(M/\langle e_\ell\rangle M)=D(M/e_\ell M)=D((1-e_\ell)M)=\Hom_{\Lambda_\ell}((1-e_\ell)M,D\Lambda_{\ell})=\Hom_\Lambda(M,D\Lambda_{\ell}).\]
This is a projective $\Gamma$-module, and we have the assertion.
\end{proof}

Now we are ready to prove Theorem \ref{recall}.

By Lemmas \ref{tau_{n+1} calculation} and \ref{criterion for whole module category},
we only have to show that the $m\times m$ triangular matrix ring $\Lambda:=T_m^{(1)}(F)$ satisfies
the condition for $n=1$ in Lemma \ref{criterion for whole module category}.
Let $\{f_1,\cdots,f_m\}$ be a complete set of orthogonal primitive idempotents of $\Lambda$ such that
\[f_iJ_{\Lambda}\simeq f_{i+1}\Lambda\]
for any $1\le i<m$ as $\Lambda$-modules. Put
\[\Lambda_{\ell}:=\Lambda/\langle f_1+\cdots+f_\ell\rangle.\]
Then one can easily check that
\[\tau^\ell(D\Lambda)\simeq T_{m-\ell}^{(1)}(F)\simeq D\Lambda_\ell\]
holds as $\Lambda$-modules.
Thus $\Lambda$ satisfies the condition for $n=1$ in Lemma \ref{criterion for whole module category}.
\qed

\section{Absolute $n$-cluster tilting subcategories}\label{section: Absolute}

In this section we shall study algebras with absolute $n$-cluster tilting objects and prove Theorem \ref{iterated Auslander} as an application.
Let us start with the following easy observations.

\begin{lemma}\label{easy gl.dim}
Let $\Lambda$ be a finite dimensional algebra with an absolute $n$-cluster tilting object $M$ ($n\ge1$).
\begin{itemize}
\item[(a)] $\Hom_\Lambda(M,I)$ is a projective-injective $\End_\Lambda(M)$-module for any injective $\Lambda$-module $I$.
\item[(b)] We have $\gl\Lambda=\id M_\Lambda$.
\item[(c)] For any indecomposable direct summand $X$ of $M$, we have either $X$ is projective or $\pd X_\Lambda\ge n$.
\end{itemize}
\end{lemma}

\begin{proof}
(a) $\Hom_\Lambda(M,I)$ is projective by $I\in\add M$. It is injective by $\Hom_\Lambda(M,I)=D\Hom_\Lambda(\nu^-I,M)$ and $\nu^-I\in\add M$.

(b) Since any $X\in\mod\Lambda$ has an exact sequence
\[0\to M_n\to\cdots\to M_1\to X\to 0\]
with $M_i\in\add M$ by \cite[Prop. 2.4.1(2-$\ell$)]{I4}, we have that $\id X_\Lambda\le\id M_\Lambda$.

(c) Immediate from $\Ext^i_\Lambda(X,\Lambda)=0$ for any $0<i<n$.
\end{proof}

Now we prove the following key result.

\begin{proposition}\label{gl.dim n}
Let $\Lambda$ be a finite dimensional algebra with an absolute $n$-cluster tilting object $M$ ($n\ge1$). Let $\Gamma=\End_\Lambda(M)$.
Then $\Ext^i_\Gamma(D\Gamma,\Gamma)=0$ for any $0<i\le n$ if and only if $\gl\Lambda\le n$.
\end{proposition}

\begin{proof}
Take an injective resolution
\begin{eqnarray}\label{injective resolution of M}
0\to M \to I_0\stackrel{f_0}{\longrightarrow}
\cdots\stackrel{f_{n-1}}{\longrightarrow}I_n\stackrel{f_n}{\longrightarrow}\Omega^{-n-1}M\to0.
\end{eqnarray}
Applying $\F=\Hom_\Lambda(M,-)$, we have an exact sequence
\begin{eqnarray}\label{injective resolution}
0\to\Gamma\to\F I_0\to\cdots\to\F I_n\to C\to0
\end{eqnarray}
where we use $\Ext^i_\Lambda(M,M)=0$ for any $0<i<n$.
By Lemma \ref{easy gl.dim}(a), each $\F I_i$ is a projective-injective $\Gamma$-module.
Since we have $\gl\Gamma\le n+1$ by Theorem \ref{Auslander correspondence}, any indecomposable summand of $C$ is injective or projective.
Conversely any indecomposable injective non-projective $\Gamma$-module $I$ is isomorphic to a summand of $C$ since
any indecomposable injective $\Gamma$-module appears in the minimal
injective resolution of the $\Gamma$-module $\Gamma$ by $\gl\Gamma<\infty$.
Hence we only have to show that $\Ext^i_\Gamma(C,\Gamma)=0$ holds for any $0<i\le n$ if and only if $\gl\Lambda\le n$.

Applying $\Hom_\Gamma(-,\Gamma)$ to the projective resolution (\ref{injective resolution}) of $C$ and using Yoneda's Lemma, 
we have an isomorphism
\[\begin{array}{ccccc}
\Hom_\Gamma(\F I_n,\Gamma)&\to\cdots\to&\Hom_\Gamma(\F I_0,\Gamma)&\to&\Hom_\Gamma(\Gamma,\Gamma)\\
\wr\uparrow&&\wr\uparrow&&\wr\uparrow\\
\Hom_\Lambda(I_n,M)&\to\cdots\to&\Hom_\Lambda(I_0,M)&\to&\Hom_\Lambda(M,M)
\end{array}\]
of complexes. Thus $\Ext^i_\Gamma(C,\Gamma)=0$ for any $0<i\le n$ if and only if the complex
\begin{eqnarray}\label{short cut}
\Hom_\Lambda(I_n,M)\stackrel{f_{n-1}}{\longrightarrow}\cdots
\stackrel{f_0}{\longrightarrow}\Hom_\Lambda(I_0,M)\to\Hom_\Lambda(M,M)
\end{eqnarray}
obtained by applying $\Hom_\Lambda(-,M)$ to \eqref{injective resolution of M} is exact.

Using the condition $\Ext^i_\Lambda(D\Lambda,M)=0$ for any $0<i<n$, one can easily check that the following conditions are equivalent.
\begin{itemize}
\item[$\bullet$] \eqref{short cut} is exact at $\Hom_\Lambda(I_i,M)$ for any $0<i<n$.
\item[$\bullet$] $\Ext^i_\Lambda(\Omega^{-n-1}M,M)=0$ for any $0<i<n$.
\end{itemize}
Since $M$ is an $n$-cluster tilting object in $\mod\Lambda$, this is equivalent to
\begin{itemize}
\item[$\bullet$] $\Omega^{-n-1}M\in\add M$.
\end{itemize}

\medskip
Again using the condition $\Ext^i_\Lambda(D\Lambda,M)=0$ for any $0<i<n$, one can easily check that the following conditions are equivalent.
\begin{itemize}
\item[$\bullet$] (\ref{short cut}) is exact at $\Hom_\Lambda(I_0,M)$.
\item[$\bullet$] $\Ext^1_\Lambda(\Im f_1,M)=0$.
\item[$\bullet$] $\cdots$
\item[$\bullet$] $\Ext^{n-1}_\Lambda(\Im f_{n-1},M)=0$.
\item[$\bullet$] $\Ext^n_\Lambda(\Omega^{-n-1}M,M)\stackrel{f_n}{\longrightarrow}\Ext^n_\Lambda(I_n,M)$ is injective.
\end{itemize}
Using Auslander-Reiten duality, this is equivalent to
\begin{itemize}
\item $\underline{\Hom}_\Lambda(\tau_n^-M,I_n)\stackrel{f_n}{\longrightarrow}\underline{\Hom}_\Lambda(\tau_n^-M,\Omega^{-n-1}M)$ is surjective.
\end{itemize}
Since we have $\add M=\add(\Lambda\oplus\tau_n^-M)$ by Proposition \ref{n-AR translation}(a), this is equivalent to
\begin{itemize}
\item $\Hom_\Lambda(M,I_n)\stackrel{f_n}{\longrightarrow}\Hom(M,\Omega^{-n-1}M)$ is surjective.
\end{itemize}

\medskip
Consequently (\ref{short cut}) is exact if and only if $\Omega^{-n-1}M\in\add M$ and
$\Hom_\Lambda(M,I_n)\stackrel{f_n}{\longrightarrow}\Hom_\Lambda(M,\Omega^{-n-1}M)$ is surjective.
This occurs if and only if $f_n$ is a split epimorphism if and only if $\id M_\Lambda\le n$ if and only if $\gl\Lambda\le n$ by Lemma \ref{easy gl.dim}(b).
\end{proof}

Now we prove the following crucial result, which gives the inductive step in our proof of Theorem \ref{iterated Auslander}.

\begin{proposition}\label{inverse iterate}
Let $\Lambda$ be a finite dimensional algebra with an absolute $n$-cluster tilting object $M$ ($n\ge1$).
If $\Gamma=\End_\Lambda(M)$ has an absolute $(n+1)$-cluster tilting subcategory, then $\gl\Lambda\le n\le\dom\Lambda$.
\end{proposition}

\begin{proof}
We have shown $\gl\Lambda\le n$ in Proposition \ref{gl.dim n}.
Take a minimal injective resolution
\begin{equation}\label{mir of Lambda}
0\to\Lambda\to I_0\xrightarrow{f_0}\cdots\xrightarrow{f_{n-2}}I_{n-1}
\xrightarrow{f_{n-1}}I_n\to0
\end{equation}
of the $\Lambda$-module $\Lambda$.
Applying $\F=\Hom_\Lambda(M,-)$, we have an exact sequence
\[0\to\F\Lambda\to\F I_0\xrightarrow{\F f_0}\cdots\xrightarrow{\F f_{n-2}}
\F I_{n-1}\xrightarrow{\F f_{n-1}}\F I_n\]
of $\Gamma$-modules since we have $\Ext_\Lambda^i(M,\Lambda)=0$ for any $0<i<n$.

Now we put $X:=\Cok(\F f_{n-1})$ and $Y:=\tau_{n+1}X$.
We have $\gl\Gamma\le n+1$ by Theorem \ref{Auslander correspondence}.
Since each $\F I_i$ is a projective-injective $\Gamma$-module by Lemma \ref{easy gl.dim}(a), we have that $X$ is an injective $\Gamma$-module.
Hence both $X$ and $Y$ belong to our absolute $(n+1)$-cluster tilting subcategory
by Proposition \ref{n-AR translation}(a).

Applying Lemma \ref{long sequence}(b) and (c) (replace $n$ there by $n+1$) to the sequence \eqref{mir of Lambda}, we have an exact sequence
\begin{eqnarray}\label{proj res of Y}
0=\F\tau_n\Lambda\to\F\tau_nI_0\xrightarrow{\F\tau_nf_0}\cdots
\xrightarrow{\F\tau_nf_{n-2}}\F\tau_nI_{n-1}\xrightarrow{\F\tau_nf_{n-1}}
\F\tau_nI_n\to Y\to0,
\end{eqnarray}
which gives a projective resolution of the $\Gamma$-module $Y$.
Thus we have $\pd Y_\Gamma\le n$.
By Lemma \ref{easy gl.dim}(c) (replace $n$ there by $n+1$), we have that $Y$ is a projective $\Gamma$-module.

Since each $f_i$ in \eqref{mir of Lambda} belongs to $J_{\mod\Lambda}$,
each $\F\tau_nf_i$ in \eqref{proj res of Y} belongs to $J_{\mod\Gamma}$.
Hence \eqref{proj res of Y} is a minimal projective resolution of $Y$.
Consequently we have $\tau_nI_i=0$ for any $0\le i<n$.
Thus we have $\pd(I_i)_\Lambda<n$ for any $0\le i<n$.
Again by Lemma \ref{easy gl.dim}(c), we have that $I_i$ is a projective
$\Lambda$-module for any $0\le i<n$. Thus $\dom\Lambda\ge n$ holds.
\end{proof}

Now we are ready to prove Theorem \ref{iterated Auslander}.

By Theorem \ref{recall}, we only have to show `only if' part.
The assertion for $n=1$ follows automatically from Proposition \ref{0-Auslander}.

We assume $n\ge2$ and put $\Lambda^{(n)}:=\Lambda$. 
By Theorem \ref{Auslander correspondence}, there exists a finite dimensional algebra $\Lambda^{(n-1)}$
and an absolute $(n-1)$-cluster tilting object $M^{(n-1)}$
in $\mod\Lambda^{(n-1)}$ such that $\Lambda^{(n)}$ is isomorphic to $\End_{\Lambda^{(n-1)}}(M^{(n-1)})$.
Clearly $\Lambda^{(n-1)}$ is also ring-indecomposable.
Applying Proposition \ref{inverse iterate} to $(\Lambda,\Gamma):=(\Lambda^{(n-1)},\Lambda^{(n)})$, we have $\gl\Lambda^{(n-1)}\le n-1\le\dom\Lambda^{(n-1)}$.

Repeating similar argument, we have a ring-indecomposable finite dimensional algebra $\Lambda^{(i)}$ ($1\le i\le n$) satisfying the following conditions:
\begin{itemize}
\item[(a)] $\gl\Lambda^{(i)}\le i\le\dom\Lambda^{(i)}$,
\item[(b)] $\Lambda^{(i)}$ has an absolute $i$-cluster tilting object $M^{(i)}$,
\item[(c)] $\Lambda^{(i+1)}$ is isomorphic to $\End_{\Lambda^{(i)}}(M^{(i)})$.
\end{itemize}
Since $\gl\Lambda^{(1)}\le 1\le\dom\Lambda^{(1)}$,
we have that $\Lambda^{(1)}$ is Morita equivalent to $T_m^{(1)}(F)$ for some division algebra $F$ and $m\ge1$ by Proposition \ref{0-Auslander}.

Using the conditions (b) and (c) inductively, we have that $\Lambda^{(i)}$ is Morita equivalent
to our absolutely $i$-complete algebra $T_m^{(i)}(F)$ for any $1\le i\le n$
since any $\add M^{(i)}$ is a unique absolute $i$-cluster tilting
subcategory by Theorem \ref{unique n-cluster}.
Thus we have the assertion.
\qed

\section{$n$-cluster tilting in derived categories}\label{section: derived}

Throughtout this section, let $\Lambda$ be a finite dimensional algebra with $\id{}_{\Lambda}\Lambda=\id\Lambda_\Lambda<\infty$. 
We denote by $\DD:=\KK^{\rm b}(\pr\Lambda)$
the homotopy category of bounded complexes of finitely generated projective $\Lambda$-modules,
and we identify it with $\KK^{\rm b}(\inj\Lambda)$.

Let us start with the following simple observation.

\begin{lemma}\label{finiteness}
Let $\Lambda$ be a finite dimensional algebra with $\gl\Lambda<\infty$. 
\begin{itemize}
\item[(a)] For any $X\in\DD$, there exist only finitely many integers $i$ satisfying $\Hom_{\DD}(\mod\Lambda,X[i])\neq0$.
\item[(b)] $\mod\Lambda$ is a functorially finite subcategory of $\DD$.
\end{itemize}
\end{lemma}

\begin{proof}
(a) Take a sufficiently large integer $k$ such that $H^i(X)=0$ holds if $i<-k$ or $k<i$.
Then we have $\Hom_{\DD}(\mod\Lambda,X[i])=0$ if either $i<-k$ or $k+\gl\Lambda<i$.

(b) We only show contravariant finiteness.
Any $X\in\DD$ is isomorphic to a bounded complex $(\cdots\to I^i\to I^{i+1}\to\cdots)$ of injective $\Lambda$-modules.
It is easily checked that the natural map $Z^0\to X$ is a right $(\mod\Lambda)$-approximation of $X$.
\end{proof}

The following easy observation is quite useful.

\begin{lemma}\label{split}
Let $\Lambda$ be a finite dimensional algebra with $\gl\Lambda\le n$.
\begin{itemize}
\item[(a)] If $X\in\DD$ satisfies $H^i(X)=0$ for any $0<i<n$, then $X\simeq Y\oplus Z$ for some $Y\in\DD^{\le0}$ and $Z\in\DD^{\ge n}$.
\item[(b)] If $X\in\DD$ satisfies $H^i(X)=0$ for any integer $i\notin n\Z$, then $X$ is isomorphic to $\bigoplus_{\ell\in\Z}H^{\ell n}(X)[-\ell n]$.
\end{itemize}
\end{lemma}

\begin{proof}
(a) We can assume that $X$ is a bounded complex $(\cdots\to I^i\stackrel{d^i}{\to}I^{i+1}\to\cdots)$ of finitely generated injective $\Lambda$-modules.
Put
\[Y:=(\cdots\to I^{n-2}\to I^{n-1}\to\Im d^{n-1}\to0\to\cdots)\in\DD^{\le0}.\]
By our assumption, we have an exact sequence
\[0\to Z^0\to I^0\to I^1\to\cdots\to I^{n-2}\stackrel{}{\to}I^{n-1}\stackrel{d^{n-1}}{\to}I^{n}.\]
It follows from $\gl\Lambda\le n$ that $\Im d^{n-1}$ is an injective $\Lambda$-module. 
Thus the inclusion map $f:\Im d^{n-1}\to I^n$ splits, so there exists $g:I^n\to\Im d^{n-1}$ such that $gf=1_{\Im d^{n-1}}$.
We have the following chain homomorphism.
\[\begin{array}{cccccccccc}
Y=&(\cdots&\to&I^{n-1}&\to&\Im d^{n-1}&\to&0&\to&\cdots)\\
&&&\parallel&&\downarrow^f&&\downarrow\\
X=&(\cdots&\to&I^{n-1}&\stackrel{d^{n-1}}{\to}&I^n&\to&I^{n+1}&\to&\cdots)\\
&&&\parallel&&\downarrow^g&&\downarrow\\
Y=&(\cdots&\to&I^{n-1}&\to&\Im d^{n-1}&\to&0&\to&\cdots)
\end{array}\]
Thus $Y$ is a direct summand of $X$, and we have the assertion.

(b) Immediate from (a).
\end{proof}

Now we are ready to prove Theorem \ref{derived}. Put $\BB:=\CC[n\Z]$.

(i) We shall show that $\BB$ is a functorially finite subcategory of $\DD$.

We only show contravariant finiteness. Fix $X\in\DD$.
By Lemma \ref{finiteness}(a), there exist only finitely many integers $\ell$ such that $\Hom_{\DD}(\CC[\ell n],X)\neq0$.
Since $\CC$ is a functorially finite subcategory of $\mod\Lambda$, 
we have that $\CC[\ell n]$ is a contravariantly finite subcategory of
$\DD$ for any $\ell$ by Lemma \ref{finiteness}(b).
Thus there exists a right $\BB$-approximation of $X$.

(ii) We shall show that $\BB$ is $n$-rigid.
We only have to show $\Hom_{\DD}(\CC[kn],\CC[\ell n+i])=0$ for $k,\ell\in\Z$ and $i$ ($0<i<n$).
If $k>\ell$, then this is clearly zero. If $k<\ell$, then this is zero by $\gl\Lambda=n$.
If $k=\ell$, then this is again zero by $n$-rigidity of $\CC$.
Thus $\BB$ is $n$-rigid.

(iii) Assume that $X\in\DD$ satisfies $\Hom_{\DD}(\BB,X[i])=0$ for any $0<i<n$.
Since $\Lambda[-\ell n]\in\BB$ for any $\ell\in\Z$, we have
\[H^{\ell n+i}(X)\simeq\Hom_{\DD}(\Lambda[-\ell n],X[i])=0\]
for any $\ell\in\Z$ and $0<i<n$.
Thus $H^i(X)=0$ holds for any integer $i\notin n\Z$.
Applying Lemma \ref{split}(b), we have $X\simeq\bigoplus_{\ell\in\Z}H^{\ell n}(X)[-\ell n]$.
Moreover we have
\[\Ext^i_\Lambda(\CC,H^{\ell n}(X))\simeq\Hom_{\DD}(\CC,H^{\ell n}(X)[i])=0\]
for any $0<i<n$. Since $\CC$ is an $n$-cluster tilting subcategory of $\mod\Lambda$, we have $H^{\ell n}(X)\in\CC$ for any $\ell\in\Z$.
Thus $X\in\BB$.

(iv) Dually we have that if $X\in\DD$ satisfies $\Hom_{\DD}(X,\BB[i])=0$ for any $0<i<n$, then $X\in\BB$.
Thus we conclude that $\BB$ is an $n$-cluster tilting subcategory of $\DD$.
\qed

\medskip
In the rest of this section we shall prove Theorem \ref{consequence}.

Let $(\DD^{\le0},\DD^{\ge0})$ be the standard $t$-structure of $\DD$, so
\begin{eqnarray*}
\DD^{\le0}&:=&\{X\in\DD\ |\ H^i(X)=0\ \mbox{ for any }i>0\},\\
\DD^{\ge0}&:=&\{X\in\DD\ |\ H^i(X)=0\ \mbox{ for any }i<0\}.
\end{eqnarray*}
Our question whether $\UU_n(\Lambda)$ forms an $n$-cluster tilting subcategory of $\DD$
is closely related to the following conditions for $\Lambda$.

\begin{definition}
Define the conditions (S$_n$) and (T$_n$) for $\Lambda$ as follows:
\begin{itemize}
\item[(S$_n$)] $\SSS_n\DD^{\ge0}\subset\DD^{\ge0}$.
\item[(T$_n$)] $\SSS_n^\ell\DD^{\ge0}\subset\DD^{\ge1}$
for sufficiently large $\ell$.
\end{itemize}
\end{definition}

It is easily shown that (S$_n$) is equivalent to
$\SSS_n^{-1}\DD^{\le0}\subset\DD^{\le0}$, and (T$_n$) is equivalent to
$\SSS_n^{-\ell}\DD^{\le0}\subset\DD^{\le-1}$ for sufficiently large $\ell$.

%

We have the following sufficient conditions for (S$_n$) and (T$_n$).

\begin{proposition}\label{t-structure}
Let $\Lambda$ be a finite dimensional algebra.
\begin{itemize}
\item[(a)] $\id{}_{\Lambda}\Lambda=\id\Lambda_\Lambda\le n$ holds
if and only if (S$_n$) holds.
\item[(b)] $\Lambda$ is $\tau_n$-finite if and only if
$\gl\Lambda\le n$ and (T$_n$) hold.
\end{itemize}
\end{proposition}

\begin{proof}
(a) To show `only if' part, take any object $X\in\DD^{\ge0}$.
Then $X$ is isomorphic to a bounded complex
\[X\simeq(\cdots\to 0\to 0\to I^0\to I^1\to\cdots)\]
of injective $\Lambda$-modules.
Since $\pd(D\Lambda)_{\Lambda}\le n$, we have that $X$ is isomorphic to a bounded complex
\[X\simeq(\cdots\to 0\to 0\to P^{-n}\to P^{1-n}\to\cdots)\]
of projective $\Lambda$-modules by taking a projective resolution.
Applying $\SSS_n$, we have
\[\SSS_nX\simeq(\cdots\to 0\to 0\to\stackrel{0}{\nu(P^{-n})}\to\stackrel{1}{\nu(P^{1-n})}\to\cdots)\in\DD^{\ge0}.\]

Conversely, assume $\id\Lambda_\Lambda>n$.
Take $X\in\mod\Lambda$ such that $\Ext^{n+1}_\Lambda(X,\Lambda)\neq0$.
Since $H^{-1}(\SSS_nX)\simeq D\Ext^{n+1}_\Lambda(X,\Lambda)\neq0$,
we have $\SSS_n\DD^{\ge0}\ {\not\subset}\ \DD^{\ge0}$.

To prove (b), we need the following relationship between two functors
$\tau_n$ and $\SSS_n$.

\begin{lemma}\label{tau and s_n}
Let $\Lambda$ be a finite dimensional algebra satisfying $\gl\Lambda\le n$.
\begin{itemize}
\item[(a)] For any $\ell\ge0$, we have an isomorphism $\tau_n^\ell H^0(-)
\to H^0(\SSS_n^\ell-)$ of functors $\DD^{\ge0}\to\mod\Lambda$.
\item[(b)] For any $\ell\ge0$, we have an isomorphism $\tau_n^{-\ell}H^0(-)
\to H^0(\SSS_n^{-\ell}-)$ of functors $\DD^{\le0}\to\mod\Lambda$.
\end{itemize}
\end{lemma}

\begin{proof}
(a) We only have to show the case $\ell=1$.
We have a morphism $\gamma:H^0(-)\to\id$ of endofunctors $\DD^{\ge0}\to\DD^{\ge0}$.
We shall show that $H^0(\SSS_n\gamma)$ gives the desired isomorphism.
For any $X\in\DD^{\ge0}$, we have a triangle $Y[-1]\to H^0(X)\xrightarrow{\gamma_X}X\to Y$ with $Y\in\DD^{\ge1}$.
Since (S$_n$) holds by $\gl\Lambda\le n$ and Proposition \ref{t-structure}(a),
we have a triangle
$\SSS_nY[-1]\to\SSS_nH^0(X)\xrightarrow{\SSS_n\gamma_X}\SSS_nX\to\SSS_nY$
with $\SSS_nY\in\DD^{\ge1}$. Applying $H^0$, we have an isomorphism
\[\tau_nH^0(X)\simeq DH^n(\RHom_\Lambda(H^0(X),\Lambda))\simeq H^0(\SSS_nH^0(X))\xrightarrow{H^0(\SSS_n\gamma_X)}H^0(\SSS_nX).\]

(b) This is shown dually.
\end{proof}

Now we shall show Proposition \ref{t-structure}(b).
Both conditions imply $\gl\Lambda\le n$.
By (a), we have that $\Lambda$ satisfies (S$_n$).
By Lemma \ref{tau and s_n}(a), we have that $\tau_n^\ell=0$ holds for
sufficiently large $\ell$ if and only if (T$_n$) holds.
Thus the assertion holds.
\end{proof}

Now we can prove Proposition \ref{left right}.

This is a direct consequence of Proposition \ref{t-structure}(b)
since both global dimension and the condition (T$_n$) are left-right
symmetric.\qed

\medskip
We give easy properties of the condition (T$_n$).

\begin{lemma}\label{derived T_n}
Let $\Lambda$ be a finite dimensional algebra satisfying (T$_n$).
\begin{itemize}
\item[(a)] For any $X,Y\in\DD$, there exist only finitely many integers $\ell$ satisfying $\Hom_{\DD}(X,\SSS_n^\ell Y)\neq0$.
\item[(b)] If a finite dimensional algebra $\Gamma$ is derived equivalent to $\Lambda$, then it also satisfies (T$_n$).
\end{itemize}
\end{lemma}

\begin{proof}
(a) Since we have $\Hom_{\DD}(X,\DD^{\le -k})=0=\Hom_{\DD}(X,\DD^{\ge k})$ for sufficiently large $k$, we have the assertion.

(b) We denote by $(\DD_\Lambda^{\le0},\DD_\Lambda^{\ge0})$ and $(\DD_\Gamma^{\le0},\DD_\Gamma^{\ge0})$
the $t$-structures of $\DD$ given by the standard $t$-structures of $\Lambda$ and $\Gamma$ respectively.
We have $\Gamma\in\DD_\Lambda^{\le k}$ and $\Lambda\in\DD_\Gamma^{\le k}$ for sufficiently large $k$.
Then we have $\DD_\Gamma^{\le 0}\subset\DD_\Lambda^{\le k}$ and $\DD_\Lambda^{\le 0}\in\DD_\Gamma^{\le k}$.
Then we have
\[\SSS_n^{-\ell(2k+1)}\DD_\Gamma^{\le0}\subset\SSS_n^{-\ell(2k+1)}\DD_\Lambda^{\le k}\subset\DD_\Lambda^{\le-k-1}\subset\DD_\Gamma^{\le-1}.\]
Thus $\Gamma$ also satisfies (T$_n$).
\end{proof}

We also need the result below. The assertion (b) was independently given
by Amiot in \cite[Prop. 5.4.2]{Am1} (see also \cite[Th. 4.10]{Am2})
and Barot-Fernandez-Platzeck-Pratti-Trepode \cite{BFPPT} for the case $n=2$.

\begin{proposition}\label{sufficient condition}
Let $\Lambda$ be a finite dimensional algebra.
\begin{itemize}
\item[(a)] If $\Lambda$ satisfies (S$_n$), then $\UU_n(\Lambda)$ is $n$-rigid.
\item[(b)] If $\Lambda$ satisfies $\gl\Lambda\le n$ and (T$_n$), then $\UU_n(\Lambda)$ is an $n$-cluster tilting subcategory of $\DD$.
\end{itemize}
\end{proposition}

\begin{proof}
(a) Since $\Lambda\in\DD^{\le0}$, we have $\SSS_n^{-\ell}\Lambda\in\DD^{\le0}$ for any $\ell\ge0$ by (S$_n$).
This implies $\Hom_{\DD}(\Lambda,\SSS_n^{-\ell}\Lambda[i])=0$ for any $\ell\ge0$ and $0<i<n$.

On the other hand, we have $\SSS_n^\ell\Lambda=\SSS_n^{\ell-1}(D\Lambda)[-n]\in\DD^{\ge n}$ for any $\ell>0$ by (S$_n$).
This implies $\Hom_{\DD}(\Lambda,\SSS_n^\ell\Lambda[i])=0$ for any $\ell>0$ and $0<i<n$.

(b) By (a) and Proposition \ref{t-structure}(a), we have that $\UU_n(\Lambda)$ is $n$-rigid.
By Lemma \ref{derived T_n}(a), we have that $\UU_n(\Lambda)$ is a functorially finite subcategory of $\DD$.

Fix any indecomposable object $X\in\DD$.
Since $\UU_n(\Lambda)$ is closed under $\SSS_n^{\pm1}$, the following conditions are equivalent (e.g. \cite[Prop. 3.5]{IY}).
\begin{itemize}
\item $\Hom_{\DD}(\UU_n(\Lambda),X[i])=0$ holds for any $0<i<n$,
\item $\Hom_{\DD}(X,\UU_n(\Lambda)[i])=0$ holds for any $0<i<n$.
\end{itemize}
Thus it remains to show that, if these conditions are satisfied, then $X\in\UU_n(\Lambda)$.
By (T$_n$), there exists an integer $\ell$ such that $\SSS_n^\ell X\in\DD^{\le0}$ and $\SSS_n^{\ell+1}X\notin\DD^{\le0}$.
Put $Y:=\SSS_n^{\ell+1}X$. Then $\SSS_n^{-1}Y\in\DD^{\le0}$ is isomorphic to a bounded complex
\[\SSS_n^{-1}Y\simeq(\cdots\to P^{-1}\to P^0\to0\to0\to\cdots)\]
of projective $\Lambda$-modules, and $Y$ is isomorphic to a bounded complex
\begin{eqnarray}\label{s_ny}
Y\simeq(\cdots\to\stackrel{n-1}{\nu(P^{-1})}\to\stackrel{n}{\nu(P^0)}\to0\to0\to\cdots)
\end{eqnarray}
of injective $\Lambda$-modules.
On the other hand, since $\Hom_{\DD}(\Lambda,Y[i])=0$ for any $0<i<n$, we have $H^i(Y)=0$ for any $0<i<n$.
Since $Y$ is indecomposable and does not belong to $\DD^{\le0}$, we have $Y\in\DD^{\ge n}$ by Lemma \ref{split}(a).
Thus the complex (\ref{s_ny}) is exact except $\nu(P^0)$. This implies $Y=H^n(Y)[-n]$ and that $H^n(Y)$ is an injective $\Lambda$-module.
Consequently we have $Y\in\add(D\Lambda)[-n]\subset\UU_n(\Lambda)$ and $X=\SSS_n^{-\ell-1}Y\in\UU_n(\Lambda)$.
\end{proof}

\medskip
Now we are ready to prove Theorem \ref{consequence}.

We only have to show the latter assertion.
By Proposition \ref{t-structure}(b), we have that $\Lambda$ satisfies (T$_n$).
Let $T\in\DD$ be a tilting complex of $\Lambda$ such that
$\Gamma=\End_{\DD}(T)$ satisfies $\gl\Gamma\le n$.
Since $\Gamma$ is derived equivalent to $\Lambda$,
we can identify $\DD$ with $\KK^{\rm b}(\pr\Gamma)$, and
we have that $\Gamma$ satisfies (T$_n$) by Lemma \ref{derived T_n}(b).
Thus $\UU_n(T)=\UU_n(\Gamma)$ is an $n$-cluster tilting subcategory of
$\DD$ by Proposition \ref{sufficient condition}(b).
\qed

\section{Auslander-Reiten quivers and relations}
\label{section: n-AR quivers}

Throughout this section, we assume that the base field $k$ is algebraically closed for simplicity.
For an arrow or a path $a$ in a quiver $Q$, we denote by $s(a)$ the start vertex and by $e(a)$ the end vertex.

\begin{definition}\label{quiver definition}
\begin{itemize}
\item[(a)] A \emph{weak translation quiver} $Q=(Q_0,Q_1,\tau)$ consists of a quiver $(Q_0,Q_1)$ with a bijection
\[\tau:Q_P\to Q_I\]
for fixed subsets $Q_P$ and $Q_I$ of $Q_0$.
Here we do not assume any relationship between $\tau$ and arrows in $Q$.
We write $\tau x=0$ symbolically for any $x\in Q_0\backslash Q_P$.
\item[(b)] Let $\Lambda$ be an $n$-complete algebra and $\MM=\MM_n(D\Lambda)$ the $\tau_n$-closure of $D\Lambda$.
Define a weak translation quiver $Q=(Q_0,Q_1,\tau_n)$ called the \emph{Auslander-Reiten quiver} of $\MM$ as follows:
\begin{itemize}
\item[$\bullet$] $Q_0$ (respectively, $Q_P$, $Q_I$) is the set of isoclasses of indecomposable objects in $\MM$ (respectively, $\MM_P$, $\MM_I$).
\item[$\bullet$] For $X,Y\in Q_0$, put $d_{XY}:=\dim_k(J_{\MM}(X,Y)/J_{\MM}^2(X,Y))$
and draw $d_{XY}$ arrows from $X$ to $Y$.
\item[$\bullet$] $\tau_n:Q_P\to Q_I$ is given by the equivalence $\tau_n:\MM_P\to\MM_I$.
\end{itemize}
\item[(c)] Again let $\Lambda$ be an $n$-complete algebra and $\UU=\UU_n(D\Lambda)$ the $\SSS_n$-closure of $D\Lambda$.
Define a weak translation quiver $Q=(Q_0,Q_1,\SSS_n)$ called the \emph{Auslander-Reiten quiver} of $\UU$ similarly,
where we put $Q_P=Q_I:=Q_0$ and we define $\SSS_n:Q_0\to Q_0$ by the equivalence $\SSS_n:\UU\to\UU$.
\end{itemize}
\end{definition}

For the case $n=1$, the Auslander-Reiten quivers of $\MM=\mod\Lambda$ and $\UU=\DD^{\rm b}(\mod\Lambda)$ are usual one \cite{ARS,ASS,H}.

In the rest, let $\CC$ be either $\MM$ of $\UU$ in Definition \ref{quiver definition}.
By the following well-known fact, all source morphisms in $\CC$ give the Auslander-Reiten quiver.

\begin{lemma}\label{source quiver}
Let $X,Y\in\CC$ be indecomposable objects and $f_0:X\to M_1$ a source morphism.
Then $d_{XY}$ is equal to the number of $Y$ appearing in the direct sum decomposition of $M_1$.
\end{lemma}

\begin{proof}
The source morphism $f_0:X\to M_1$ induces an isomoprhism
\[\Hom_{\CC}(M_1,Y)/J_{\CC}(M_1,Y)\simeq J_{\CC}(X,Y)/J_{\CC}^2(X,Y).\]
Since we assumed that $k$ is algebraically closed, we have that $\dim_k(\Hom_{\CC}(M_1,Y)/J_{\CC}(M_1,Y))$ is equal to
the number of $Y$ appearing in the direct sum decomposition of $M_1$. Thus we have the assertion.
\end{proof}

We consider a presentation of the category $\CC$ by its Auslander-Reiten quiver with relations.

\begin{definition}
For a quiver $Q$, define an additive category $P(Q)$ called the
\emph{path category} of $Q$ as follows.
\begin{itemize}
\item The set of indecomposable objects in $P(Q)$ is $Q_0$.
\item For any $x,y\in Q_0$, $\Hom_{P(Q)}(x,y)$ is a $k$-vector space with the basis consisting of all paths from $x$ to $y$ in $Q$.
\end{itemize}
\end{definition}

The presentation of $\CC$ can be decided similarly to algebras.
By (a) below, the category $\CC$ is equivalent to some factor category
$P(Q)/I$ of the path category $P(Q)$ of the Auslander-Reiten quiver $Q$
of $\CC$.
By (b) below, the first two terms of source sequences in $\CC$ describe
generators of $I$.

\begin{lemma}\label{presentation}
Let $Q$ be the Auslander-Reiten quiver of $\CC$.
\begin{itemize}
\item[(a)] Assume that we have a morphism $\PPP(a)\in J_{\CC}(X,Y)$
for any arrow $a:X\to Y$ in $Q$, and that
$\{\PPP(a)\ |\ s(a)=X,\ e(a)=Y\}$ forms a $k$-basis of
$J_{\CC}(X,Y)/J^2_{\CC}(X,Y)$ for any $X,Y\in Q_0$.
Then $\PPP$ extends uniquely to a full dense functor $\PPP:P(Q)\to\CC$.
\item[(b)] Assume that we have a full dense functor $\PPP:P(Q)\to\CC$, and that any $x\in Q_0$ has the source sequence with the first two terms
\[\PPP(x)\xrightarrow{(a)}\bigoplus_{a\in Q_1,\ s(a)=x}\PPP(e(a))\xrightarrow{(\PPP(r_{a,i}))}\bigoplus_{1\le i\le m_x}\PPP(e(r_{a,i})).\]
Then the kernel of $\PPP$ is generated by
$\{\sum_{a\in Q_1,\ s(a)=x}r_{a,i}a\ |\ x\in Q_0,\ 1\le i\le m_x\}$.
\end{itemize}
\end{lemma}

\begin{proof}
Since $Q$ is locally finite and acyclic, the path category of $Q$ coincides with its complete path category.
We refer to \cite[Prop. 3.1(b)]{BIRSm} for (a), and to \cite[Prop. 3.6]{BIRSm} for (b).
\end{proof}

If $Q$ is the Auslander-Reiten quiver of $\CC$, then we often identify
objects of $\CC$ with those of $P(Q)$, and we denote the image $\PPP(a)$
of a morphism $a$ in $P(Q)$ under $\PPP$ by the same letter $a$.

\subsection{Cones and cylinders of weak translation quivers}

Throughout this subsection, let $\Lambda$ be an $n$-complete algebra with the $\tau_n$-closure $\MM=\MM_n(D\Lambda)=\add M$ of $D\Lambda$.
We denote by $Q=(Q_0,Q_1,\tau_n)$ the Auslander-Reiten quiver of $\MM$.
Then $\Gamma:=\End_\Lambda(M)$ is $(n+1)$-complete by Theorem \ref{main},
so satisfies the conditions (A$_{n+1}$)--(C$_{n+1}$), (S$_{n+1}$) and
(T$_{n+1}$). We denote by
\[\NN=\MM_{n+1}(D\Gamma)\ \mbox{ and }\ \UU=\UU_{n+1}(D\Gamma)\]
the $\tau_{n+1}$-closure and the $\SSS_{n+1}$-closure of $D\Gamma$ respectively.
They are $(n+1)$-cluster tilting subcategories of $\mod\Gamma$ and
$\KK^{\rm b}(\pr\Gamma)$ respectively by Theorem \ref{consequence}.
The aim of this subsection is to draw the Auslander-Reiten quivers of
$\NN$ and $\UU$ respectively by using $Q$.
The key construction is the following.

\begin{definition}\label{quiver}
Let $Q=(Q_0,Q_1,\tau)$ be a weak translation quiver in general.
\begin{itemize}
\item[(a)] We define a weak translation quiver $Q'=(Q'_0,Q_1',\tau')$ called the \emph{cone} of $Q$ as follows:
\begin{itemize}
\item[$\bullet$] $Q_0':=\{(x,\ell)\ |\ x\in Q_0,\ \ell\ge0,\ \tau^\ell x\neq0\}$.
\item[$\bullet$] There are the following two kinds of arrows.
\begin{itemize}
\item ${(x,\ell)}_1:(x,\ell)\to(\tau x,\ell-1)$ for any $(x,\ell)\in Q'_0$ satisfying $\ell>0$.
\item $(a,\ell):(x,\ell)\to(y,\ell)$ for any arrow $a:x\to y$ in $Q$ satisfying $(x,\ell),(y,\ell)\in Q'_0$.
\end{itemize}
\item[$\bullet$] $Q'_P:=\{(x,\ell)\in Q'_0\ |\ (x,\ell+1)\in Q'_0\}$ and
$Q'_I:=\{(x,\ell)\in Q'_0\ |\ \ell>0\}$.
\item[$\bullet$] Define a bijection $\tau':Q'_P\to Q'_I$ by $\tau'(x,\ell):=(x,\ell+1)$.
\end{itemize}
\item[(b)] We define a weak translation quiver $Q''=(Q''_0,Q''_1,\tau'')$ called the \emph{cylinder} of $Q$ as follows:
\begin{itemize}
\item[$\bullet$] $Q_0''=Q''_P=Q''_I:=Q_0\times\Z=\{(x,\ell)\ |\ x\in Q_0,\ \ell\in\Z\}$.
\item[$\bullet$] There are the following two kinds of arrows.
\begin{itemize}
\item ${(x,\ell)}_1:(x,\ell)\to(\tau x,\ell-1)$ for any $(x,\ell)\in Q''_0$ satisfying $x\in Q_P$.
\item $(a,\ell):(x,\ell)\to(y,\ell)$ for any arrow $a:x\to y$ in $Q$ and $\ell\in\Z$.
\end{itemize}
\item[$\bullet$] Define a bijection $\tau'':Q''_0\to Q''_0$ by $\tau''(x,\ell):=(x,\ell+1)$.
\end{itemize}
\end{itemize}
\end{definition}

To simplify our description of relations below,
we use the following convention:
When we consider the path category $P(Q')$,
we regard $(x,\ell)$ as a zero object if it
does not belong to $Q'_0$,
and regard $(a,\ell)$ as a zero morphism
if it does not belong to $Q'_1$.
We use the same convention for $P(Q'')$.

\begin{example}\normalfont
Consider the following translation quivers, where dotted arrows indicate $\tau$.
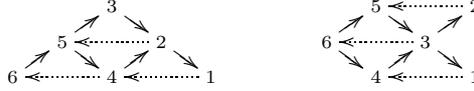
\[\xymatrix@C=0.3cm@R0.1cm{
&&{\scriptstyle 3}\ar[dr]\\
&{\scriptstyle 5}\ar[dr]\ar[ur]&&{\scriptstyle 2}\ar[dr]\ar@{..>}[ll]\\
{\scriptstyle 6}\ar[ur]&&{\scriptstyle 4}\ar[ur]\ar@{..>}[ll]&&{\scriptstyle 1}\ar@{..>}[ll]}\ \ \ \ \ \ \ \ \ \ \xymatrix@C=0.3cm@R0.1cm{
&{\scriptstyle 5}\ar[dr]&&{\scriptstyle 2}\ar@{..>}[ll]\\
{\scriptstyle 6}\ar[dr]\ar[ur]&&{\scriptstyle 3}\ar[dr]\ar[ur]\ar@{..>}[ll]\\
&{\scriptstyle 4}\ar[ur]&&{\scriptstyle 1}\ar@{..>}[ll]}\]
Their cones are given by the following, which coincide with the Auslander-Reiten quivers in Section \ref{section: nCT1}.
{\tiny
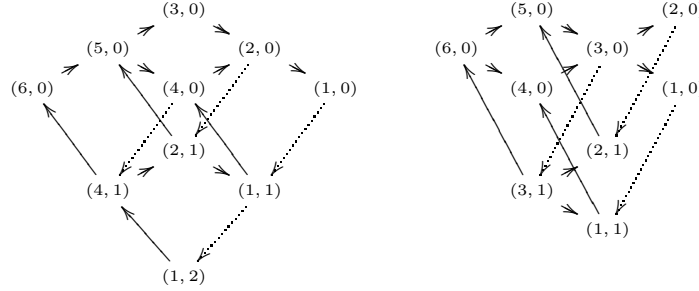
\[\xymatrix@C=0.2cm@R0.1cm{
&&(3,0)\ar[dr]\\
&(5,0)\ar[dr]\ar[ur]
&&(2,0)\ar[dr]\ar@{..>}[dddl]\\
(6,0)\ar[ur]
&&(4,0)\ar[ur]\ar@{..>}[dddl]
&&(1,0)\ar@{..>}[dddl]\\
\\
&&(2,1)\ar[uuul]\ar[dr]\\
&(4,1)\ar[uuul]\ar[ur]
&&(1,1)\ar[uuul]\ar@{..>}[dddl]\\
\ \\
\ \\
&&(1,2)\ar[uuul]}
\ \ \ \ \ \ \ \ \ \xymatrix@C=0.2cm@R0.1cm{
&(5,0)\ar[dr]
&&(2,0)\ar@{..>}[ddddl]\\
(6,0)\ar[dr]\ar[ur]
&&(3,0)\ar[dr]\ar[ur]\ar@{..>}[ddddl]\\
&(4,0)\ar[ur]
&&(1,0)\ar@{..>}[ddddl]\\
\\
&&(2,1)\ar[uuuul]\\
&(3,1)\ar[dr]\ar[ur]\ar[uuuul]\\
&&(1,1)\ar[uuuul]
}\]}
Their cylinders are given by the following.
{\tiny
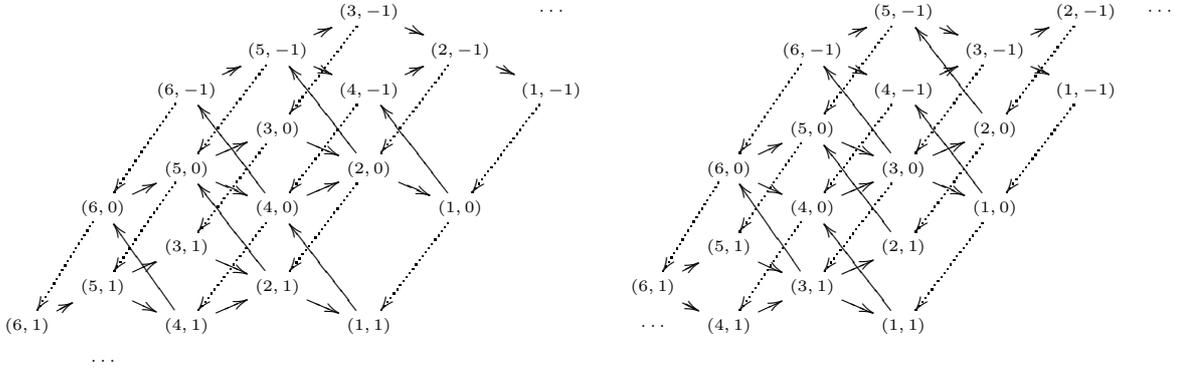
\[\xymatrix@C=0.2cm@R0.1cm{
&&&&&(3,-1)\ar[dr]\ar@{..>}[dddl]&&\cdots\\
&&&&(5,-1)\ar[dr]\ar[ur]\ar@{..>}[dddl]
&&(2,-1)\ar[dr]\ar@{..>}[dddl]\\
&&&(6,-1)\ar[ur]\ar@{..>}[dddl]
&&(4,-1)\ar[ur]\ar@{..>}[dddl]
&&(1,-1)\ar@{..>}[dddl]\\
&&&&(3,0)\ar[dr]\ar@{..>}[dddl]\\
&&&(5,0)\ar[dr]\ar[ur]\ar@{..>}[dddl]
&&(2,0)\ar[dr]\ar[uuul]\ar@{..>}[dddl]\\
&&(6,0)\ar[ur]\ar@{..>}[dddl]
&&(4,0)\ar[ur]\ar[uuul]\ar@{..>}[dddl]
&&(1,0)\ar[uuul]\ar@{..>}[dddl]\\
&&&(3,1)\ar[dr]\\
&&(5,1)\ar[dr]\ar[ur]&&(2,1)\ar[uuul]\ar[dr]\\
&(6,1)\ar[ur]&&(4,1)\ar[uuul]\ar[ur]
&&(1,1)\ar[uuul]\\
&&\cdots}
\ \ \ \ \ \xymatrix@C=0.2cm@R0.1cm{
&&&(5,-1)\ar[dr]\ar@{..>}[dddl]&&(2,-1)\ar@{..>}[dddl]&\cdots\\
&&(6,-1)\ar[dr]\ar[ur]\ar@{..>}[dddl]&&(3,-1)\ar[dr]\ar[ur]\ar@{..>}[dddl]\\
&&&(4,-1)\ar[ur]\ar@{..>}[dddl]&&(1,-1)\ar@{..>}[dddl]\\
&&(5,0)\ar[dr]\ar@{..>}[dddl]&&(2,0)\ar[uuul]\ar@{..>}[dddl]\\
&(6,0)\ar[dr]\ar[ur]\ar@{..>}[dddl]&&(3,0)\ar[uuul]\ar[dr]\ar[ur]\ar@{..>}[dddl]\\
&&(4,0)\ar[ur]\ar@{..>}[dddl]&&(1,0)\ar[uuul]\ar@{..>}[dddl]\\
&(5,1)\ar[dr]&&(2,1)\ar[uuul]\\
(6,1)\ar[dr]\ar[ur]&&(3,1)\ar[dr]\ar[ur]\ar[uuul]\\
\cdots&(4,1)\ar[ur]&&(1,1)\ar[uuul]
}\]}
\end{example}

To draw the Auslander-Reiten quivers of $\NN$ and $\UU$, we introduce some notations.
We use the functor
\[\G:=D\Hom_\Lambda(-,M):\mod\Lambda\to\mod\Gamma.\]
For any path $p=a_1\cdots a_m$ in $Q$ and $\ell\in\Z$, we define
a path in $Q'$ and $Q''$ by
\[(p,\ell):=(a_1,\ell)\cdots(a_m,\ell).\]
For any morphism $r=\sum_pc_pp$ in $P(Q)$ with paths $p$ in $Q$ and $c_p\in k$, we define a morphism in $P(Q')$ and $P(Q'')$
by $(r,\ell):=\sum_pc_p(p,\ell)$.
For any arrow $a:X\to Y$ in $Q$, we choose a morphism
$\tau_n^-a:\tau_n^-X\to\tau_n^-Y$ in $P(Q)$ whose image under
the functor $\PPP:P(Q)\to\MM$ gives the image of $a$ under the functor
$P(Q)\stackrel{\PPP}{\to}\MM\stackrel{\tau_n^-}{\to}\MM$.

We have the following presentation of the category $\NN=\MM_{n+1}(D\Gamma)$.

\begin{theorem}\label{cone}
Under the above circumstances,
the Auslander-Reiten quiver of $\NN$ is given by
the cone $Q'=(Q'_0,Q'_1,\tau_{n+1})$ of the Auslander-Reiten quiver
$Q=(Q_0,Q_1,\tau_n)$ of $\MM$.
Moreover, $\NN$ is presented by the quiver $Q'$ with relations
\begin{itemize}
\item[$\bullet$] $(r,\ell)=0$ for any relation $r=0$ for $\MM$ and $\ell\ge0$,
\item[$\bullet$] ${(Y,\ell)}_1\cdot(\tau_n^-a,\ell)=(a,\ell-1)\cdot{(X,\ell)}_1$ for any arrow $a:\tau_nX\to Y$ in $Q$ and $\ell>0$.
\end{itemize}
\end{theorem}

\begin{proof}
For any object $X\in\MM$ and $\ell\ge0$, we put
\[(X,\ell):=\tau_{n+1}^\ell\G X\in\NN.\]
Under this notation, we have a bijection between $Q'_0$ and isoclasses of
indecomposable objects in $\NN$ by Corollary \ref{indecomposable}.
Moreover $Q'_P$ (respectively, $Q'_I$) corresponds to isoclasses of
indecomposable objects in $\NN_P$ (respectively, $\NN_I$), and the
equivalence $\tau_{n+1}:\NN_P\to\NN_I$ corresponds to the bijection
$\tau':Q'_P\to Q'_I$.

For any morphism $a:X\to Y$ in $\MM$ (or an arrow $a:X\to Y$ in $Q$) and $\ell\ge0$, we define a morphism
\[(a,\ell):=\tau_{n+1}^\ell\G a:(X,\ell)=\tau_{n+1}^\ell\G X\to(Y,\ell)=\tau_{n+1}^\ell\G Y\]
in $\NN$. For any object $(X,\ell)\in\NN$ with $\ell>0$, we define a morphism
\[{(X,\ell)}_1:=\tau_{n+1}^{\ell-1}\alpha_X:(X,\ell)=\tau_{n+1}^\ell\G X\to(\tau_nX,\ell-1)=\tau_{n+1}^{\ell-1}\G\tau_nX\]
in $\NN$. Under these notations, we shall describe all arrows and relations
starting at each vertex $(X,\ell)\in Q'_0$.
We divide into two cases.

(i) Consider the case $(X,\ell)$ with $\ell=0$.

Let $X\stackrel{f_0}{\to}M_1\stackrel{f_1}{\to}M_2$ be the first two terms of the source sequence of $X$ in $\MM$.
By Proposition \ref{injective source}, we have the first two terms
$(X,0)\stackrel{(f_0,0)}{\longrightarrow}(M_1,0)\stackrel{(f_1,0)}{\longrightarrow}(M_2,0)$
of the source sequence of $(X,0)$ in $\NN$.
By Lemma \ref{source quiver}, all arrows starting at $(X,0)$ are given
by $(a,0)$ for each arrow $a$ in $Q$ starting at $X$.
By Lemma \ref{presentation}(b), all relations starting at $(X,0)$ are given
by $(r,0)=0$ for each relation $r=0$ in $\MM$ starting at $X$.

(ii) Consider the case $(X,\ell)$ with $\ell>0$.

Let $\tau_nX\stackrel{f_0}{\to}M_1\stackrel{f_1}{\to}M_2$ be the first two terms of the source sequence of $\tau_nX$ in $\MM$.
By Proposition \ref{n+1-almost split}, we have the first two terms 
\[(X,\ell)\xrightarrow{{(\tau_n^- f_0,\ell)\choose\tau_{n+1}^{\ell-1}\beta_{\tau_nX}}}(\tau_n^{-}M_1,\ell)\oplus(\tau_nX,\ell-1)
\xrightarrow{{(\tau_n^- f_1,\ell)  \ \ \ \ \ 0\ \ \ \ \choose\tau_{n+1}^{\ell-1}\beta_{M_1}\ -(f_0,\ell-1)}}
(\tau_n^{-}M_2,\ell)\oplus(M_1,\ell-1)\]
of the $(n+1)$-almost split sequence of $(X,\ell)$ in $\NN$. 
By our definition of $\beta$ in Lemma \ref{beta}, we have
$\tau_{n+1}^{\ell-1}\beta_{\tau_nX}={(X,\ell)}_1$ and
\[\tau_{n+1}^{\ell-1}\beta_{M_1}\simeq{\tau_{n+1}^{\ell-1}\alpha_{\tau_n^-M_1}\choose0}
={{(\tau_n^-M_1,\ell)}_1\choose0}:
(\tau_n^{-}M_1,\ell)\to(M_1,\ell-1)\simeq(\tau_n\tau_n^-M_1,\ell-1)\oplus
(I,\ell-1)\]
for a decomposition $M_1\simeq(\tau_n\tau_n^-M_1)\oplus I$ with $I\in\II(\MM)$.
By Lemma \ref{source quiver}, all arrows starting at $(X,0)$ are given
by $(X,\ell)_1$ and $(a,0)$ for each arrow $a$ in $Q$ starting at $X$.
By Lemma \ref{presentation}(b), all relations starting at $(X,\ell)$
appear in equalities
\[(\tau_n^- f_0,\ell)\cdot(\tau_n^- f_1,\ell)=0\ \mbox{ and }\ {{(\tau_n^-M_1,\ell)}_1\choose0}\cdot(\tau_n^- f_0,\ell)
=(f_0,\ell-1)\cdot{(X,\ell)}_1.\]
The former equality gives relations $(r,\ell)=0$ for each relation
$r=0$ in $\MM$ starting at $X$.
The latter one gives relations ${(Y,\ell)}_1\cdot(\tau_n^-a,\ell)=
(a,\ell-1)\cdot{(X,\ell)}_1$ for each arrow $a:\tau_nX\to Y$ in $Q$.
(Notice that $(Y,\ell)_1$ which does not belong to $Q'_1$ appears
in the lower half of the morphism ${{(\tau_n^-M_1,\ell)}_1\choose0}$.)

Thus we have the desired assertions.
\end{proof}

Next we shall give a presentation of the category $\UU=\UU_{n+1}(D\Gamma)$.
We need the following information, which is similar to Lemma \ref{tau directed}.

\begin{lemma}\label{U directed}
Fix an indecomposable object $X\in\MM$ and $\ell\in\Z$. Take a source morphism $f_0:X\to M_1$ in $\MM$.
\begin{itemize}
\item[(a)] Any morphism $\SSS_{n+1}^{\ell}\G X\to\SSS_{n+1}^i\G M$
with $i>\ell$ is zero.
\item[(b)] Any morphism $\SSS_{n+1}^{\ell}\G X\to\SSS_{n+1}^\ell\G M$
which is not a split monomorphism factors through
$\SSS_{n+1}^\ell\G f_0:\SSS_{n+1}^\ell\G X\to\SSS_{n+1}^\ell\G M_1$.
\item[(c)] If $X\in\MM_P$, then $\SSS_{n+1}\G X\simeq
\tau_{n+1}\G X$ and any morphism
$\SSS_{n+1}^{\ell}\G X\to\SSS_{n+1}^i\G M$ with $i<\ell$ factors through
$\SSS_{n+1}^{\ell-1}\alpha_X:\SSS_{n+1}^\ell\G X\simeq
\SSS_{n+1}^{\ell-1}\tau_{n+1}\G X\to\SSS_{n+1}^{\ell-1}\G\tau_nX$.
\item[(d)] If $X\in\PP(\MM)$, then any morphism $\SSS_{n+1}^{\ell}\G X\to\SSS_{n+1}^i\G M$ with $i<\ell$ is zero.
\end{itemize}
\end{lemma}

\begin{proof}
(a) Since
$\SSS_{n+1}^{\ell-i}\G M=\SSS_{n+1}^{1+\ell-i}(\Gamma[n+1])\in\DD^{\le -n-1}$
holds by (S$_{n+1}$), we have
\[\Hom_{\DD}(\SSS_{n+1}^{\ell}\G X,\SSS_{n+1}^i\G M)
\simeq\Hom_{\DD}(\SSS_{n+1}^{\ell-i}\G X,D\Gamma)
\simeq DH^0(\SSS_{n+1}^{\ell-i}\G X)=0.\]

(b) This is clear since the functor $\SSS_{n+1}^\ell\G:\MM\to\UU$ is fully faithful.

(c) Since $\G X\in\NN_P$ by Corollary \ref{indecomposable}(b), we have
$\SSS_{n+1}\G X\simeq\tau_{n+1}\G X$ by (C$_{n+1}$).
Applying $\SSS_{n+1}^{-i}$, we only have to show that
{\small\[DH^0(\SSS_{n+1}^{\ell-i-1}\alpha_X):DH^0(\SSS_{n+1}^{\ell-i-1}\G\tau_nX)\simeq\Hom_{\DD}(\SSS_{n+1}^{\ell-i-1}\G\tau_nX,D\Gamma)\to
DH^0(\SSS_{n+1}^{\ell-i}\G X)\simeq
\Hom_{\DD}(\SSS_{n+1}^{\ell-i}\G X,D\Gamma)\]}
is surjective.
By Lemma \ref{tau and s_n} (replace $n$ there by $n+1$), this is equal
to the dual of
\[\tau_{n+1}^{\ell-i-1}\alpha_X:\tau_{n+1}^{\ell-i}\G X\to
\tau_{n+1}^{\ell-i-1}\G\tau_nX.\]
This is injective since $\alpha_X$ is injective and the functor $\tau_{n+1}$ preserves monomorphisms.


(d) We have $\pd(\G X)_\Gamma\le n$ by Corollary \ref{indecomposable}(b).
Since $\SSS_{n+1}^{i-\ell}\G M=\SSS_{n+1}^{1+i-\ell}(\Gamma[n+1])\in
\DD^{\le -n-1}$ holds by (S$_{n+1}$), we have 
\[\Hom_{\DD}(\SSS_{n+1}^{\ell}\G X,\SSS_{n+1}^i\G M)\simeq\Hom_{\DD}(\G X,\SSS_{n+1}^{i-\ell}\G M)=0.\]
\end{proof}

Consequently, we have the result (a) below which is an analogue of
Proposition \ref{source}, and the result (b) below which is an analogue
of Corollary \ref{indecomposable}.

\begin{proposition}\label{S approximation}
\begin{itemize}
\item[(a)] For any object $X\in\MM$ and $\ell\in\Z$, take a source morphism $f_0:X\to M_1$ in $\MM$.
\begin{itemize}
\item[(i)] If $X\in\PP(\MM)$, then a left almost split morphism of $\SSS_{n+1}^\ell\G X$ in $\UU$ is given by
$\SSS_{n+1}^\ell\G f_0:\SSS_{n+1}^\ell\G X\to\SSS_{n+1}^\ell\G M_1$.
\item[(ii)] If $X\in\MM_P$, then a left almost split morphism of $\SSS_{n+1}^\ell\G X$ in $\UU$ is given by
\[{\SSS_{n+1}^\ell\G f_0\choose\SSS_{n+1}^{\ell-1}\alpha_X}:\SSS_{n+1}^\ell\G X\to(\SSS_{n+1}^\ell\G M_1)\oplus(\SSS_{n+1}^{\ell-1}\G\tau_nX).\]
\end{itemize}
\item[(b)] A bijection from isoclasses of indecomposable objects in $\UU$
to pairs $(X,\ell)$ of isoclasses of indecomposable objects $X\in\MM$ and
$\ell\in\Z$ is given by $\SSS_{n+1}^\ell\G X\leftrightarrow(X,\ell)$.
\end{itemize}
\end{proposition}

\begin{proof}
(a) This is immediate from Lemma \ref{U directed}(a)--(d).

(b) Any indecomposable object in $\UU$ is isomorphic to $\SSS_{n+1}^\ell\G X$ for some $X\in\MM$ and $\ell\in\Z$.
If $\SSS_{n+1}^\ell\G X\simeq\SSS_{n+1}^{\ell'}\G X'$, then $\ell=\ell'$ holds by Lemma \ref{U directed}(a), and $X\simeq X'$ holds since the functor
$\SSS_{n+1}^\ell\G:\MM\to\UU$ is fully faithful.
Thus we have the desired bijection.
\end{proof}

We have the following presentation of the category $\UU=\UU_{n+1}(D\Gamma)$.

\begin{theorem}\label{cylinder}
Under the above circumstances, the Auslander-Reiten quiver of $\UU$
is given by the cylinder $Q''=(Q''_0,Q''_1,\SSS_{n+1})$ of the
Auslander-Reiten quiver $Q=(Q_0,Q_1,\tau_n)$ of $\MM$.
Moreover, $\UU$ is presented by the quiver $Q''$ with relations
\begin{itemize}
\item[$\bullet$] $(r,\ell)=0$ for any relation $r=0$ for $\MM$ and $\ell\in\Z$,
\item[$\bullet$] ${(Y,\ell)}_1\cdot(a,\ell)=0$ for any arrow $a:X\to Y$ in $Q$ with $X\in Q\backslash Q_P$ and $\ell\in\Z$,
\item[$\bullet$] ${(Y,\ell)}_1\cdot(\tau_n^-a,\ell)=(a,\ell-1)\cdot{(X,\ell)}_1$ for any arrow $a:\tau_nX\to Y$ in $Q$ and $\ell\in\Z$.
\end{itemize}
\end{theorem}

\begin{proof}
The proof is similar to that of Theorem \ref{cone}.
For any object $X\in\MM$, we put
\[(X,\ell):=\SSS_{n+1}^\ell\G X\in\UU.\]
Under this notation, we have a bijection between $Q_0''$ and isoclasses
of indecomposable objects in $\UU$ by Proposition \ref{S approximation}(b).
For any morphism $a:X\to Y$ in $\MM$ (or an arrow $a:X\to Y$ in $Q$) and $\ell\ge0$, we define a morphism
\[(a,\ell):=\SSS_{n+1}^\ell\G a:(X,\ell)=\SSS_{n+1}^\ell\G X\to(Y,\ell)=\SSS_{n+1}^\ell\G Y.\]
For any object $(X,\ell)\in\UU$, we define a morphism
\[{(X,\ell)}_1:=\SSS_{n+1}^{\ell-1}\alpha_X:(X,\ell)=\SSS_{n+1}^\ell\G X\to(\tau_nX,\ell-1)=\SSS_{n+1}^{\ell-1}\G\tau_nX.\]
Under these notations, we shall describe all arrows and relations starting at each vertex $(X,\ell)\in Q''_0$.
Since $\SSS_{n+1}$ is an autoequivalence of $\UU$, we only have to consider two cases (i) and (ii) below.

(i) Consider the case $(X,0)$ with $X\in\PP(\MM)$.

Let $X\stackrel{f_0}{\to}M_1\stackrel{f_1}{\to}M_2$ be the first two terms
of the source sequence of $X$ in $\MM$.
By Proposition \ref{S approximation}(a) and Lemma \ref{U directed}(d),
it is easily checked that the first two terms
of the source sequence of $(X,0)$ in $\UU$ is given by
\[(X,0)\stackrel{(f_0,0)}{\longrightarrow}(M_1,0)\xrightarrow{{(f_1,0)\choose{(M_1,0)}_1}}(M_2,0)\oplus(\tau_nM_1,1).\]
By Lemma \ref{presentation}(b), all relations starting at $(X,0)$
are given by equalities
\[(f_1,0)\cdot(f_0,0)=0\ \mbox{ and }\ {(M_1,0)}_1\cdot(f_0,0)=0.\]
The former equality gives a relation $(r,0)=0$ for each relation
$r=0$ in $\MM$ starting at $X$.
The latter equality gives a relation ${(Y,0)}_1\cdot(a,0)=0$ for each
arrow $a:X\to Y$ in $Q$.

(ii) Consider the case $(X,1)$ with $X\in\MM_P$.

Let $\tau_nX\stackrel{f_0}{\to}M_1\stackrel{f_1}{\to}M_2$ be the first two terms of the source sequence of $\tau_nX$ in $\MM$.
By Proposition \ref{n+1-almost split}, we have the first two terms 
\begin{equation}\label{2 terms}
(X,1)\xrightarrow{{(\tau_n^- f_0,1)\choose\beta_{\tau_nX}}}(\tau_n^{-}M_1,1)\oplus(\tau_nX,0)
\xrightarrow{{(\tau_n^- f_1,1) \ \ \ 0\ \ \ \choose\ \ \beta_{M_1}\ \ \ -(f_0,0)}}
(\tau_n^{-}M_2,1)\oplus(M_1,0)
\end{equation}
of the $(n+1)$-almost split sequence of $(X,1)$ in $\NN$.
We have $\beta_{\tau_nX}=\alpha_X={(X,1)}_1$ and
\[\beta_{M_1}\simeq{\alpha_{\tau_n^-M_1}\choose0}={{(\tau_n^-M_1,1)}_1\choose0}:
(\tau_n^{-}M_1,1)\to(M_1,0)\simeq(\tau_n\tau_n^-M_1,0)\oplus(I,0)\]
for a decomposition $M_1\simeq(\tau_n\tau_n^-M_1)\oplus I$ with $I\in\II(\MM)$.

By Proposition \ref{S approximation} and \cite[Prop. 3.9]{IY},
we have that \eqref{2 terms} is the first two terms of a source
sequence of $(X,1)$ also in $\UU$.
By Lemma \ref{presentation}(b), all relations starting at $(X,1)$
are given by equalities
\[(\tau_n^- f_0,1)\cdot(\tau_n^- f_1,1)=0\ \mbox{ and }\ {{(\tau_n^-M_1,1)}_1\choose0}\cdot(\tau_n^- f_0,1)=(f_0,0)\cdot{(X,1)}_1.\]
The former equality gives a relation $(r,1)=0$ for each relation $r=0$
in $\MM$ starting at $X$. The latter one gives a relation
${(Y,1)}_1\cdot(\tau_n^-a,1)=(a,0)\cdot{(X,1)}_1$ for each arrow
$a:\tau_nX\to Y$ in $Q$.
\end{proof}

\subsection{Examples}

Throughout this subsection, let $Q=(Q_0,Q_1)$ be a Dynkin quiver and
$\Lambda^{(1)}:=kQ$ the path algebra of $Q$. Let $n\ge1$.
By Corollary \ref{main2}, we have an $n$-complete algebra $\Lambda^{(n)}$
with the cone $\Lambda^{(n+1)}$ for any $n\ge1$.
We denote by
\[\MM^{(n)}:=\MM_n(D\Lambda^{(n)})\ \mbox{ and }\ \UU^{(n)}:=\UU_n(D\Lambda^{(n)})\]
the $\tau_n$-closure and the $\SSS_n$-closure of $D\Lambda^{(n)}$ respectively.
They are $n$-cluster tilting subcategories of $\mod\Lambda^{(n)}$ and
$\KK^{\rm b}(\pr\Lambda^{(n)})$ respectively.
Let us draw the Auslander-Reiten quivers of $\MM^{(n)}$ and $\UU^{(n)}$.
As usual, we denote by $\tau=\tau_1:\mod\Lambda^{(1)}\to\mod\Lambda^{(1)}$
the Auslander-Reiten translation of $\Lambda^{(1)}$.
Let $I_x$ be the indecomposable injective $\Lambda^{(1)}$-module corresponding
to the vertex $x\in Q_0$ and
\[\ell_x:=\sup\{\ell\ge0\ |\ \tau^\ell I_x\neq0\}.\]
Since $Q$ is a Dynkin quiver, we have $\ell_x<\infty$ for any $x\in Q_0$.
For $\ell\in\Z$, we put
\[\Delta_\ell^{(n)}:=\{(\ell_1,\cdots,\ell_n)\in\Z^n\ |\ \ell_1,\cdots,\ell_n\ge0,\ \ell_1+\cdots+\ell_n\le \ell\}.\]
For $1\le i\le n$, we put
\[\mbox{\boldmath $e$}_i:=(\stackrel{1}{0},\cdots,\stackrel{i-1}{0},\stackrel{i}{1},\stackrel{i+1}{0},\cdots,\stackrel{n}{0})\in\Z^n\ \mbox{ and }\ 
\mbox{\boldmath $v$}_i:=\left\{\begin{array}{cc}
-\mbox{\boldmath $e$}_i&i=1,\\
\mbox{\boldmath $e$}_{i-1}-\mbox{\boldmath $e$}_i&1<i\le n.
\end{array}\right.\]

\begin{definition}
Let $Q$ be a Dynkin quiver and $n\ge1$.
\begin{itemize}
\item[(a)] We define a weak translation quiver $Q^{(n)}=(Q_0^{(n)},Q_1^{(n)},\tau_n)$ as follows:
\begin{itemize}
\item[$\bullet$] $Q_0^{(n)}:=\{(x,\mbox{\boldmath $\ell$})\ |\ x\in Q_0,\ \mbox{\boldmath $\ell$}\in\Delta_{\ell_x}^{(n)}\}$.
\item[$\bullet$] There are the following $(n+1)$ kinds of arrows
if their start and end vertices belong to $Q_0^{(n)}$.
\begin{itemize}
\item $(a^*,\mbox{\boldmath $\ell$}):(x,\mbox{\boldmath $\ell$})\to(w,\mbox{\boldmath $\ell$})$ for any arrow $a:w\to x$ in $Q$.
\item $(b,\mbox{\boldmath $\ell$}):(x,\mbox{\boldmath $\ell$})\to(y,\mbox{\boldmath $\ell$}+\mbox{\boldmath $v$}_1)$ for any arrow $b:x\to y$ in $Q$.
\item $(x,\mbox{\boldmath $\ell$})_i:(x,\mbox{\boldmath $\ell$})\to(x,\mbox{\boldmath $\ell$}+\mbox{\boldmath $v$}_i)$ for any $1<i\le n$.
\end{itemize}
\item[$\bullet$] $Q_P^{(n)}:=\{(x,\mbox{\boldmath $\ell$})\in Q_0^{(n)} |\ (x,\mbox{\boldmath $\ell$}+\mbox{\boldmath $e$}_n)\in Q_0^{(n)}\}$ and
$Q_I^{(n)}:=\{(x,\mbox{\boldmath $\ell$})\in Q_0^{(n)} |\ (x,\mbox{\boldmath $\ell$}-\mbox{\boldmath $e$}_n)\in Q_0^{(n)}\}$.
\item[$\bullet$] Define a bijection $\tau_n:Q_P^{(n)}\to Q_I^{(n)}$ by
$\tau_n(x,\mbox{\boldmath $\ell$}):=(x,\mbox{\boldmath $\ell$}+\mbox{\boldmath $e$}_n)$.
\end{itemize}
\item[(b)] We define a weak translation quiver $\widetilde{Q}^{(n)}=(\widetilde{Q}_0^{(n)},\widetilde{Q}_1^{(n)},\SSS_n)$ as follows:
\begin{itemize}
\item[$\bullet$] $\widetilde{Q}_0^{(n)}=\widetilde{Q}_P^{(n)}=\widetilde{Q}_I^{(n)}
:=\{(x,\ell_1,\cdots,\ell_n)\ |\ x\in Q_0,\ (\ell_1,\cdots,\ell_{n-1})\in\Delta_{\ell_x}^{(n-1)},\ \ell_n\in\Z\}$.
\item[$\bullet$] There are the following $(n+1)$ kinds of arrows
if their start and end vertices belong to $\widetilde{Q}_0^{(n)}$.
\begin{itemize}
\item $(a^*,\mbox{\boldmath $\ell$}):(x,\mbox{\boldmath $\ell$})\to(w,\mbox{\boldmath $\ell$})$ for any arrow $a:w\to x$ in $Q$.
\item $(b,\mbox{\boldmath $\ell$}):(x,\mbox{\boldmath $\ell$})\to(y,\mbox{\boldmath $\ell$}+\mbox{\boldmath $v$}_1)$ for any arrow $b:x\to y$ in $Q$.
\item $(x,\mbox{\boldmath $\ell$})_i:(x,\mbox{\boldmath $\ell$})\to(x,\mbox{\boldmath $\ell$}+\mbox{\boldmath $v$}_i)$ for any $1<i\le n$.
\end{itemize}
\item[$\bullet$] Define a bijection $\SSS_n:\widetilde{Q}_0^{(n)}\to \widetilde{Q}_0^{(n)}$ by
$\SSS_n(x,\mbox{\boldmath $\ell$}):=(x,\mbox{\boldmath $\ell$}+\mbox{\boldmath $e$}_n)$.
\end{itemize}
\end{itemize}
\end{definition}

To simplify our description of relations below,
we use the following convention:
When we consider the path category $P(Q^{(n)})$,
we regard $(x,\mbox{\boldmath $\ell$})$ as a zero object if it
does not belong to $Q^{(n)}_0$,
and regard $(a^*,\mbox{\boldmath $\ell$})$, e.t.c.
as a zero morphism if it does not belong to $Q^{(n)}_1$.
We use the same convention for $P(\widetilde{Q}^{(n)})$.

Now we have the presentations of $\MM^{(n)}$ and $\UU^{(n)}$ as follows.

\begin{theorem}\label{n-AR quiver}
Under the above circumstances, we have the following assertions.
\begin{itemize}
\item[(a)] The Auslander-Reiten quivers of $\MM^{(n)}$ and $\UU^{(n)}$ are given by $Q^{(n)}$ and $\widetilde{Q}^{(n)}$ respectively.
\item[(b)] The categories $\MM^{(n)}$ and $\UU^{(n)}$ are presented by quivers $Q^{(n)}$ and $\widetilde{Q}^{(n)}$ with the following relations respectively:
For any $\mbox{\boldmath $\ell$}\in\Z^n$ and $1<i,j\le n$,
\begin{eqnarray*}
&(w,\mbox{\boldmath $\ell$})_i\cdot(a^*,\mbox{\boldmath $\ell$})=(a^*,\mbox{\boldmath $\ell$}+\mbox{\boldmath $v$}_i)\cdot(x,\mbox{\boldmath $\ell$})_i&\mbox{for any arrow $a:w\to x$ in $Q$,}\\
&(y,\mbox{\boldmath $\ell$}+\mbox{\boldmath $v$}_1)_i\cdot(b,\mbox{\boldmath $\ell$})=(b,\mbox{\boldmath $\ell$}+\mbox{\boldmath $v$}_i)\cdot(x,\mbox{\boldmath $\ell$})_i&\mbox{for any arrow $b:x\to y$ in $Q$,}\\
&(x,\mbox{\boldmath $\ell$}+\mbox{\boldmath $v$}_j)_i\cdot(x,\mbox{\boldmath $\ell$})_j=(x,\mbox{\boldmath $\ell$}+\mbox{\boldmath $v$}_i)_j\cdot(x,\mbox{\boldmath $\ell$})_i&\mbox{for any $x\in Q_0$,}\\
&{\displaystyle\sum_{a\in Q_1,\ e(a)=x}(a,\mbox{\boldmath $\ell$})\cdot(a^*,\mbox{\boldmath $\ell$})
=\sum_{b\in Q_1,\ s(b)=x}(b^*,\mbox{\boldmath $\ell$}+\mbox{\boldmath $v$}_1)\cdot(b,\mbox{\boldmath $\ell$})}&\mbox{for any $x\in Q_0$.}
\end{eqnarray*}
\end{itemize}
\end{theorem}

We have the quiver with relations of $\Lambda^{(n+1)}$
by taking the opposite of those of $\MM^{(n)}$ for any $n\ge1$.

\begin{proof}
It is well known that the assertions are valid for $n=1$ \cite{H}.

Clearly $Q^{(n)}$ and $\widetilde{Q}^{(n)}$ are the cone and the cylinder
of $Q^{(n-1)}$ respectively under the following identifications
for $x\in Q_0$, $a\in Q_1$ and $\mbox{\boldmath $\ell$}\in\Z^{n-1}$ and 
\begin{eqnarray*}
(x,\mbox{\boldmath $\ell$},\ell_n)\longleftrightarrow
((x,\mbox{\boldmath $\ell$}),\ell_n),&&
(a,\mbox{\boldmath $\ell$},\ell_n)\longleftrightarrow
((a,\mbox{\boldmath $\ell$}),\ell_n),\\
(a^*,\mbox{\boldmath $\ell$},\ell_n)\longleftrightarrow
((a^*,\mbox{\boldmath $\ell$}),\ell_n),&&
(x,\mbox{\boldmath $\ell$},\ell_n)_i\longleftrightarrow
\left\{\begin{array}{cc}
((x,\mbox{\boldmath $\ell$})_i,\ell_n)&1\le i<n,\\
((x,\mbox{\boldmath $\ell$}),\ell_n)_1&i=n.
\end{array}\right.
\end{eqnarray*}
It is easily checked that our relations for $Q^{(n)}$ and
$\widetilde{Q}^{(n)}$ are obtained from our relations of $Q^{(n-1)}$
by applying Theorems \ref{cone} and \ref{cylinder} respectively.
Thus the assertion follows inductively.
\end{proof}

\begin{example}\normalfont
For simplicity, we denote by
\[x\ell_1\cdots \ell_n\ \ \ (\mbox{respectively, }\ \ x\ell_1\cdots\dot{\ell}_i\cdots\ell_n,\ \ \ a^*\ell_1\cdots \ell_n,\ \ \ b\ell_1\cdots \ell_n)\]
the vertex $(x,\mbox{\boldmath $\ell$})\in Q_0^{(n)}$ (respectively, the arrow $(x,\mbox{\boldmath $\ell$})_i$, $(a^*,\mbox{\boldmath $\ell$})$, $(b,\mbox{\boldmath $\ell$})\in Q_1^{(n)}$) for $\mbox{\boldmath $\ell$}=(\ell_1,\cdots,\ell_n)$.

(a) Let $Q$ be the quiver 
${\scriptstyle\bf 1}\stackrel{a}{\longrightarrow}{\scriptstyle\bf 2}\stackrel{b}{\longrightarrow}{\scriptstyle\bf 3}\stackrel{c}{\longrightarrow}{\scriptstyle\bf 4}$ of type $A_4$.
In this case we have $\Lambda^{(n)}=T_4^{(n)}(k)$ in Theorem \ref{recall}.
Then the Auslander-Reiten quiver of $\MM^{(1)}$ is the following.
{\tiny
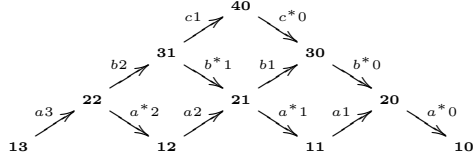
\[\xymatrix@C=0.5cm@R0.3cm{
&&&{\scriptstyle\bf 40}\ar[dr]^{c^*0}\\
&&{\scriptstyle\bf 31}\ar[ur]^{c1}\ar[dr]^{b^*1}&&{\scriptstyle\bf 30}\ar[dr]^{b^*0}\\
&{\scriptstyle\bf 22}\ar[ur]^{b2}\ar[dr]^{a^*2}&&{\scriptstyle\bf 21}\ar[ur]^{b1}\ar[dr]^{a^*1}&&{\scriptstyle\bf 20}\ar[dr]^{a^*0}\\
{\scriptstyle\bf 13}\ar[ur]^{a3}&&{\scriptstyle\bf 12}\ar[ur]^{a2}&&{\scriptstyle\bf 11}\ar[ur]^{a1}&&{\scriptstyle\bf 10}}\]}
The Auslander-Reiten quiver of $\MM^{(2)}$ is the following.
{\tiny
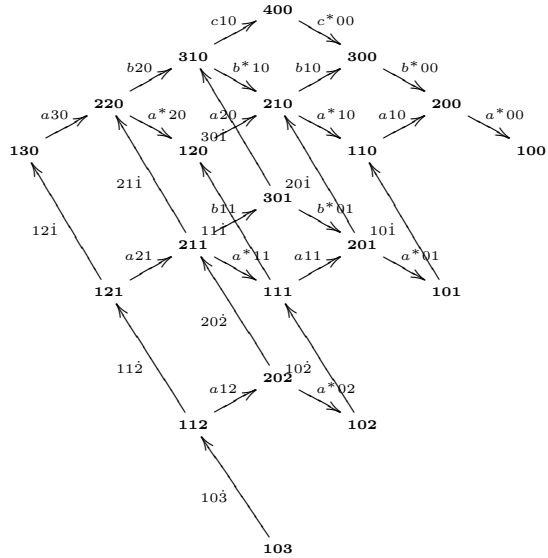
\[\xymatrix@C=0.5cm@R0.3cm{
&&&{\scriptstyle\bf 400}\ar[dr]^{c^*00}\\
&&{\scriptstyle\bf 310}\ar[ur]^{c10}\ar[dr]^{b^*10}&&{\scriptstyle\bf 300}\ar[dr]^{b^*00}\\
&{\scriptstyle\bf 220}\ar[ur]^{b20}\ar[dr]^{a^*20}&&{\scriptstyle\bf 210}\ar[ur]^{b10}\ar[dr]^{a^*10}&&{\scriptstyle\bf 200}\ar[dr]^{a^*00}\\
{\scriptstyle\bf 130}\ar[ur]^{a30}&&{\scriptstyle\bf 120}\ar[ur]^{a20}&&{\scriptstyle\bf 110}\ar[ur]^{a10}&&{\scriptstyle\bf 100}\\
 &&&{\scriptstyle\bf 301}\ar[dr]^{b^*01}\ar[uuul]^{30\dot{1}}\\
 &&{\scriptstyle\bf 211}\ar[ur]^{b11}\ar[dr]^{a^*11}\ar[uuul]^{21\dot{1}}&&{\scriptstyle\bf 201}\ar[dr]^{a^*01}\ar[uuul]^{20\dot{1}}\\
 &{\scriptstyle\bf 121}\ar[ur]^{a21}\ar[uuul]^{12\dot{1}}&&{\scriptstyle\bf 111}\ar[ur]^{a11}\ar[uuul]^{11\dot{1}}&&{\scriptstyle\bf 101}\ar[uuul]^{10\dot{1}}\\
\ \\
&&&{\scriptstyle\bf 202}\ar[dr]^{a^*02}\ar[uuul]^{20\dot{2}}\\
&&{\scriptstyle\bf 112}\ar[ur]^{a12}\ar[uuul]^{11\dot{2}}&&{\scriptstyle\bf 102}\ar[uuul]^{10\dot{2}}\\
\ \\
\ \\
&&&{\scriptstyle\bf 103}\ar[uuul]^{10\dot{3}}
}\]}
The Auslander-Reiten quiver of $\MM^{(3)}$ is the following.
{\tiny
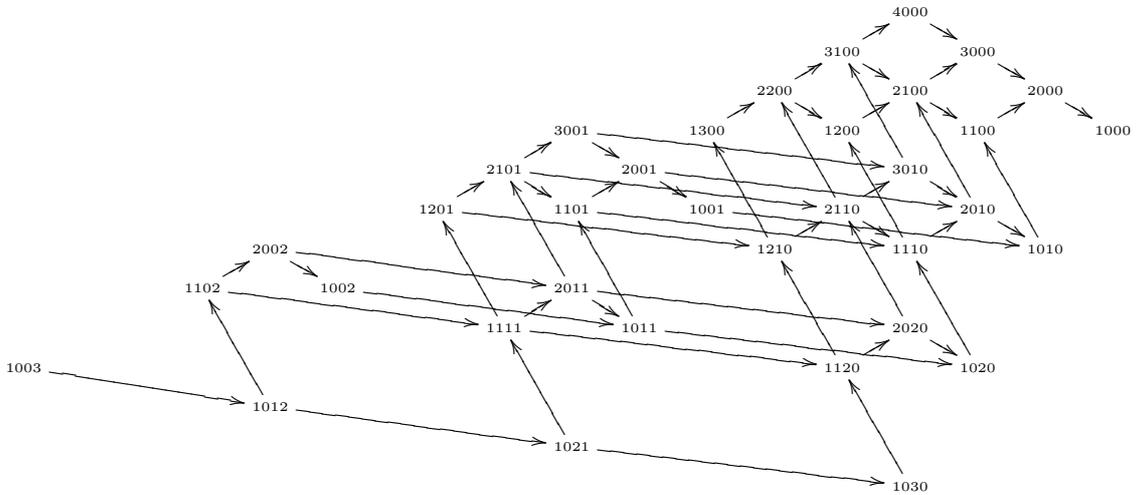
\[\xymatrix@C=0.2cm@R0.2cm{
&&&&&&&&&&&&&{\scriptstyle 4000}\ar[dr]&&&&\\
&&&&&&&&&&&&{\scriptstyle 3100}\ar[ur]\ar[dr]&&{\scriptstyle 3000}\ar[dr]&&&\\
&&&&&&&&&&&{\scriptstyle 2200}\ar[ur]\ar[dr]&&{\scriptstyle 2100}\ar[ur]\ar[dr]&&{\scriptstyle 2000}\ar[dr]&&\\
&&&&&&&&{\scriptstyle 3001}\ar[dr]\ar[drrrrr]&&
{\scriptstyle 1300}\ar[ur]&&{\scriptstyle 1200}\ar[ur]&&{\scriptstyle 1100}\ar[ur]&&{\scriptstyle 1000}\\
&&&&&&&{\scriptstyle 2101}\ar[ur]\ar[dr]\ar[drrrrr]&&{\scriptstyle 2001}\ar[dr]\ar[drrrrr]&&
&&{\scriptstyle 3010}\ar[dr]\ar[uuul]\\
&&&&&&{\scriptstyle 1201}\ar[ur]\ar[drrrrr]&&{\scriptstyle 1101}\ar[ur]\ar[drrrrr]&&{\scriptstyle 1001}\ar[drrrrr]
&&{\scriptstyle 2110}\ar[ur]\ar[dr]\ar[uuul]&&{\scriptstyle 2010}\ar[dr]\ar[uuul]&&\\
&&&{\scriptstyle 2002}\ar[dr]\ar[drrrrr]&&&&&&&
&{\scriptstyle 1210}\ar[ur]\ar[uuul]&&{\scriptstyle 1110}\ar[ur]\ar[uuul]&&{\scriptstyle 1010}\ar[uuul]&&\\
&&{\scriptstyle 1102}\ar[ur]\ar[drrrrr]&&{\scriptstyle 1002}\ar[drrrrr]&&
&&{\scriptstyle 2011}\ar[dr]\ar[uuul]\ar[drrrrr]\\
&&&&&&&{\scriptstyle 1111}\ar[ur]\ar[uuul]\ar[drrrrr]&&{\scriptstyle 1011}\ar[uuul]\ar[drrrrr]&
&&&{\scriptstyle 2020}\ar[dr]\ar[uuul]\\
{\scriptstyle 1003}\ar[drrr]&\ \ \ \ \ \ \ \ \ \ \ \ &&&&&&&&&
&&{\scriptstyle 1120}\ar[ur]\ar[uuul]&&{\scriptstyle 1020}\ar[uuul]\\
&&&{\scriptstyle 1012}\ar[uuul]\ar[drrrrr]\\
&&&&&&&&{\scriptstyle 1021}\ar[uuul]\ar[drrrrr]\\
&&&&&&&&&&&&&{\scriptstyle 1030}\ar[uuul]
}\]}

On the other hand, the Auslander-Reiten quiver of $\UU^{(1)}$ is the following.
{\tiny
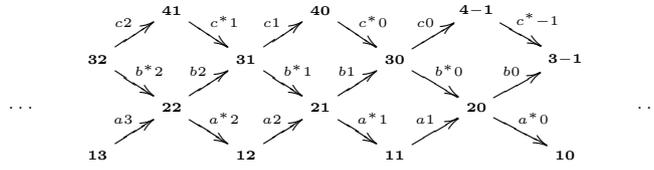
\[\xymatrix@C=0.5cm@R0.3cm{
&&{\scriptstyle\bf 41}\ar[dr]^{c^*1}&&{\scriptstyle\bf 40}\ar[dr]^{c^*0}&&{\scriptstyle\bf 4-1}\ar[dr]^{c^*-1}\\
&{\scriptstyle\bf 32}\ar[ur]^{c2}\ar[dr]^{b^*2}&&{\scriptstyle\bf 31}\ar[ur]^{c1}\ar[dr]^{b^*1}&&{\scriptstyle\bf 30}\ar[dr]^{b^*0}\ar[ur]^{c0}&&
{\scriptstyle\bf 3-1}\\
\cdots&&{\scriptstyle\bf 22}\ar[ur]^{b2}\ar[dr]^{a^*2}&&{\scriptstyle\bf 21}\ar[ur]^{b1}\ar[dr]^{a^*1}&&{\scriptstyle\bf 20}\ar[dr]^{a^*0}\ar[ur]^{b0}&&\cdots\\
&{\scriptstyle\bf 13}\ar[ur]^{a3}&&{\scriptstyle\bf 12}\ar[ur]^{a2}&&{\scriptstyle\bf 11}\ar[ur]^{a1}&&{\scriptstyle\bf 10}}\]}
The Auslander-Reiten quiver of $\UU^{(2)}$ is the following.
{\tiny
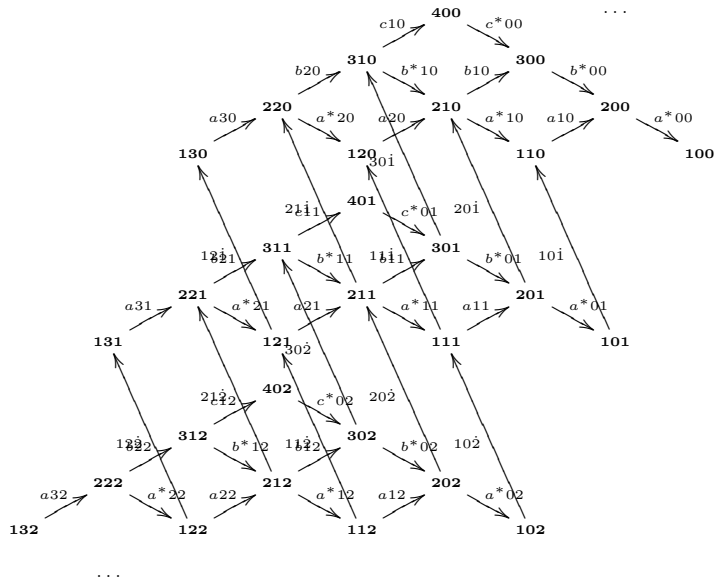
\[\xymatrix@C=0.5cm@R0.3cm{
&&&&&{\scriptstyle\bf 400}\ar[dr]^{c^*00}&&\cdots\\
&&&&{\scriptstyle\bf 310}\ar[ur]^{c10}\ar[dr]^{b^*10}&&{\scriptstyle\bf 300}\ar[dr]^{b^*00}\\
&&&{\scriptstyle\bf 220}\ar[ur]^{b20}\ar[dr]^{a^*20}&&{\scriptstyle\bf 210}\ar[ur]^{b10}\ar[dr]^{a^*10}&&{\scriptstyle\bf 200}\ar[dr]^{a^*00}\\
&&{\scriptstyle\bf 130}\ar[ur]^{a30}&&{\scriptstyle\bf 120}\ar[ur]^{a20}&&{\scriptstyle\bf 110}\ar[ur]^{a10}&&{\scriptstyle\bf 100}\\
 &&&&{\scriptstyle\bf 401}\ar[dr]^{c^*01}&&\\
 &&&{\scriptstyle\bf 311}\ar[ur]^{c11}\ar[dr]^{b^*11}&&{\scriptstyle\bf 301}\ar[dr]^{b^*01}\ar[uuuul]^{30\dot{1}}\\
 &&{\scriptstyle\bf 221}\ar[ur]^{b21}\ar[dr]^{a^*21}&&{\scriptstyle\bf 211}\ar[ur]^{b11}\ar[dr]^{a^*11}\ar[uuuul]^{21\dot{1}}&&{\scriptstyle\bf 201}\ar[dr]^{a^*01}\ar[uuuul]^{20\dot{1}}\\
 &{\scriptstyle\bf 131}\ar[ur]^{a31}&&{\scriptstyle\bf 121}\ar[ur]^{a21}\ar[uuuul]^{12\dot{1}}&&{\scriptstyle\bf 111}\ar[ur]^{a11}\ar[uuuul]^{11\dot{1}}&&{\scriptstyle\bf 101}\ar[uuuul]^{10\dot{1}}\\
  &&&{\scriptstyle\bf 402}\ar[dr]^{c^*02}&&\\
  &&{\scriptstyle\bf 312}\ar[ur]^{c12}\ar[dr]^{b^*12}&&{\scriptstyle\bf 302}\ar[dr]^{b^*02}\ar[uuuul]^{30\dot{2}}\\
  &{\scriptstyle\bf 222}\ar[ur]^{b22}\ar[dr]^{a^*22}&&{\scriptstyle\bf 212}\ar[ur]^{b12}\ar[dr]^{a^*12}\ar[uuuul]^{21\dot{2}}&&{\scriptstyle\bf 202}\ar[dr]^{a^*02}\ar[uuuul]^{20\dot{2}}\\
  {\scriptstyle\bf 132}\ar[ur]^{a32}&&{\scriptstyle\bf 122}\ar[ur]^{a22}\ar[uuuul]^{12\dot{2}}&&{\scriptstyle\bf 112}\ar[ur]^{a12}\ar[uuuul]^{11\dot{2}}&&{\scriptstyle\bf 102}\ar[uuuul]^{10\dot{2}}\\
  &\cdots
 }\]}
The Auslander-Reiten quiver of $\UU^{(3)}$ is the following.
{\tiny
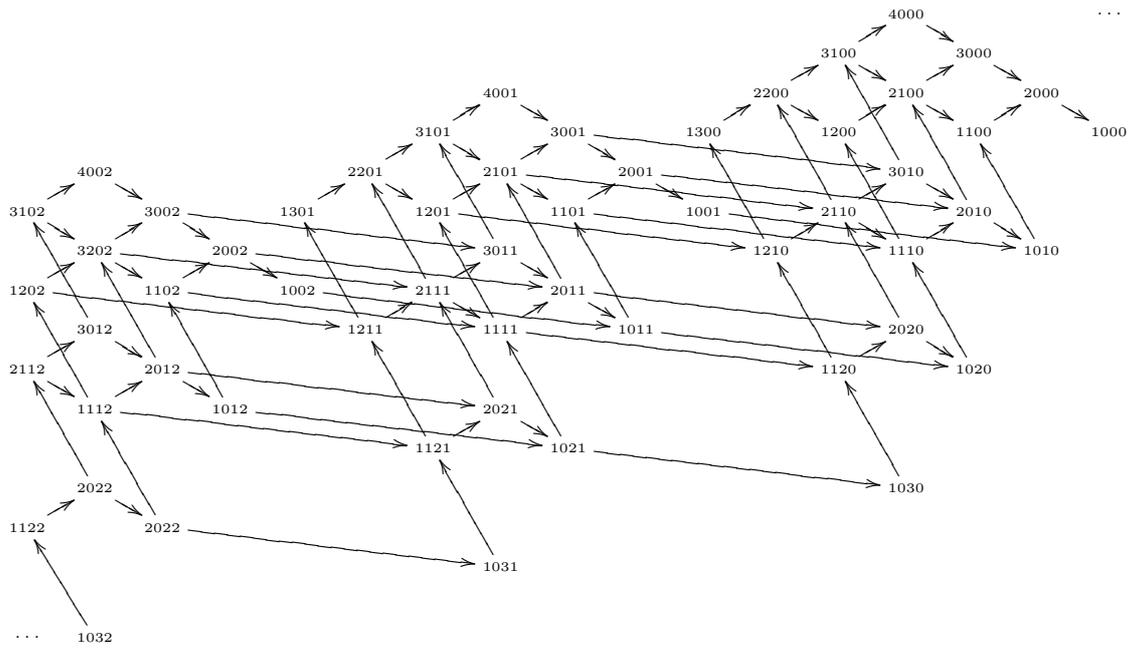
\[\xymatrix@C=0.2cm@R0.2cm{
&&&&&&&&&&&&&{\scriptstyle 4000}\ar[dr]&&&\cdots&\\
&&&&&&&&&&&&{\scriptstyle 3100}\ar[ur]\ar[dr]&&{\scriptstyle 3000}\ar[dr]&&&\\
&&&&&&&{\scriptstyle 4001}\ar[dr]&&&&{\scriptstyle 2200}\ar[ur]\ar[dr]&&{\scriptstyle 2100}\ar[ur]\ar[dr]&&{\scriptstyle 2000}\ar[dr]&&\\
&&&&&&{\scriptstyle 3101}\ar[ur]\ar[dr]&&{\scriptstyle 3001}\ar[dr]\ar[drrrrr]&&
{\scriptstyle 1300}\ar[ur]&&{\scriptstyle 1200}\ar[ur]&&{\scriptstyle 1100}\ar[ur]&&{\scriptstyle 1000}\\
&{\scriptstyle 4002}\ar[dr]&&&&{\scriptstyle 2201}\ar[ur]\ar[dr]&&{\scriptstyle 2101}\ar[ur]\ar[dr]\ar[drrrrr]&&{\scriptstyle 2001}\ar[dr]\ar[drrrrr]&&
&&{\scriptstyle 3010}\ar[dr]\ar[uuul]\\
{\scriptstyle 3102}\ar[ur]\ar[dr]&&{\scriptstyle 3002}\ar[dr]\ar[drrrrr]&&{\scriptstyle 1301}\ar[ur]&&
{\scriptstyle 1201}\ar[ur]\ar[drrrrr]&&{\scriptstyle 1101}\ar[ur]\ar[drrrrr]&&{\scriptstyle 1001}\ar[drrrrr]
&&{\scriptstyle 2110}\ar[ur]\ar[dr]\ar[uuul]&&{\scriptstyle 2010}\ar[dr]\ar[uuul]&&\\
&{\scriptstyle 3202}\ar[ur]\ar[dr]\ar[drrrrr]&&{\scriptstyle 2002}\ar[dr]\ar[drrrrr]&&&&{\scriptstyle 3011}\ar[dr]\ar[uuul]&&&
&{\scriptstyle 1210}\ar[ur]\ar[uuul]&&{\scriptstyle 1110}\ar[ur]\ar[uuul]&&{\scriptstyle 1010}\ar[uuul]&&\\
{\scriptstyle 1202}\ar[ur]\ar[drrrrr]&&{\scriptstyle 1102}\ar[ur]\ar[drrrrr]&&{\scriptstyle 1002}\ar[drrrrr]&&
{\scriptstyle 2111}\ar[ur]\ar[dr]\ar[uuul]&&{\scriptstyle 2011}\ar[dr]\ar[uuul]\ar[drrrrr]\\
&{\scriptstyle 3012}\ar[dr]\ar[uuul]&&&&{\scriptstyle 1211}\ar[ur]\ar[uuul]&&{\scriptstyle 1111}\ar[ur]\ar[uuul]\ar[drrrrr]
&&{\scriptstyle 1011}\ar[uuul]\ar[drrrrr]&
&&&{\scriptstyle 2020}\ar[dr]\ar[uuul]\\
{\scriptstyle 2112}\ar[ur]\ar[dr]&&{\scriptstyle 2012}\ar[dr]\ar[uuul]\ar[drrrrr]&&&&&&&&
&&{\scriptstyle 1120}\ar[ur]\ar[uuul]&&{\scriptstyle 1020}\ar[uuul]\\
&{\scriptstyle 1112}\ar[ur]\ar[uuul]\ar[drrrrr]&&{\scriptstyle 1012}\ar[uuul]\ar[drrrrr]&&&&{\scriptstyle 2021}\ar[dr]\ar[uuul]\\
&&&&&&{\scriptstyle 1121}\ar[ur]\ar[uuul]&&{\scriptstyle 1021}\ar[uuul]\ar[drrrrr]\\
&{\scriptstyle 2022}\ar[dr]\ar[uuul]&&&&&&&&&&&&{\scriptstyle 1030}\ar[uuul]\\
{\scriptstyle 1122}\ar[ur]&&{\scriptstyle 2022}\ar[uuul]\ar[drrrrr]\\
&&&&&&&{\scriptstyle 1031}\ar[uuul]\\
\\
\cdots&{\scriptstyle 1032}\ar[uuul]\\
}\]}

(b) Let $Q$ be the quiver $\def\arraystretch{.5}\begin{array}{c}
{\scriptstyle\bf 2}\stackrel{a}{\longrightarrow}{\scriptstyle\bf 1}\stackrel{c}{\longleftarrow}{\scriptstyle\bf 4}\\
\ \ \uparrow^b\\
{\scriptstyle\bf 3}
\end{array}$ of type $D_4$.
The Auslander-Reiten quiver of $\MM^{(1)}$ is the following.
{\tiny
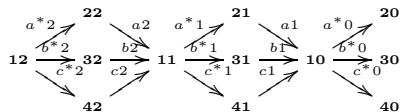
\[\xymatrix@C=0.5cm@R0.3cm{
&{\scriptstyle\bf 22}\ar[dr]^{a2}&&{\scriptstyle\bf 21}\ar[dr]^{a1}&&{\scriptstyle\bf 20}\\
{\scriptstyle\bf 12}\ar[ur]^{a^*2}\ar[dr]^{c^*2}\ar[r]^{b^*2}&{\scriptstyle\bf 32}\ar[r]^{b2}&{\scriptstyle\bf 11}\ar[ur]^{a^*1}\ar[dr]^{c^*1}\ar[r]^{b^*1}&{\scriptstyle\bf 31}\ar[r]^{b1}&{\scriptstyle\bf 10}\ar[dr]^{c^*0}\ar[ur]^{a^*0}\ar[r]^{b^*0}&{\scriptstyle\bf 30}\\
&{\scriptstyle\bf 42}\ar[ur]^{c2}&&{\scriptstyle\bf 41}\ar[ur]^{c1}&&{\scriptstyle\bf 40}}\]}
The Auslander-Reiten quiver of $\MM^{(2)}$ is the following.
{\tiny
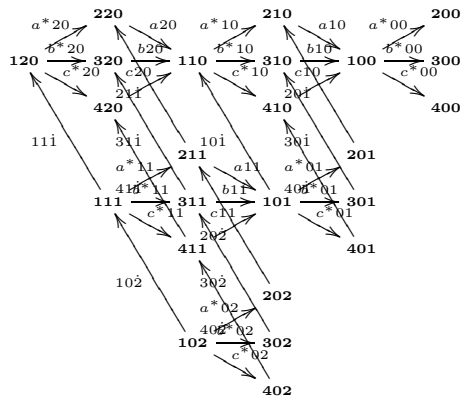
\[\xymatrix@C=0.5cm@R0.3cm{
&{\scriptstyle\bf 220}\ar[dr]^{a20}&&{\scriptstyle\bf 210}\ar[dr]^{a10}&&{\scriptstyle\bf 200}\\
{\scriptstyle\bf 120}\ar[ur]^{a^*20}\ar[dr]^{c^*20}\ar[r]^{b^*20}&{\scriptstyle\bf 320}\ar[r]^{b20}&{\scriptstyle\bf 110}\ar[ur]^{a^*10}\ar[dr]^{c^*10}\ar[r]^{b^*10}&{\scriptstyle\bf 310}\ar[r]^{b10}&{\scriptstyle\bf 100}\ar[dr]^{c^*00}\ar[ur]^{a^*00}\ar[r]^{b^*00}&{\scriptstyle\bf 300}\\
&{\scriptstyle\bf 420}\ar[ur]^{c20}&&{\scriptstyle\bf 410}\ar[ur]^{c10}&&{\scriptstyle\bf 400}\\
&&{\scriptstyle\bf 211}\ar[dr]^{a11}\ar[uuul]^{21\dot{1}}&&{\scriptstyle\bf 201}\ar[uuul]^{20\dot{1}}\\
&{\scriptstyle\bf 111}\ar[ur]^{a^*11}\ar[dr]^{c^*11}\ar[r]^{b^*11}\ar[uuul]^{11\dot{1}}&{\scriptstyle\bf 311}\ar[r]^{b11}\ar[uuul]^{31\dot{1}}&{\scriptstyle\bf 101}\ar[ur]^{a^*01}\ar[dr]^{c^*01}\ar[r]^{b^*01}\ar[uuul]^{10\dot{1}}&{\scriptstyle\bf 301}\ar[uuul]^{30\dot{1}}\\
&&{\scriptstyle\bf 411}\ar[ur]^{c11}\ar[uuul]^{41\dot{1}}&&{\scriptstyle\bf 401}\ar[uuul]^{40\dot{1}}\\
&&&{\scriptstyle\bf 202}\ar[uuul]^{20\dot{2}}\\
&&{\scriptstyle\bf 102}\ar[ur]^{a^*02}\ar[dr]^{c^*02}\ar[r]^{b^*02}\ar[uuul]^{10\dot{2}}&{\scriptstyle\bf 302}\ar[uuul]^{30\dot{2}}\\
&&&{\scriptstyle\bf 402}\ar[uuul]^{40\dot{2}}
}\]}
The Auslander-Reiten quiver of $\MM^{(3)}$ is the following.
{\tiny\[\xymatrix@C=0.2cm@R0.2cm{
&&&&&&&&&&&&
&{\scriptstyle 2200}\ar[dr]&&{\scriptstyle 2100}\ar[dr]&&{\scriptstyle 2000}\\
&&&&&&&&&&&&
{\scriptstyle 1200}\ar[ur]\ar[dr]\ar[r]&{\scriptstyle 3200}\ar[r]&{\scriptstyle 1100}\ar[ur]\ar[dr]\ar[r]&{\scriptstyle 3100}\ar[r]&{\scriptstyle 1000}\ar[dr]\ar[ur]\ar[r]&{\scriptstyle 3000}\\
&&&&&&
&{\scriptstyle 2101}\ar[dr]\ar[drrrrrrr]&&{\scriptstyle 2001}\ar[drrrrrrr]&&&
&{\scriptstyle 4200}\ar[ur]&&{\scriptstyle 4100}\ar[ur]&&{\scriptstyle 4000}\\
&&&&&&
{\scriptstyle 1101}\ar[ur]\ar[dr]\ar[r]\ar[drrrrrrr]&{\scriptstyle 3101}\ar[r]\ar[drrrrrrr]&{\scriptstyle 1001}\ar[ur]\ar[dr]\ar[r]\ar[drrrrrrr]&{\scriptstyle 3001}\ar[drrrrrrr]&&&
&&{\scriptstyle 2110}\ar[dr]\ar[uuul]&&{\scriptstyle 2010}\ar[uuul]\\
&{\scriptstyle 2002}\ar[drrrrrrr]&&&&&
&{\scriptstyle 4101}\ar[ur]\ar[drrrrrrr]&&{\scriptstyle 4001}\ar[drrrrrrr]&&&
&{\scriptstyle 1110}\ar[ur]\ar[dr]\ar[r]\ar[uuul]&{\scriptstyle 3110}\ar[r]\ar[uuul]&{\scriptstyle 1010}\ar[ur]\ar[dr]\ar[r]\ar[uuul]&{\scriptstyle 3010}\ar[uuul]\\
{\scriptstyle 1002}\ar[ur]\ar[dr]\ar[r]\ar[drrrrrrr]&{\scriptstyle 3002}\ar[drrrrrrr]&&&&&
&&{\scriptstyle 2011}\ar[uuul]\ar[drrrrrrr]&&&&
&&{\scriptstyle 4110}\ar[ur]\ar[uuul]&&{\scriptstyle 4010}\ar[uuul]\\
&{\scriptstyle 4002}\ar[drrrrrrr]&&&&&
&{\scriptstyle 1011}\ar[ur]\ar[dr]\ar[r]\ar[uuul]\ar[drrrrrrr]&{\scriptstyle 3011}\ar[uuul]\ar[drrrrrrr]&&&&
&&&{\scriptstyle 2020}\ar[uuul]\\
&&&&&&
&&{\scriptstyle 4011}\ar[uuul]\ar[drrrrrrr]&&&&
&&{\scriptstyle 1020}\ar[ur]\ar[dr]\ar[r]\ar[uuul]&{\scriptstyle 3020}\ar[uuul]\\
&&&&&&\ \ \ \ \ \ \ \ \ \ &&&&&&
&&&{\scriptstyle 4020}\ar[uuul]
}\]}
\end{example}

\end{document}